%% file: main-weight-control.tex
\documentclass[10pt, one column, draft cls]{IEEEtran}
\IEEEoverridecommandlockouts 

\input{config}

\usepackage{tikz}
\usetikzlibrary{positioning}
\usetikzlibrary{arrows}

\title{Robustness of DC Power Networks \\ under Weight Control}
\author{Qin Ba\thanks{The authors are with the Sonny Astani Department of Civil and Environmental Engineering at the University of Southern California, Los Angeles, CA. \texttt{\{qba,ksavla\}@usc.edu}. They were supported in part by NSF CAREER ECCS Project No. 1454729. Partial results from this paper appeared as \cite{Ba.Savla:ACC16}. This paper contains proofs and other technical details missing in \cite{Ba.Savla:ACC16}, as well as several new results.} \qquad Ketan Savla}

\date{\today}

\begin{document}
\maketitle

\input{abstract}

\input{introduction}

\input{problem-formulation}

\input{mor-and-mincut-qin}

\input{multiplicative-disturbance}

\input{multi-level-formulation}

\input{tree-reducible-net}

\input{decentralized-control-1}

\input{decentralized-control-2}

\input{simulations-qin}

\input{conclusions}
\bibliographystyle{ieeetr}
\bibliography{ksmain,savla}

\input{appendix}

\end{document}

%% file: config.tex
\usepackage{amsmath,amssymb,bbm,color,psfrag,amsfonts,amsthm}
\usepackage{dsfont}
\usepackage{graphicx}                       
\usepackage[pdftex,pdfpagelabels,bookmarks,hyperindex,hyperfigures]{hyperref}
\usepackage{wrapfig}
\usepackage{pdfpages}
\usepackage[ruled,lined]{algorithm2e} 
   \SetKwInOut{Input}{input}\SetKwInOut{Output}{output}
\usepackage{accents}

\usepackage{paralist}   

\usepackage{enumitem}
\setlist{parsep=0pt,listparindent=\parindent}

\usepackage{booktabs}  
\usepackage{multirow}
\usepackage[normalem]{ulem}

\usepackage{caption}
\usepackage{subcaption}

\usepackage{cite}

\newcounter{problem}
\newenvironment{problem}[1][]
{\refstepcounter{problem} \medskip 
\noindent
\begin{tabular}{@{}p{\textwidth}} 
\\ \toprule 
 \textbf{Problem~\theproblem:}  #1 
   \\  \midrule
   }
   {\vspace{-0.6cm} \\ \bottomrule \end{tabular}  \medskip}

\newtheorem{theorem}{Theorem}

\newtheorem{lemma}{Lemma}
\newtheorem{proposition}{Proposition}

\newtheorem{definition}{Definition}
\newtheorem{assumption}{Assumption}

\newtheorem{remark}{Remark}
\newtheorem{example}{Example}

\usepackage{color}

\newcommand{\ubar}[1]{\underaccent{\bar}{#1}}

\newcommand{\ones}{\mathbf 1}
\newcommand{\real}{\mathbb{R}}

\newcommand{\range}{{\mathcal R}}

\newcommand{\diag}{\mathop{\bf diag}}

\newcommand{\inv}{{\mathrm{\textbf{inv}}}}

\newcommand{\argmin}{\mathop{\rm argmin}}
\newcommand{\argmax}{\mathop{\rm argmax}}

\newcommand{\sign}{\mathop{\bf sign}}
\newcommand{\setdef}[2]{\left\{#1 \; | \; #2\right\}}
\newcommand{\spn}[1]{ \mathrm{span}\{#1\} }    
 
\newcommand{\gmax}{g_{\mathrm{max}}}
\newcommand{\weq}{w_{\mathrm{eq}}}
\newcommand{\weqlow}{w_{\mathrm{eq}}^l}
\newcommand{\wequp}{w_{\mathrm{eq}}^u}
\newcommand{\weqminus}{w_{\mathrm{eq}}^-}
\newcommand{\weqplus}{w_{\mathrm{eq}}^+}
\newcommand{\marginparold}{\alpha^*}

\newcommand{\weqopt}{\weq^{\mathrm{opt}}}
\newcommand{\wopt}{w^{\mathrm{opt}}}
\newcommand{\weqm}{w^*}

\newcommand{\trans}[1]{#1^{\mathrm{T}}}
\newcommand{\RR}{\mathbb{R}}
\newcommand{\NN}{\mathbb{N}}
\newcommand{\mc}{\mathcal}

\newcommand{\clower}{c^l}
\newcommand{\cupper}{c^u}
\newcommand{\wlower}{w^l}
\newcommand{\wupper}{w^u}
\newcommand{\zerobf}{\mathbf{0}}
\newcommand{\onebf}{\mathbf{1}}
\newcommand{\until}[1]{\{1,\dots,#1\}}
\newcommand{\param}{\nu}
\newcommand{\nongenset}{\mathbf{\Delta}_{NG}}
\newcommand{\map}[3]{#1: #2 \rightarrow #3}
\newcommand{\ds}{\displaystyle}

\newcommand{\eg}{{\it e.g.}}
\newcommand{\ie}{{\it i.e.}}

\newcommand{\wrt}{{\it w.r.t. }}

\usepackage{etoolbox}
\newtoggle{long}
\toggletrue{long}


%% file: abstract.tex
\begin{abstract}
We study, possibly distributed, robust weight control policies for DC power networks that change link susceptances or \emph{weights} within specified operational range in response to balanced disturbances to the supply-demand vector. The margin of robustness for a given control policy is defined as the radius of the largest $\ell_1$ ball in the space of balanced disturbances under which the link flows can be asymptotically contained within their specified limits. For centralized control policies, there is no post-disturbance dynamics, and hence the control design as well as margin of robustness are obtained from solution to an optimization problem, referred to as the \emph{weight control problem}, which is non-convex in general. We establish relationship between feasible sets for DC power flow and associated network flow, which is used to establish an upper bound on the margin of robustness in terms of the min cut capacity. This bound is proven to be tight if the network is tree-like, or if the lower bound of the operation range of weight control is zero. An explicit expression for the flow-weight Jacobian is derived and is used to devise a projected sub-gradient algorithm to solve the relaxed weight control problem. An exact multi-level programming approach to solve the weight control problem for reducible networks, based on recursive application of equivalent bilevel formulation for relevant class of non-convex network optimization problems, is also proposed. The lower level problem in each recursion corresponds to replacing a sub-network by a (virtual) link with equivalent weight and capacities. The equivalent capacity function for tree-reducible networks is shown to possess a strong quasi-concavity property, facilitating easy solution to the multilevel programming formulation of the weight control problem. Robustness analysis for natural decentralized control policies that decrease weights on overloaded links, and increase weights on underloaded links with increasing flows is provided for parallel networks. 
Illustrative simulation results for a benchmark IEEE network are also included. 
\end{abstract}

%% file: introduction.tex

\section{Introduction}
Robustness to man-made and natural disturbances is becoming an important consideration in the design and operation of critical infrastructure networks, such as the power grid, in part due to potential catastrophic consequences caused by ensuing cascading failures which can also affect other dependent systems.  
Disturbances to power networks are usually in the form of line failures and fluctuations in the supply-demand profile, \eg, due to renewables. From a control design perspective, the objective is to ensure that variations in power flow quantities caused by such external disturbances do not violate physical constraints such as exceeding line thermal limits or voltage collapse. The most well-studied control strategies range from load/frequency/voltage control to changing, usually shedding, of supply and demand, possibly combined with intentional islanding of smaller sections of a power network.

In this paper, we consider a DC model for (transmission) power networks, which is subject to balanced disturbances to the supply-demand vector, \ie, disturbance vectors whose entries add up to zero. Such disturbances can result, \eg, from the tripping of an active line. Alternately, one could attribute such disturbances to the residual of actual disturbances which can not be handled by other control means. We consider the relatively less studied control strategy that uses information about link flows and susceptances, or \emph{weights}, and disturbance, to change line weights in order to ensure that the line flows remain within prescribed limits. 
The control policies can be constrained in the available information, \eg, in the decentralized case, the controller on a given link has access only to information about the weight and flow on itself. The only dynamics in this paper are to be attributed to distributed control settings, where each controller has incomplete information, thereby leading to iterative control actions. 
The margin of robustness of a given control policy is defined as the radius of the largest $\ell_1$ ball in the space of balanced disturbances under which the link flows can be asymptotically contained within their specified limits. This notion of margin of robustness is related to \emph{system loadability}, \eg, see \cite{Klump.Overbye:97}, which quantifies deviations in supply-demand vector in terms of percentage of the nominal, \ie, pre-disturbance, value under which the system remains feasible. The objective of this paper is to compute margin of robustness of weight control strategies, and to design control policies which are provably maximally robust.

The weight control strategy in this paper is motivated by FACTS devices, which allow online control of line properties in power networks. These devices are typically expensive, with the cost depending on the range of operation, \eg, see \cite{Saravanan.Slochanal.ea:07}. This has motivated research on optimal placement of a given number of FACTS devices, \eg, see \cite{Gerbex.Cherkaoui.ea:01,Saravanan.Slochanal.ea:07}. In our framework, such economic aspects can be incorporated implicitly by constraining the control to change line weights within specified limits. The formal analysis in this paper is to be contrasted with previous work on coordinated control of FACTS devices, \eg, see \cite{hug2009decentralized}, to improve system loadability, or usage of FACTS devices to improve efficiency~\cite{Singh.David:01} and security~\cite{Gerbex.Cherkaoui.ea:03}. Use of FACTS devices in the context of the optimal power flow problem has also been explored, \eg, see \cite{Shao.Vittal:06}. However, none of these or related works, to the best of our knowledge,  provide formal performance guarantees.

%


The weight control problem in this paper is related to the so called impedance interdiction problem which has been studied for DC power flow models in \cite{bienstock2010nk}. The objective in such interdiction problems is to analyze the vulnerability of power networks against an adversary who changes the susceptances of power lines subject to budget constraints. The weight control problem in this paper can also be considered to be relaxation of the transmission switching and network topology optimization problem for power networks, \eg, see \cite{Hedman.Oren.ea:11}, where the objective is to choose a subset among all possible links, subject to connectivity constraints, that allow to transfer power between given load and supply nodes, subject to thermal capacity constraints. The non-triviality of this problem can be attributed to non-monotonicity of power flow with respect to changes in demand-supply profile or changes in graph topology of the network (see Example~\ref{eg:load-nonmonotonicity} for simple illustrations). In the topology control problem, the control actions associated with every link only take binary values corresponding to on/off status of the link, or equivalently corresponding to the weight of that link being equal to either zero or its nominal value. On the other hand, in our weight control problem, we allow a continuum of control actions that includes these two values.    
  
In the centralized case, when our model has no post-disturbance dynamics, computation of margin of robustness and design of robust weight control strategies can be posed as an optimization problem, referred to as the \emph{weight control problem}, which is non-convex in general. The solution to this problem also gives an upper bound on the margin of robustness for any, including decentralized, weight control policies. We establish relationship between feasible sets for DC power flow and associated network flow, which is used to establish an upper bound on the margin of robustness in terms of the min cut capacity. This bound is proven to be tight if the network is tree-like, or if the lower bound of the operation range of weight control is zero, \ie, allowing for disconnecting links.

We propose a projected sub-gradient algorithm to solve the weight control problem for multiplicative disturbances. A key component of this algorithm is the flow-weight Jacobian, \ie, a matrix whose elements give the sensitivities of link flows with respect to link weights. We provide an explicit expression for this Jacobian, and make connections with existing results that characterize change in link flows due to removal of links for DC power flow models, \eg, see \cite{Wood.Wollenberg:12,Lai.Low:Allerton13,Bienstock:16}.  

We also provide a multilevel programming approach to solve the weight control problem for reducible networks. Specifically, we first identify a class of network
 optimization problems which can be equivalently converted into a bilevel formulation over two sub-networks which have only two nodes in common, and one of which does not contain any supply-demand nodes, with the possible exception of the common nodes. The upper and lower level optimization problems, although both non-convex, are shown to have similar structure and lead to significant computational savings when solving the weight control problem using an exhaustive search method. Interestingly, the lower level problem can be interpreted as defining \emph{equivalent capacities} of the underlying sub-network for a given \emph{equivalent weight} of the same sub-network. While the latter is reminiscent of the notion of equivalent resistance from circuit theory, the former appears to be novel. While the equivalent capacities are expectedly dependent on the weight of the underlying network, a remarkable aspect of this definition is that this dependence can be expressed entirely in terms of the equivalent weight of the underlying sub-network. 

The proposed bilevel formulation can be applied recursively in a nested fashion to yield a multilevel framework, where each iteration of bilevel formulation results in additional computational savings. Further computational savings are possible when the network is tree-reducible, \ie, when it can be reduced to a tree by sequentially replacing series and parallel sub-networks with equivalent, in terms of weight and capacity, links. This is because series and parallel (meta-)networks are proven to admit an invariance of a certain strong quasi-concavity property from the equivalent capacity function of their constituent links to the equivalent capacity function of the network itself; and the strong quasi-concavity property of the equivalent capacity function is shown to facilitate its explicit computation.

  

We then study robustness properties of a couple of \emph{natural} decentralized control policies. Under the first controller, weight on an overloaded link is decreased if its weight is greater than the lower limit of the operation range. Such a controller is proven to be maximally robust for parallel networks if the initial weight is no less than a solution to the weight control problem. We then consider a second controller which augments the first controller by additionally increasing weight \emph{altrustically} on an underloaded link if the flow on it is increasing and if its weight is less than the upper limit of the operation range. Such a controller is proven to be maximally robust for parallel networks with two links. 



In summary, the paper makes several novel contributions. First, 
a novel robust control problem is formulated where the control strategy consists of changing link weights in response to possibly decentralized information about link flows, weights and disturbance. We then establish connection between the margin of robustness and cut capacities of associated flow network. Second, we provide an explicit expression for the flow-weight Jacobian. While being of independent interest, its utility is demonstrated in a projected gradient descent algorithm to solve the problem for multiplicative disturbances. Third, we identify a class of non-convex network optimization problems which can be equivalently formulated as bilevel problems, and identify conditions under which the robust control problem belongs to this class. The notion of equivalent capacity is introduced, which possibly of independent interest, is shown to correspond to the lower level problem, and is shown to possess a strong quasi-concavity property for series and parallel networks. The input-output invariance of this property for series and parallel networks allows efficient computation of equivalent capacity, and hence solution to the robust control problem, for tree-reducible networks. Fourth, we provide robustness guarantees for a couple of natural decentralized control policies for parallel networks. Illustrative simulations on an IEEE benchmark network are also provided.

The rest of the paper is organized as follows. Section~\ref{sec:problem-formulation} formally states the robust weight control problem, formulates an optimization problem for robust control design and margin of robustness computation for centralized control policies, and establishes connection with classical network flow problems. Section~\ref{sec:multiplicative} provides an explicit expression for the flow-weight Jacobian and a gradient descent algorithm that utilizes it. 
Section~\ref{sec:bi-level} describes a multi-level programming approach for solving the robust control problem for centralized control policies under multiplicative disturbances and introduces the definition of equivalent capacity function for a network. Section~\ref{sec:simple-prob-class} establishes input-output invariance of a strong quasi-concavity property for series and parallel network and utilizes it to efficiently solve the multilevel formulation of the robust control problem for tree reducible networks.  Section~\ref{sec:decentralized} provides robustness guarantees for a couple of natural decentralized weight control policies. Illustrative simulation results are provided in Section~\ref{sec:simulations}. Concluding remarks and comments on directions for future research are provided in Section~\ref{sec:conclusions}. A few technical lemmas are collected in the Appendix.

We conclude this section by defining a few key notations to be used throughout the paper. $\real$, $\real_{\geq 0}$, $\real_{>0}$, $\real_{\leq 0}$ and $\real_{< 0}$ will stand for real, non-negative real, strictly positive real, non-positive real, and strictly negative real, respectively, and $\NN$ denotes the set of natural numbers. $\zerobf$ and $\onebf$ will denote the vector of all zeros and all ones, respectively, where the size of the vector will be clear from the context. Given two vectors $a, b \in \real^n$, $a \leq b$ (resp., $a<b$) would imply $a_i \leq b_i$ (resp., $a_i<b_i$) for all $i \in \until{n}$. Given a vector $a \in \real^n$, $\diag(a)$ denotes a diagonal matrix, whose diagonal entries correspond to elements of $a$, and $ a_{B} $ denotes the sub-vector of $ a $ corresponding to the subset $ B \subset \until{n} $.
We refer the reader to standard textbooks on graph theory, \eg, \cite{Diestel:03}, for a thorough overview of key concepts and definitions for graphs -- we recall a few important ones here for the sake of completeness. 
A directed multigraph is the pair $(\mc V,\mc E)$ of a finite set $\mc V$ of
nodes, and of a multiset $\mc E$ of links consisting of ordered pairs of nodes
(\ie, we allow for parallel links between a pair of nodes). We adopt the
convention that a directed multigraph does not contain a self-loop. A simple
directed graph is a directed graph $(\mc V,\mc E)$ having no multiple edges or
self-loops. A directed path in a digraph is a sequence of vertices in which
there is a (directed) edge pointing from each vertex in the sequence to its
successor in the sequence. A directed cycle is a directed path (with at least
one edge) whose first and last vertices are the same. If $i=(v_1,v_2)\in\mc E$ is
a link, where $v_1, v_2 \in \mc V$, we shall write $\sigma(i)=v_1$ and
$\tau(i)=v_2$  for its tail and head node, respectively. 
The sets of outgoing  and incoming links of a node $v\in\mc V$ will be denoted by $\mc E^+_v:=\{i \in\mc E:\,\sigma(i)=v\}$ and $\mc E^-_v:=\{i \in\mc E:\,\tau(i)=v\}$ respectively. The $\sign$ function is defined as $\sign(x)=+1$ if $x>0$, $=-1$ if $x<0$, and $=0$ if $x=0$. Given a map $\map{f}{X}{Y}$, $\range(f)$ will denote the range of $f$. With slight abuse of notation, we will also use $\mc R(A)$ to denote the range space of matrix $A$. A function $\map{f}{X}{\real}$ is quasiconvex if, for all $x_1, x_2 \in X$ and $\theta \in [0,1]$, we have $f(\theta x_1 + (1-\theta) x_2) \leq \max \{f(x_1), f(x_2)\}$. $f$ is quasiconcave if $-f$ is quasiconvex.

%% file: problem-formulation.tex

\section{Problem Formulation}
\label{sec:problem-formulation}
In this section, we formulate the problem of robust weight control, and provide preliminary results. We start by reviewing the DC power flow model. 

\subsection{DC Power Flow Model}
\label{sec:dc-model}
In the DC power flow model, it is
assumed that the transmission lines are lossless and the voltage magnitudes at nodes are
constant at 1.0 unit. Power flow on links is bidirectional; however, it is convenient to model the graph topology of the power network by a \emph{directed} multigraph $\mc G = (\mc V, \mc E)$, where the directions assigned to the links are arbitrary. 
We make the following assumption on the graph topology throughout the paper.

\begin{assumption}
\label{ass:connectivity}
$\mc G$ is weakly connected.  
\end{assumption}

Assumption~\ref{ass:connectivity} is without loss of generality because the results of this paper can be applied to every connected component of $\mc G$. The graph
topology is associated with a node-link incidence matrix $A \in \{-1,0,+1\}^{\mc
  V \times \mc E}$ that is consistent with the directions of links in $\mc E$,
\ie, for all $v \in \mc V$ and $i \in \mc E$, $A_{v i}$ is equal to $-1$
if $v=\tau(i)$, is equal to $+1$ if $v=\sigma(i)$, and is equal to zero otherwise. The links are associated with a flow
vector $f \in \RR^{\mc E}$, and the nodes are associated with phase angles $\phi
\in \RR^{\mc V}$ and a supply-demand vector $p \in \RR^{\mc V}$. 
The sign of components of $f$ are to be
interpreted as being consistent with the directional convention chosen for links
in $\mc E$. The direction of the flow on link \( i \) is in the same or opposite
direction of link \( i \) for \( f_{i} > 0 \) and \( f_{i} < 0 \),
respectively. 
A component of $p$ is positive (resp., negative) if the corresponding node is associated with a generator (resp., load). 
Throughout the paper, we shall assume that the supply-demand vector is \emph{balanced}, \ie, 
\begin{assumption}
\label{ass:balanced}
$\onebf^T p=0$.
\end{assumption}
Assumptions~\ref{ass:connectivity} and \ref{ass:balanced} will be standing assumptions throughout the paper. We also associate with the network a vector $w \in \real_{>0}^{\mc E}$
whose components give link susceptances.  
Hereafter, we shall refer to the susceptances as \emph{weights} on the
links. We let $W=\diag(w) \in \real_{>0}^{\mc E \times \mc E}$ denote the matrix representation of $w$.
We define the Laplacian of a network as follows. 
\begin{definition}
\label{def:laplacian}
  Given a network with directed multigraph \( \mathcal{G}= (\mathcal{V}, \mathcal{E}) \) having 
  node-link incidence matrix $A \in \{-1,0,+1\}^{\mc
  V \times \mc E}$ and weight
  matrix \( W \in \real_{>0}^{\mc E \times \mc E}\), its weighted Laplacian matrix is defined as
\begin{equation*}
L_{\mc G}(W):= A W \trans{A}. 
\end{equation*}
\end{definition}
For brevity in notation, we shall drop explicit dependence of $L$ on $W$ or $\mc G$ when clear from the context. 
\begin{remark}
 \leavevmode
 \begin{enumerate}[label = (\alph*)]
 \item While the Laplacian matrix is usually defined for networks with simple directed graphs,
   Definition \ref{def:laplacian} considers the general directed multigraph setting. Several distinct networks can have the same Laplacian -- we elaborate on this in Section~\ref{subsec:lap} in the Appendix. 
\item The Laplacian of a network does not depend on the specific choice of directionality for
  links. 
 \end{enumerate}
\end{remark} 

In a DC power network, the quantities defined above are related by Kirchhoff's law and Ohm's law as follows: 
\begin{equation}
 \label{model}
 \begin{split}
  A f &=  p \\ 
 f &= W \trans{A} \phi
  \end{split}  
\end{equation}

The next result provides an explicit expression for $f$ satisfying \eqref{model}. 
\begin{lemma}
\label{lemma:flow-solution}
Consider a power network with directed multigraph $\mc G=(\mc V, \mc E)$ having 
  node-link incidence matrix $A \in \{-1,0,+1\}^{\mc
  V \times \mc E}$, line weights $w \in \real_{>0}^{\mc E}$ and supply-demand vector $p \in \real^{\mc V}$. There exists a unique $f \in \RR^{\mc E}$ satisfying \eqref{model}, and is given by:
\begin{equation}
\label{eq:sol}
f= W A^T L^{\dagger} p =: f^{\mathcal{G}}(w, p)
\end{equation}
where 
$L^{\dagger}$ is the Moore-Penrose pseudo-inverse of $L$.
\end{lemma}
\begin{proof}
Substituting the second equation into the first in (\ref{model}), we get $L \phi = p$. 
Note that $L$ is positive semidefinite and has rank $n-1$, and that the null space of $L$ and $L^{\dagger}$ is $\spn{\onebf}$\cite{von2007tutorial}. Since $p$ satisfies Assumption~\ref{ass:balanced}, $p$ is in the range space of $L$, and hence the solution to $L \phi = p$ is given by $\phi = L^{\dagger}p + \bar{\phi}\onebf $, where $ \bar{\phi} $ is an arbitrary scalar.
Since $ A^T \onebf = 0 $, the flow solution is
unique: $f = W \trans{A} \phi = W \trans{A} L^{\dagger}p$.
\end{proof}

We shall drop explicit dependence of $f$ on $\mc G$, $w$ or $p$ when clear from the context. 





\begin{remark}
\label{rem:flow-sol}
\leavevmode
\begin{enumerate}[label = (\alph*)]
\item In the proof of Lemma~\ref{lemma:flow-solution}, the particular phase angle solution $ L_{\mc G}^{\dagger}p $ is the minimum norm solution and is
orthogonal to $ \onebf$, and the flow solution in (\ref{eq:sol}) is the minimum
weighted norm solution satisfying the flow conservation constraint, \ie, the
first equation in  (\ref{model}) (see Section~\ref{sec:flow-sol-prop} in the Appendix for more details). 
\item In \cite{bienstock2010nk}, a result similar to \eqref{eq:sol} is provided as $f = W \trans{\tilde{A}} (\tilde{A} W \trans{\tilde{A}})^{-1} \tilde{p}$, where $ \tilde{A} $ and $ \tilde{p} $ are one row reduced versions of $ A $ and $ p $, respectively.
\end{enumerate}
\end{remark}

We are interested in \emph{feasible} flows, \ie, flows that satisfy the following lower and upper line capacity constraints
\begin{equation}
\label{eq:line-capacity}
\clower \leq f \leq \cupper
\end{equation}
We call a network \emph{feasible} if the flows on all its links are feasible. 
Throughout this paper, we make the following rather natural assumption on line capacities:
\begin{assumption}
\label{ass:zero-interior}
$\clower < \zerobf < \cupper$ 
\end{assumption}

The capacities $\clower$ and $\cupper$ are typically symmetrical about $\zerobf$. 
%
We adopt the following natural standing assumption throughout the paper.
  \begin{assumption}
\label{ass:initial-flow-feasible}
 The initial flow $f_0 = f(w_{0}, p_{0})$ satisfies \(c^{l}\le f_{0} \le c^{u} \).
  \end{assumption}



\subsection{Weight Control Policies and the Margin of Robustness}
We are interested in quantifying disturbances on the supply-demand vector under which
\eqref{eq:line-capacity} continues to be satisfied, using $w$ as control. Disturbances will be modeled by change in the supply-demand vector. We assume that the disturbance induces a one-shot change to the system (as opposed to being a process). Formally, under a disturbance, the supply-demand vector changes irreversibly, at time $t=0$, from a
nominal value $p_0$ to a value $p_{\Delta}=p_0 + \Delta$, with
$\onebf^T \Delta  = 0$, and hence $\onebf^T p_{\Delta}=0$. We emphasize that the disturbance happens only at $t=0$, and is not a process. Such a balanced
disturbance can be caused, \eg, by removal of a link from the network. The
network responds by changing the weights dynamically, which in turn also induces
dynamics in the line flows due to \eqref{eq:sol}. This dynamics can be written as: 
\begin{equation}
\label{eq:ss-model-decent-control}
\begin{aligned}
\dot{w}_{i}(t)& = u_{i}\left( \mc W(t), \mc F(t), \Delta \right) 
\end{aligned}
\end{equation}
where $\mc W(t)=\{w(\kappa): \kappa \in [0,t]\}$, and $\mc
F(t)=\{f(w(\kappa)): \kappa \in [0,t]\}$ are the
historical values of line weights and flows, respectively, through time $t$.
The weight control in \eqref{eq:ss-model-decent-control} is required to satisfy the following constraints
\begin{equation}
\label{eq:weight-constraints}
\zerobf \leq \wlower \leq w \leq \wupper
\end{equation}
where $\wlower$ and $\wupper$ are the lower and upper limits, respectively, for the operation range of the weight controller. 
The dynamical system \eqref{eq:ss-model-decent-control} will be called \emph{feasible} under a given disturbance $\Delta$ and control policy $u$ if \eqref{eq:line-capacity} is satisfied asymptotically. 
For a given network \( \mathcal{G} = (\mc V, \mc E) \) with initial weight $ w_{0}\in \real^{ \mathcal{E}}_{>0} $, link weight bounds $w^l \in \real_{\ge 0}^{\mc E}$ and $w^u \in \real_{>0}^{\mc E}$, link capacity bounds $c^l \in \real_{<0}^{\mc E}$ and $c^u \in \real_{>0}^{\mc E}$, and initial supply-demand vector $ p_{0} \in \real^{ \mathcal{V}} $, the \emph{margin of robustness} of a given control policy $u$ is defined as 
\begin{equation}
\label{eq:robustness-margin-def}
\begin{split}
R(u,\mc G, w_{0}, w^{l}, w^{u}, c^{l}, c^{u}, p_{0}):=\sup\{\beta \geq 0: \, \eqref{eq:ss-model-decent-control} \text{ is }  \text{feasible under } u \quad  \forall \, \Delta \text{ s. t. } \|\Delta\|_1 \leq \beta\}.
\end{split}
\end{equation}

The choice of the $\ell_1$ norm in \eqref{eq:robustness-margin-def} is justified because we consider only balanced disturbances, and therefore $\|\triangle\|_1$ is equal to twice the cumulative deviation in supply (or demand). The following example provides a simple illustration of the increase in margin
of robustness when the line weights are controllable.  

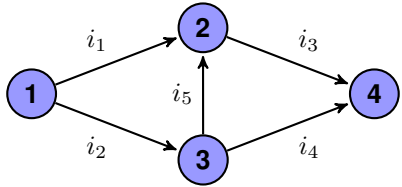
\begin{figure}[htbp]
\centering
\begin{tikzpicture}[->,>=stealth',shorten >=1pt,auto,node distance=0.4cm and 1.8cm,
                   thick,main node/.style={circle,draw,font=\sffamily\bfseries, fill=blue!40}]
 \node[main node] (1) {1};
 \node[main node] (2) [above right =of 1] {2};
 \node[main node] (3) [below right =of 1] {3};
 \node[main node] (4) [below right =of 2] {4};

 \path[every node/.style={font=\sffamily}]
   (1) edge node{$i_{1}$} (2)
   (1) edge node[below left]{$i_{2}$} (3)
   (3) edge node{\( i_{5} \)} (2)
   (2) edge node{\( i_{3} \)} (4)
   (3) edge node[below right]{\( i_{4} \)} (4);
\end{tikzpicture}
\caption{Network used in Examples~\ref{eg:topology-nonmonotonicity} and \ref{eg:load-nonmonotonicity}.}
\label{fig:load-nonmon}
\end{figure}

\begin{example}
\label{eg:topology-nonmonotonicity}
 Consider the network shown in Figure~\ref{fig:load-nonmon} with $\wupper = [1 \, \, \, 3 \, \, \, 1 \, \, \, 1 \, \, \, 1]^T$, $\wlower$ the same as $\wupper$, except for $\wlower_2=0$; $\cupper= -\clower= [1 \, \, \, 1 \, \, \, 1 \, \, \, 0.5 \, \, \, 1]^T $ and $p_0=[1 \, \, 0 \, \, \, 0 \, \, -1]^T$. The flow corresponding to weight $\wupper$ and load $ p_{0} $ is $f(\wupper,p_0)=[0.33 \, \, \,  0.67 \, \, \, 0.44 \, \, \, 0.56 \, \, \, 0.11]^T$ which is infeasible due to the excessive flow on link $ e_4 $. However, the flow under the same load $ p_{0} $ but with weight $\wlower$ is $f(\wlower,p_0)=[1.00 \, \, \, 0 \, \, \, 0.67 \, \, \, 0.33 \, \, \, -0.33]^T$ which is feasible. 
\end{example}

Choosing $w_i=0$ for some link $i$, \eg, for link $2$ in Example~\ref{eg:topology-nonmonotonicity}, corresponds to disconnecting that link. 
Such \emph{line tripping} strategies have been considered in the context of cascading failures \cite{you2003self, you2004slow, andersson2005causes}.

Our objectives in this paper are:  (i) to provide a framework for tractable computation of (approximations of)  
\begin{equation}
\label{eq:Rstar-def}
R^*(\mc U, \mc G, w_{0}, w^{l}, w^{u}, c^{l}, c^{u}, p_{0} ):=\sup_{u \in \mc U} \, R(u, \mc G, w_{0}, w^{l}, w^{u}, c^{l}, c^{u}, p_{0} )
\end{equation}
for a given class $\mc U$ of control policies, and (ii) to find $u^* \in \mc U$ such that $R(u^*,\mc G, w_{0}, w^{l}, w^{u}, c^{l}, c^{u}, p_{0})$ is a close approximation of, if not equal to $R^*(\mc U, \mc G, w_{0}, w^{l}, w^{u}, c^{l}, c^{u}, p_{0})$.
We shall drop explicit dependence of  $ R $ and  $ R^{*} $ on, $u$, $ \mathcal{U} $, $\mc G$, $w_{0}$, $w^{l}$, $ w^{u} $, $ c^{l} $,  $ c^{u} $, or $ p_{0} $ when clear from the context. 

In this paper, we specifically consider cases when $\mc U$ is the set of centralized or \emph{decentralized} control policies. The latter corresponds to control policies satisfying $u_i(\mc W(t),\mc F(t), \Delta ) \equiv u_i(\mc W_i(t), \mc F_i(t))$, \ie, controller on link $i$ has access to the historical values of weights and flows only on link $i$, and no information about disturbance. 

\subsection{Upper Bound on the Margin of Robustness}
\label{sec:upper-bound}
It is easy to see that $R^*(\mc U)$ with $\mc U$ being the set of centralized policies (that have access to information about disturbance $\Delta$) serves as an upper bound to $R^*(\mc U)$ for any class of control policies, including decentralized control policies. 

Under a centralized control policy, the new link weights $w^*(\Delta)$ in
response to disturbance $\Delta$ are chosen instantaneously, \ie, there is no
dynamics in $w$. The centralized policy then corresponds to setting
$w^*(\Delta)$ to be equal to any $w \in \real_{\ge 0}^{\mc E}$ satisfying: 
\begin{equation}
\label{eq:centralized}
\begin{aligned}
\clower  \leq f(w, p_{\Delta}) \leq \cupper \\
 \wlower  \le w \le \wupper 
\end{aligned}
\end{equation}
if \eqref{eq:centralized} is feasible, and (arbitrarily) equal to $\wupper$
otherwise. Here, \( f(w, p_{\Delta}) \) is as given in (\ref{eq:sol}). 
The margin of robustness of such a centralized policy can be easily seen to be equal to the solution of the following optimization problem: 
%

\begin{equation}
\label{eq:weight-control-original}
\nu^{*}(\mathcal{G}, w^{l}, w^{u}, c^{l}, c^{u},  p_{0}):= \underset{\delta: \, \|\delta\|_1=1, \, \onebf^T \delta = 0}{\text{min}} \; \nu(\delta, \mathcal{G}, w^{l}, w^{u}, c^{l}, c^{u}, p_{0}) 
\end{equation}
where 
\begin{equation}
\label{eq:weight-control-nu}
\begin{aligned}
\nu( \delta, \mathcal{G}, w^{l}, w^{u}, c^{l}, c^{u}, p_{0}) = \; \, \,& \underset{ w \in \real_{\geq 0}^{\mc E}; \, \mu \geq 0}{\text{max}} &&  \mu \\
& \text{subject to}  &&  
  \clower \leq f(w, p_{0}+ \mu \delta) \leq \cupper \\
& && \wlower \le w \le \wupper 
\end{aligned}
\end{equation} 
%
For brevity, the explicit dependence of $ \nu( \delta, \mathcal{G}, w^{l}, w^{u}, c^{l}, c^{u}, p_{0})  $ and $ \nu^{*}( \mathcal{G}, w^{l}, w^{u}, c^{l}, c^{u}, p_{0})  $ on $ \mathcal{G} $, $ w^{l} $, $ w^{u} $, $ c^{l} $, $ c^{u} $ or $ p_{0} $ is dropped when clear from the context. 
Notice while a control policy and its margin of robustness may depend on the initial weight $ w_{0} $, 
the upper bound, as defined in \eqref{eq:weight-control-original}-\eqref{eq:weight-control-nu} does not.
\eqref{eq:weight-control-original} differs from \eqref{eq:centralized} only in
parameterization of the set of disturbances in terms of disturbances on a unit
$\ell_1$-ball and magnitude $\param$. 
(\ref{eq:weight-control-nu})
only considers disturbances along the direction \( \delta \) and \(
\nu(\delta, \mathcal{G})  \) gives the maximal magnitude of such
disturbances under which the system can be made feasible within the
specified operation range on \( w \). (\ref{eq:weight-control-original}) then
considers all possible directions of balanced disturbances and 
$\param^*( \mathcal{G})$ is an upper bound on every, including decentralized, control policies as we show next. 

\begin{lemma}
\label{lem:upper-bound}
For a network with directed multigraph \( \mathcal{G} = (\mc V, \mc E) \), link weight bounds $w^l \in \real_{\ge0}^{\mc E}$ and $w^u \in \real_{>0}^{\mc E}$, link capacity bounds $c^l \in \real_{<0}^{\mc E}$ and $c^u \in \real_{>0}^{\mc E}$, and initial supply-demand vector $ p_{0} \in \real^{ \mathcal{V}}$, there exists a $\Delta \in \real^{\mc V}$ with $\|\Delta\|_1$ arbitrarily greater than $\param^*(\mc G)$ such that the system~\eqref{eq:ss-model-decent-control} is infeasible under every, including decentralized, control policy $u$.
\end{lemma}
\begin{proof}
Let $(\delta^*,\param^*)$ correspond to an optimal solution of \eqref{eq:weight-control-original}. It is easy to see that for $(\delta^*,\param^*(1+\epsilon))$, $\epsilon>0$, there is no feasible $w$ in \eqref{eq:weight-control-original}, and hence the system is infeasible under the perturbation $\Delta=(1+\epsilon) \param^* \delta^*$ under any control policy $u$. Since this is true for every $\epsilon>0$, this gives the lemma.
\end{proof}  

Lemma~\ref{lem:upper-bound} implies that 
$R^* \leq \param^*$, or equivalently, $R(u) \leq \param^*$ for all $u$. The next example shows that $\|\Delta\|_1 > R^*$ is not sufficient for infeasibility.
\begin{example}
\label{eg:load-nonmonotonicity}
Consider the network shown in Figure~\ref{fig:load-nonmon}, with  $\cupper = - \clower = 5.5 \, \onebf$, and $\wlower = \wupper = [1 \, \, \, 3 \, \, \, 3 \, \, \, 1 \, \, \, 1]^T = w$ (say). This implies that the only admissible control policy is the trivial $u \equiv \zerobf$. The flow corresponding to load $p_0=[8 \, \, 0 \, \, \, 0 \, \, -8]^T$ is $f(w,p_0) = [3.2 \, \, \, 4.8 \, \, \, 4.8 \, \, \, 3.2 \, \, \, 1.6]^T$ which is feasible. Consider two perturbations $\Delta_1 = [1.5 \, \, \, -0.5 \, \, \, 0.5 \, \, \, -1.5]^T$ and $\Delta_2=[2 \, \, -2 \, \, \, 2 \, \, -2]^T$. Note that $\|\Delta_1\|_1 = 4 < 8 = \|\Delta_2\|_1$. The flows under these perturbations are $f(w,p_{\Delta_1})=[3.95 \, \, \, 5.55 \, \, \, 5.55 \, \, \, 3.95 \, \, \, 2.1]^T$ and  $f(w,p_{\Delta_2})=[4.6 \, \, \, 5.4 \, \, \, 5.4 \, \, \, 4.6 \, \, \, 2.8]^T$. 
Since $f(w,p_{\Delta_1})$ is infeasible, $R^* \leq \|\Delta_1\|_1=4$. However, the flow under $\Delta_2$, whose norm is greater than $R^*$ is feasible. 
Note also that, element-wise, $\Delta_1$ and $\Delta_2$ have the same signs, and magnitude of $\Delta_1$ is smaller than $\Delta_2$. In other words, $\Delta_2$ dominates $\Delta_1$ element-wise, and yet the system is feasible under $\Delta_2$, but not under $\Delta_1$. Such non-monotonicity is directly attributable to non-monotonicities of flow distribution with respect to the supply-demand vector in power networks.
\end{example}

%% file: mor-and-mincut-qin.tex
\section{Relationship between margin of robustness and min-cut capacity}
In the weight control problem (\ref{eq:weight-control-original})-(\ref{eq:weight-control-nu}), the flexibility of controlling weight enables us to adjust the flow
distribution over the networks to maximize the margin of robustness. This is similar to a
classical network flow problem of choosing a feasible flow distribution to optimize a given cost function. In this section, we formally investigate the relationship between the weight control and the network flow problem. We show that, under appropriate conditions, the weight control problem (\ref{eq:weight-control-nu}) is
equivalent to a network flow problem with the margin of robustness for a given disturbance being the
objective function. We use this relationship to make connections between the margin of robustness of a given DC power network and the min cut capacity of a certain associated flow network. These results are reminiscent of our previous
work in \cite{Como.Savla.ea:Part1TAC10,Como.Savla.ea:Part2TAC10} on robustness of transport networks.

We begin by exploring the relationship between feasible flow sets for network flow and weight controlled DC power networks. 

\subsection{Relationship between Feasible Flow Sets for Network Flow and
  Weight Controlled DC Power Network}
The difference between a flow network and a DC power network is in their
different physics: classical network flow has capacity and flow conservation constraints, whereas DC power networks have additional constraints in the form of Ohm's law. Let us define the set of \emph{feasible flow} for
flow networks, \( \mathcal{F}_{1} \),  and for weight controlled DC networks, \( \mathcal{F}_{2} \),  
as follows \footnote{In contrast to standard convention, we do not include non-negativity constraints in $\mc F_1$ since the underlying graph is undirected. The non-negativity constraints can be imposed on the directed graph formed by a simple \emph{extension}: for every undirected link in $\mc G$, there are two directed links in the extended directed graph.}: 
\begin{align*}
  \mathcal{F}_{1} &:= \{f \in \real^{\mc E}\,|\, {A}f = p, c^{l} \le f \le c^{u}\} \\
  \mathcal{F}_{2} &:= \{ f \in \real^{\mc E}\,|\, \exists w \in [w^l,w^u], \, \, \phi \in \real^{\mc V}, \text{ s.t. } {A}f = p, c^{l} \le f \le c^{u}, f = w\trans{A} \phi \}
\end{align*}
Since $ \mathcal{F}_{2} $ has additional constraints, it is straightforward to see that \(  \mathcal{F}_{2} \subseteq \mathcal{F}_{1} \). 

For a network with directed graph \( \mc G = ( \mathcal{V}, \mathcal{E} ) \), a cycle \(
\mathcal{C} \), is a subset of \( \mathcal{E} \) that forms a loop. \(
\mathcal{C} \) consists of forward link set \( \mathcal{C}_{F} \) and backward link
set \( \mathcal{C}_{B} \), where the \emph{forward links} and \emph{backward links} are the links 
along clockwise and counter-clockwise direction of \( \mathcal{C} \),
respectively \cite{bertsimas1997introduction}. 
We say that a flow \( f \in \real^{\mc E}\) contains a \emph{circulation} if there
exists a cycle \( \mathcal{C} \) such that \( f_{i} > 0 \) for all \( i\in
\mathcal{C}_{F} \) and \( f_{i}<0 \) for all \( i\in \mathcal{C}_{B} \). Let \(
\mathcal{F}_{0} := \{f \in \real^{ \mathcal{E}} \, |\, f \text{ does not contain a circulation}\}
\). We then have the following relationship between \(
\mathcal{F}_{0} \), \( \mathcal{F}_{1} \) and \( \mathcal{F}_{2} \). 

\begin{proposition}
\label{prop:feasible-flow-relation}
For a network with undirected multigraph \( \mathcal{G} = (\mc V, \mc E) \), link weight bounds $w^l \in \real_{\ge0}^{\mc E}$ and $w^u \in \real_{>0}^{\mc E}$, link capacity bounds $c^l \in \real_{<0}^{\mc E}$ and $c^u \in \real_{>0}^{\mc E}$, and supply-demand vector $ p \in \real^{ \mathcal{V}}$,
\begin{equation}
\label{eq:F2-subset-F1-cap-F0}
\mathcal{F}_{2} \subseteq \mathcal{F}_{1} \cap \mathcal{F}_{0}
\end{equation}
Moreover,
\begin{enumerate}
\item if $\mc G$ is a tree, then $\mc F_1 = \mc F_2$
\item if $w^l=0$, then $\mathcal{F}_{2} =
 \mathcal{F}_{1} \cap \mathcal{F}_{0}$
\end{enumerate}
\end{proposition}
\begin{proof}
Since \( \mathcal{F}_{2} \subseteq \mathcal{F}_{1} \), in order to prove \eqref{eq:F2-subset-F1-cap-F0}, it is sufficient to prove
  that \( f \in \mathcal{F}_{0} \) for all \( f \in \mathcal{F}_{2} \), \ie, a
  feasible flow for a DC network does not contain a circulation. This is proven by contradiction as follows. For a flow \( f \in \mathcal{F}_{2} \), suppose there exists a circulation on a cycle \( \mathcal{C}
  \). Applying Ohm's law on all the links in \( \mathcal{C} \), we then get that
$f_{i}/w_{i} = \phi_{\sigma(i)}  -\phi_{\tau(i)}$ for all $i \in
\mathcal{C}_{F}$, and $-f_{i}/w_{i} = \phi_{\sigma(i)} - \phi_{\tau(i)}$ for all $i \in \mathcal{C}_{B}$. Taking summation over all links in $\mc C$, we then get that 
%
\begin{equation*}
0 < \sum_{i\in \mathcal{C}_{F}} f_{i}/w_{i} - \sum_{i\in \mathcal{C}_{B}}
f_{i}/w_{i}  = \sum_{i\in \mathcal{C}_{F}} \phi_{\tau(i)}  + \sum_{i\in
  \mathcal{C}_{B}} \phi_{\sigma(i)} - \sum_{i\in \mathcal{C}_{F}}\phi_{\sigma(i)}
- \sum_{i\in \mathcal{C}_{B}}\phi_{\tau(i)}  = 0. 
\end{equation*}
where the inequality is due to the definition of circulation, and the last equality to zero is due to the definition of a cycle. This leads to a contradiction. 

In order to prove (1), it is sufficient to prove that \( \mathcal{F}_{1} \subseteq \mathcal{F}_{2} \), \ie, \( f \in \mathcal{F}_{2} \)
for any \( f \in \mathcal{F}_{1} \) for a tree network. Pick arbitrary \( f \in
\mathcal{F}_{1} \) and \( w \in [w^{l}, w^{u}] \). It is sufficient to show
that the constraint \( f = w\trans{A} \phi \) is satisfied for some \( \phi
\in \real^{\mc V}\). Let \( \bar{A} \) be the subvector and submatrix of $ \phi $ and $ A $ respectively with the
first row removed. Since $\mc G$ is a tree, \( A \) has independent columns, and \( \bar{A} \) is full rank. Let $ \bar{\phi} := (\trans{\bar{A}})^{-1} w^{-1} f  \in \real^{|\mc V|-1}$. It is then easy to see that \( f = w\trans{A} \phi \) is satisfied for $\phi:=[0 \, \, \bar{\phi}^T]^T$. 

In order to prove (1), it is sufficient to prove that \(\mathcal{F}_{1} \cap \mathcal{F}_{0} \subseteq \mathcal{F}_{2}  \). Pick arbitrary \( f \in  \mathcal{F}_{1} \cap \mathcal{F}_{0} \). To prove \( f\in \mathcal{F}_{2} \) is to show there exist \( w\in [w^{l},
w^{u}] \) and \( \phi \) such that the constraint \( f = w \trans{A}\phi \) is satisfied. We now construct such \( w \) and \( \phi \) as follows. Maintain the directions of links with positive flow
and reverse the directions of links with negative flow. Since \( f\in
\mathcal{F}_{0} \), there is no directed cycle in the network with the new
direction assigned. Hence, there exists a topological ordering of the nodes in $\mc V$. Pick a strictly decreasing sequence $(\phi_1, \ldots, \phi_{|\mc V|})$, and assign it the nodes as per the topological ordering. Let $\tilde{w}_i:=f_{i}/( \phi_{\sigma(i)} -
  \phi_{\tau(i)} ) >0$ for all $i \in \mc E$. Finally, choose the link weights as: $w= \eta \tilde{w}$, where $\eta=\min_{i \in \mc E}
  w_{i}^{u}/w_{i} > 0$.
\end{proof}

\begin{remark}
\label{rem:tree-flow}
  If the underlying undirected graph \( \mathcal{G} \) of a network is a tree, then the flow
solution to (\ref{model}) is uniquely determined by the 
flow conservation equation \( Af = p \) and hence changing weight \( w \) does not
affect the value of \( f \). Therefore, as we show in Section~\ref{sec:min-cut}, the weight control
problem (\ref{eq:weight-control-nu})  of a such a network reduces to a network flow problem. 
\end{remark}

\subsection{Relating Margin of Robustness to Min-Cut Capacity}
\label{sec:min-cut}
Proposition~\ref{prop:feasible-flow-relation} implies that, for a network whose underlying graph is a tree, 
(\ref{eq:weight-control-original})-(\ref{eq:weight-control-nu}) 
is equivalent to:
\begin{equation}
  \label{opt:tree-weight-control}
  \begin{aligned}
    \nu_{0}^{*} := \underset{\delta: \,\|\delta\|_1=1, \, \onebf^T \delta = 0}{\text{min}} \nu_{0}(\delta)
  \end{aligned}
\end{equation}
where
\begin{equation}
\label{eq:tree-weight-control-nu}
 \begin{aligned}
\nu_{0}(\delta) := \; \, \,& \underset{ \mu \geq 0, \, f}{\text{max}} &&  \mu \\
& \text{subject to}  &&  Af = p_0 + \mu \delta  \\
&&&  \clower \leq f \leq \cupper \\
\end{aligned}
\end{equation}

If the underlying graph is not a tree, a feasible flow \(
f \in \mathcal{F}_{1} \) can contain circulations, i.e., \( f \notin \mathcal{F}_{0} \), and hence  \( f \notin \mathcal{F}_{2} \) by Proposition~\ref{prop:feasible-flow-relation}. In this case, it is possible to eliminate circulations from $f$ to obtain a \( \tilde{f} \in \mathcal{F}_{1} \cap \mathcal{F}_{0}
\) as follows. Set $\tilde{f}=f$. While $\tilde{f}$ contains a circulation for some cycle $\mc C$, update $\tilde{f}=\tilde{f}- \min_{i \in \mc C} \tilde{f}_i \, \onebf_{\mc C}$, where $\onebf_{\mc C}$ is a
  binary vector containing one for entries corresponding to $\mc C$, and zero otherwise. Moreover, it is easy to see that, if $(\mu,f)$ is feasible for \eqref{eq:tree-weight-control-nu}, then $(\mu,\tilde{f})$ is also feasible. Proposition~\ref{prop:feasible-flow-relation} implies that the flow obtained by removing circulation satisfies \( \tilde{f}\in \mathcal{F}_{2} \) if \( w^{l} = 0 \). 
%
Therefore,
(\ref{eq:weight-control-original})-(\ref{eq:weight-control-nu})  is
equivalent to (\ref{opt:tree-weight-control})-(\ref{eq:tree-weight-control-nu}) 
when \( w^{l} =0 \). 

In summary, if the underlying graph of a network is a tree or \( w^{l} =0 \), then the nonconvex problem
(\ref{eq:weight-control-original})-(\ref{eq:weight-control-nu}) is equivalent to 
(\ref{opt:tree-weight-control})-(\ref{eq:tree-weight-control-nu}), whose inner problem (\ref{eq:tree-weight-control-nu}) is convex. Indeed, (\ref{eq:tree-weight-control-nu})
is a classical network flow problem and can be solved efficiently for a given disturbance \( \delta \). However, computational tractability of 
the minimax problem
(\ref{opt:tree-weight-control})-(\ref{eq:tree-weight-control-nu}) is not readily apparent. The next result establishes a useful property of $\nu_0(\delta)$, which in turn will lead to an efficient solution methodology for (\ref{opt:tree-weight-control})-(\ref{eq:tree-weight-control-nu}). 
%
\begin{lemma}
  \label{lem:tree-concave-nu}
 For a network with directed multigraph \( \mathcal{G} = (\mc V, \mc E) \), link
 capacity bounds $c^l \in \real_{<0}^{\mc E}$ and $c^u \in \real_{>0}^{\mc E}$,
 and initial supply-demand vector $ p_{0} \in \real^{ \mathcal{V}} $,
  \(\nu_{0}(\delta) \) 
defined in (\ref{eq:tree-weight-control-nu}) is quasiconcave. 
\end{lemma}
\begin{proof}
  Given arbitrary \( \delta_{1} \) and \(\delta_{2}\), we show that \(
  \nu_{0}(\theta \delta_{1}+ (1- \theta)\delta_{2}) \ge
  \min\{\nu_{0}(\delta_{1}), \nu_{0}(\delta_{2})\} \) for all \( \theta \in [0, 1] \). Let \( u^{*}_{1} =
  \nu_{0}(\delta_{1}) \) and \( u_{2}^{*}=
  \nu_{0}(\delta_{2})  \). Without loss of generality, assume
  \( u^{*}_{1} \le u_{2}^{*} \) and we need to prove \(
  \nu_{0}(\theta \delta_{1}+ (1- \theta)\delta_{2}) \ge
  u_{1}^{*} \). It is sufficient to show that \( u = u_{1}^{*} \) is feasible to
  (\ref{eq:tree-weight-control-nu}) when \( \delta = \theta \delta_{1} +
  (1-\theta)\delta_{2} \). 

When \( \delta = \theta \delta_{1} +   (1-\theta)\delta_{2} \), \( u =
u_{1}^{*} \), the equality constraint becomes 
\begin{align*}
  A f &= p_{0} + u_{1}^{*} (\theta \delta_{1} + (1-\theta) \delta_{2}) = \theta
  (p_{0}+  u_{1}^{*} \delta_{1}) + (1-\theta_{2}) (p_{0} +u_{1}^{*}
  \delta_{2}) \\
&= \theta A f_{1}^{*} + (1-\theta) A f'_{2}
\end{align*}
where \( f_{1}^{*} \) and \( f'_{2} \) are some flow on the network under
disturbed supply-demand vector \( p_{0} + u_{1}^{*} \delta_{1} \) and \(
p_{0}+ u_{1}^{*} \delta_{2} \), respectively. By setting \( f = \theta f_{1}^{*} + (1-\theta) f'_{2}
\), the flow conservation constraint is satisfied. For feasibility of $ ( \theta \delta_{1} +   (1-\theta)\delta_{2}, u_{1}^{*}) $, what remains to be shown is that  such  \( f \) satisfies the
capacity constraint. It is sufficient to show that there exist \( f_{1}^{*} \) and \( f'_{2} \)
that are feasible. $ f_{1}^{*} $ can be selected as the optimal solution to (\ref{eq:tree-weight-control-nu}) corresponding to $ u_{1}^{*} $ and hence feasible. In order to see that there exists feasible
\( f'_{2} \), note that the feasible set of (\ref{eq:tree-weight-control-nu})
is a polyhedron, \(\nu = 0\), \( f= f_{0} \) and \( \nu = \nu_{2}^{*} \) and \(
f= f_{2}^{*} \) are feasible, where $ f_{2}^{*} $ is the optimal flow solution corresponding to $ u_{2}^{*} $, and \( u_{1}^{*} \le u_{2}^{*}
\) is convex combination of 0 and \( u_{2}^{*} \). Therefore, \(
\nu_{0}(\theta \delta_{1}+ (1- \theta)\delta_{2}, \mathcal{G}_{t} ) \ge 
  u_{1}^{*} \) and \( \nu_{0}(\delta) \) is quasiconcave. 
  \end{proof}

\begin{lemma}
\label{lem:extreme-pt-optimal}
Consider a network with directed multigraph \( \mathcal{G} = (\mc V, \mc E) \), link
 capacity bounds $c^l \in \real_{<0}^{\mc E}$ and $c^u \in \real_{>0}^{\mc E}$,
 and initial supply-demand vector $ p_{0} \in \real^{ \mathcal{V}} $. Then, $\nu_0^*$ defined in \eqref{opt:tree-weight-control} is equal to $\min_{\delta \in \Delta_0} \nu_0(\delta)$, where \( \Delta_{0} :=
 \{\delta \in \real^{\mathcal{V}} \, |\, \exists \, s, t \in \mathcal{V}, \delta_{s} = 1/2, \delta_{t} = -1/2,
 \delta_{v} = 0 \, \, \forall \, v \in \mathcal{V}\setminus \{s, t\} \} \), and $\nu_0(\delta)$ is as defined in \eqref{eq:tree-weight-control-nu}.
\end{lemma}
\begin{proof}
The feasible set \( \{\delta \in \real^{\mathcal{V}} \,|\, \|\delta\|_{1}=1, \trans{\onebf} \delta =0 \} \) for (\ref{opt:tree-weight-control}) is a polytope. 
We now show that \( \{\delta \in \real^{\mathcal{V}} \,|\, \|\delta\|_{1}=1, \trans{\onebf} \delta =0 \} \)  is the
convex hull of set \( \Delta_{0} \). The result then follows by using Lemma~\ref{lem:tree-concave-nu}, and Lemma~\ref{lem:min-quasi-concave} (in the Appendix).


Pick an arbitrary
\( \delta \in \real^{\mathcal{V}} \) with \(  \|\delta\|_{1}=1 \) and \(
\trans{\onebf} \delta =0 \). We now show that there exist $ \{\eta_{k}\} $ and $\{\delta^{0}_{k}\} $ satisfying $\eta_k \geq 0$ and $\delta^{0}_{k} \in \Delta_{0}$ for all $k$, and $\sum_k \eta_k =\|\delta\|_{1}= 1$. Let $\tilde{\delta}=\delta$, and $k=1$. While $\tilde{\delta} \neq \zerobf$, do the following. Let \( \mathcal{V}^{+} := \{ v \; |\; \tilde{\delta}_{v} >0\} \), \(
  \mathcal{V}^{-} := \{ v \; |\; \tilde{\delta}_{v} <0\} \), and pick $v^* \in \argmin_{v \in \mc V^+ \cup \mc V^-}$, and let $\eta_k := 2 | \tilde{\delta}_{v^*}|$. If $v^* \in \mc V^+$, then let $\tilde{\delta}_{v^*}=\tilde{\delta}_{v^*}-\eta_k/2$, pick arbitrary $v' \in \mc V^-$, and let $\tilde{\delta}_{v'}=\tilde{\delta}_{v'}+\eta_k/2$. $\delta^0_k$ is then chosen such that $\delta^0_{k,v^*}=1/2$, $\delta^0_{k,v'}=-1/2$, and $\delta^0_{k,v}=0$ for all $v \in \mc V \setminus \{v^*,v'\}$. One can similarly choose $\delta^0_k$ when $v^* \in \mc V^-$. We then set $k=k+1$, and repeat the process for selecting $\delta^0_k$ and $\eta_k$ while $\tilde{\delta} \neq \zerobf$.
\end{proof}

Lemma \ref{lem:extreme-pt-optimal} implies that, in order to solve (\ref{opt:tree-weight-control})-(\ref{eq:tree-weight-control-nu}), it is sufficient to consider a finite number of disturbance directions $ \delta \in \Delta_{0} $, each with only one positive and one negative component. Then a naive solution strategy to compute $\nu^*$ for a network with tree topology or $w^l=0$ is to solve (\ref{eq:tree-weight-control-nu}) for all the disturbance directions in $ \Delta_{0} $ and then take the minimum. However, by using  the Max-Flow-Min-Cut theorem, \eg, \cite[Theorem 8.6]{Korte.Vygen:02}, one can execute this step in a simpler way as we describe next. In order to do this, we first construct a flow network associated with the given network, where we recall the standing Assumption~\ref{ass:initial-flow-feasible}.

\begin{definition}[Associated Flow Network]
\label{def:expanded-graph}
Consider a network with directed multigraph \( \mathcal{G} = (\mc V, \mc E) \), link weight bounds
$w^l \in \real_{\ge 0}^{\mc E}$ and $w^u \in \real_{>0}^{\mc E}$, link
capacity bounds $c^l \in \real_{<0}^{\mc E}$ and $c^u \in \real_{>0}^{\mc E}$, initial supply-demand vector $ p_{0} \in \real^{ \mathcal{V}}$, and initial weights $w_0 \in [w^l,w^u]$. Let $f_0$ be the corresponding initial flow, as given by \eqref{eq:sol}. The associated flow network $(\mc G^{\mathrm{fl}}, c^{\mathrm{fl}})$ consists of a directed graph $\mc G^{\mathrm{fl}}=(\mc V, \mc E^{\mathrm{fl}})$, where $\mathcal{E}^{\mathrm{fl}}$
is the union of $\mc E$ and as well as reversed versions of links in $\mc E$, and link capacities $c^{\mathrm{fl}}$ defined
as $c^{\mathrm{fl}}_{i}:=c^u_{i} - f_{0, i}$ if $i \in \mc E$, and
$c^{\mathrm{fl}}_{i}:=-c^l_{i}+ f_{0, i}$ if $i \in \mathcal{E}^{\mathrm{fl}} \setminus \mc E$. Assumption \ref{ass:initial-flow-feasible} imply that $c^{\mathrm{fl}} \geq 0$.
\end{definition}
A \emph{cut} in \(
\mathcal{G}^{\mathrm{fl}} \) is
a partition of the node set \( \mathcal{V} \) into two nonempty subsets: \(
\mathcal{V}_{c} \) and its complement \( \mathcal{V} \setminus \mathcal{V}_{c}
\) \cite{Bertsekas:98} \footnote{ A cut is denoted as \( \mathcal{V}_{c} -
\mathcal{V}\setminus \mathcal{V}_{c} \) cut. It is uniquely
  determined by and determines a node set \( \mathcal{V}_{c} \). The partition is
  ordered in the sense that the cut \( \mathcal{V}_{c} - \mathcal{V}\setminus
  \mathcal{V}_{c} \) is distinct from the cut \( \mathcal{V} \setminus
  \mathcal{V}_{c} - \mathcal{V}_{c} \).}. \emph{Cut capacity} is a function \(
\map{C}{2^{\mathcal{V}}\setminus \{ \emptyset \cup \mathcal{V}\} \times \real_{\ge 0}^{\mc E}}{\real_{\ge
    0} } \) over the cuts and flow capacities and defined as: 
\begin{equation*}
C(\mathcal{V}_{c},c^{\mathrm{fl}})  = \sum_{i: \sigma(i) \in \mc V_c, \, \tau(i) \notin \mc V_c} c^{\mathrm{fl}}_i
\end{equation*}
The min-cut capacity \( C_{\min}(\mathcal{G}^{\mathrm{fl}},c^{\mathrm{fl}}) \) of \(
\mathcal{G}^{\mathrm{fl}} \) is the minimum cut capacity among all cuts in \(
\mathcal{G}^{\mathrm{fl}} \), \ie, \( C_{\min}(\mathcal{G}^{\mathrm{fl}},c^{\mathrm{fl}}) =
\min_{\emptyset \subsetneq \mathcal{V}_{c} \subsetneq \mathcal{V} } C(\mathcal{V}_{c},c^{\mathrm{fl}}) \). 
The next proposition relates the margin of robustness 
to the min-cut capacity of the associated flow network. 

\begin{proposition}
\label{prop:margin-min-cut-upper-bound}
Consider a network with directed multigraph \( \mathcal{G}= (\mathcal{V}, \mathcal{E})
  \), link weight bounds $w^l \in \real_{\ge 0}^{\mc E}$ and $w^u \in \real_{>0}^{\mc E}$, link
capacity bounds $c^l \in \real_{<0}^{\mc E}$ and $c^u \in \real_{>0}^{\mc E}$, initial supply-demand vector $ p_{0} \in \real^{ \mathcal{V}}$, and initial link weights $w_0 \in [w^l,w^u]$. Then,
its margin of robustness $\nu^{*}(\mc G)$ is upper bounded as $\nu^{*}(\mathcal{G}) \le 2 C_{\min}(\mathcal{G}^{\mathrm{fl}},c^{\mathrm{fl}})$, 
where $(\mathcal{G}^{\mathrm{fl}},c^{\mathrm{fl}})$ is the associated flow network (cf. Definition~\ref{def:expanded-graph}). Moreover, if $\mc G$ is a tree or $w^l=0$, then $\nu^{*}(\mathcal{G}) = 2 C_{\min} (\mathcal{G}^{\mathrm{fl}},c^{\mathrm{fl}})$. In particular, if $\mc G$ is a tree, then $\nu^{*}(\mathcal{G}) = 2 \min_{i \in \mathcal{E}} \{ f_{0, i} - c_{i}^{l},
c_{i}^{u} - f_{0, i} \}$, where $f_0$ is the initial flow, as given by \eqref{eq:sol}.
\end{proposition}
\begin{proof}
We first prove the equality for the case when $\mc G$ is a tree or $w^l=0$. In this case, \(
\nu^{*}(\mathcal{G}) = \nu_{0}^{*} \), and hence it is equivalent to proving \(
\nu_{0}^{*} = 2(\mathcal{G}^{\mathrm{fl}},c^{\mathrm{fl}}) \). Following Lemma
\ref{lem:extreme-pt-optimal}, for a given \( \delta \in \Delta_{0} \) with \( \delta_{s} = 1/2
\), \( \delta_{t} = -1/2 \) and \(\delta_{v} = 0\) for all \( v\in
\mathcal{V}\setminus \{s, t\} \), the Max-Flow-Min-Cut
theorem, \eg \cite[Theorem 8.6]{Korte.Vygen:02}, implies that \( \nu_{0}(\delta)=
2 \min_{\mathcal{V}_{c}: s\in \mathcal{V}_{c}, t \notin \mathcal{V}_{c} }
C(\mathcal{V}_{c},c^{\mathrm{fl}})\). Therefore, \( \nu^{*}_{0} =   \underset{\delta \in \Delta_{0}}{\text{min}} \nu_{0}(\delta) = 2
\min_{\emptyset \subsetneq \mathcal{V}_{c} \subsetneq \mathcal{V} }
C(\mathcal{V}_{c},c^{\mathrm{fl}}) = 2 C_{\min}(\mathcal{G}^{\mathrm{fl}},c^{\mathrm{fl}}) \) .

It is easy to see that $\nu^*(\mc G)$ is upper bounded by the margin of robustness for a network with the same attributes for $(\mc G, w^u, c^l, c^u, p_0, w_0)$ and $w^l=0$ (since it expands the feasible set in \eqref{eq:weight-control-nu}). We have already shown in the previous paragraph that the latter is equal to $2 C_{\min}(\mathcal{G}^{\mathrm{fl}},c^{\mathrm{fl}})$.

 
The exact expression of $\nu^*(\mc G)$ when $\mc G$ is a tree follows from the fact that, in this case, each link separates the
network, and hence \( C_{\min}(\mathcal{G}^{\mathrm{fl}},c^{\mathrm{fl}}) = \min_{i\in \mathcal{E}^{\mathrm{fl}}} c^{\mathrm{fl}}_{i} \).
\end{proof}

There exists an extensive literature on efficient computation of min-cut
capacity, which can be used to provide upper bound or exact characterization of the margin of robustness under special cases, as per Proposition~\ref{prop:margin-min-cut-upper-bound}. However, computing the exact value of margin of robustness in the general case requires solution to the non-convex
problem \eqref{eq:weight-control-original}. In Sections~\ref{sec:multiplicative} and \ref{sec:bi-level}, we propose methodologies to compute this margin for more general networks: we provide a projected gradient descent algorithm
(Section~\ref{subsec:projected-subgradient}) for multiplicative disturbances, and a multilevel programming approach (Section~\ref{sec:bi-level})  for nongenerative disturbances. 


%% file: multiplicative-disturbance.tex

\section{The Multiplicative Disturbance Case}
\label{sec:multiplicative}
In this section, we restrict our attention to the class  of disturbances that are multiplicative. 
Formally, we let the set of $\delta$ over which the minimum is taken in \eqref{eq:weight-control-original} be $\{p_0/\|p_0\|_1,-p_0/\|p_0\|_1\}$. Let $\nu^*_M$ denote the corresponding solution to \eqref{eq:weight-control-original} for such a restriction of $\delta$. 
For $ \delta = p_{0}/\|p_0\|_1 $ and $ \delta = -p_{0}/\|p_0\|_1 $, the set of disturbed supply-demand vectors can be parameterized as $(1+ \mu/\|p_0\|_1)p_0$ and $(1- \mu/\|p_0\|_1)p_0$, respectively. Therefore, letting $ \alpha = 1 + \mu/\|p_0\|_1 $ and $ \alpha =  \mu/\|p_0\|_1 -1 $, respectively, for these two cases, solution to (\ref{eq:weight-control-original}) can be obtained from: 
\begin{equation}
\label{eq:weight-control-param}
\begin{aligned}
&  \underset{w \in \real_{>0}^{\mc E}; \, \alpha \geq 0}{\text{max}}  &&  \alpha \\
& \text{subject to}  &&  
  \clower \leq \alpha f(w, p_{0}) \leq \cupper \\
& && \wlower \le w \le \wupper 
\end{aligned}
\end{equation} 
and a counterpart of \eqref{eq:weight-control-param} where $p_0$ is replaced with $-p_0$ as follows. 
Let $ \alpha^{*}_{+} $ denote the optimal solution to \eqref{eq:weight-control-param}, and let $ \alpha^{*}_{-} $ denote the optimal solution to the counterpart of \eqref{eq:weight-control-param} where $p_0$ is replaced with $-p_0$. Then, $\nu_M^*$ can be written as:
\begin{equation}
\label{eq:nu-alpha-relationship}
\param_M^* = \|p_0\|_1 \min \{ \alpha_+^* -1, \alpha_-^*+1\}.
\end{equation} 
The assumed feasibility of the pre-disturbance state of the network (cf. Assumption~\ref{ass:initial-flow-feasible}) implies that $ \alpha_+^* \geq 1$, and hence \eqref{eq:nu-alpha-relationship} is well-defined.
\begin{remark}
When the flow capacities are symmetrical, \ie, $|c^{l}| = |c^{u}|$, we have $
\alpha_{+}^{*} = \alpha_{-}^{*}$\, , and (\ref{eq:nu-alpha-relationship}) is reduced
to $ \param^* =\|p_0\|_1(\nu_+^* -1) $. For the general case of asymmetrical flow
capacities, $ \alpha_{+}^{*} \neq  \alpha_{-}^{*} $. Small disturbances in the $ -p_{0} $ direction decrease the supply and demand and hence the link flows, and are therefore favorable. However, if \(
\alpha_{-}^{*}< \alpha_{+}^{*}-2 \), then \eqref{eq:nu-alpha-relationship} implies that the margin of robustness under disturbances in the $- p_0$ direction is less than that under disturbances in the $+p_0$ direction.
\end{remark} 

We now present a gradient descent algorithm as a solution methodology for \eqref{eq:weight-control-param} which is nonconvex in general. The descent direction depends on flow-weight Jacobian, which we discuss next. In particular, we provide an exact expression for the flow-weight Jacobian which could be of independent interest.

\subsection{The Flow-weight Jacobian}
\label{sec:jacobian-exact}
Let $J(w)=\left[\frac{\partial f(w)}{\partial w} \right] \in \real^{\mc E \times \mc E}$ be the flow-weight Jacobian for the flow function $f(w)$ in \eqref{eq:sol}. We provide an explicit expression for $J(w)$ in the next result, whose proof depends on \cite[Theorem 4.3]{golub1973differentiation}.  For the sake of completeness, we reproduce this result from \cite{golub1973differentiation} and also provide a concise proof in Appendix \ref{sec:laplacian-derivative}. 
%

\begin{proposition}
\label{prop:Jacobian}
For a network with directed multigraph $\mc G =(\mc V, \mc E)$ with node-link incidence matrix $A$, link weights $w \in \real_{>0}^{\mc E}$, and supply-demand vector $p \in \real^{\mc V}$, the flow-weight Jacobian is given by:
%
\begin{equation}
\label{eq:jocobi}
J(w) =  (I - W A^T L^\dagger A) \diag(A^TL^\dagger p)
\end{equation} 
\end{proposition}
\begin{proof}
The Laplacian $ L(w) = A W A^T $ is Fr\'echet differentiable~\cite{behmardi2008introduction} with respect to $w_i$ for all $i \in \mc E$. 
Indeed, the corresponding derivative is given by 
\begin{equation}
\label{eq:lap-der}
\frac{\partial L}{\partial w_i} =\frac{\partial (AWA^T)}{\partial w_i}  =  a_i \trans{a}_i
\end{equation}
where $ a_i $ is the $i$-th column of matrix $ A$.
Since $L(w)$ is a Laplacian, it has a constant rank $=|\mc V|-1$ for all $w \in \real_{>0}^{\mc E}$. 
Therefore, Theorem~\ref{th:derivative_pseudoinverse} in the Appendix implies that the derivative of $L^\dagger $ is given by:
\begin{equation}
\label{eq:L-der-wi}
\frac{\partial L^\dagger}{\partial w_i}  = - L^\dagger \frac{\partial L}{\partial w_i} L^\dagger + L^\dagger\trans{{L^\dagger}}  \frac{\partial L^T}{\partial w_i}   (I - L L^\dagger) + (I - L^\dagger L) \frac{\partial L^T}{\partial w_i}  \trans{{L^\dagger}} L^\dagger
\end{equation}
In order to simplify \eqref{eq:L-der-wi}, using singular value decomposition, one can write $L L^{\dagger}=L^{\dagger} L = U U^T$, where $U$ is a $n \times (n-1)$ orthogonal matrix, whose columns are all orthogonal to $\onebf$, where $n=|\mc V|$. Therefore, $I-LL^{\dagger}$ and $I-L^{\dagger} L$ are both projection matrices onto $\onebf$. That is, $I-L L^{\dagger}=\onebf_{n \times n}/n = I-L^{\dagger} L$, where $\onebf_{n \times n}$ is a matrix all of whose entries are one. Therefore, using \eqref{eq:lap-der}, and noting that $ \trans{a_i}\ones =0 $,
\begin{equation}
\label{eq:lap-svd}
 \frac{\partial \trans{L}}{\partial w_i} (I - L L^\dagger)  = a_i \trans{a_i} \frac{\ones_{n\times n}}{n} = 0 = (I - L^\dagger L) \frac{\partial \trans{L}}{\partial w_i}
 \end{equation}
 Substituting \eqref{eq:lap-der} and \eqref{eq:lap-svd} in \eqref{eq:L-der-wi}, we get that 
\begin{equation}
\label{eq:derivative-pinv-laplacian}
\frac{\partial L^\dagger}{\partial w_i}  = - L^\dagger \frac{\partial L}{\partial w_i} L^\dagger = - L^\dagger a_i \trans{a_i} L^\dagger 
\end{equation}

Therefore, the $i$-th column of the Jacobian is: 
\begin{align}
J_i(w) &= \frac{\partial f(w) }{\partial w_i}   =  \frac{\partial W }{\partial w_i} A^T L^\dagger p + W A^T \frac{\partial L^\dagger }{\partial w_i}  p \nonumber \\
\label{eq:Jacobian-column-i}
 & = \trans{a_i} L^\dagger p e_i  -  W A^T  L^\dagger a_i \trans{a_i} L^\dagger p 
\end{align}
where $ e_i $ is the vector whose $i$-th component is equal to one, and all other entries are zero. 
When written in matrix form, \eqref{eq:Jacobian-column-i} gives \eqref{eq:jocobi}.
\end{proof}

\begin{remark} 
\leavevmode
\label{rem:Jacobian-interpretation}
\begin{enumerate}[label = (\alph*)]
\item The expression for the $i$-th column of Jacobian, as given in
  \eqref{eq:Jacobian-column-i}, has the following useful interpretation.
  Substituting $\trans{a_i} L^\dagger p = f_i/w_i$ in
  \eqref{eq:Jacobian-column-i}, we get that
\begin{equation}
\label{eq:Jacobian_column}
J_i(w)
	 = \frac{f_i}{w_i} e_i -   W A^T L^\dagger \frac{f_i}{w_i} a_i  
\end{equation}
%
Recall that the entries of the column $J_i(w)$ give the sensitivities of flows on various links with respect to change in weight on link $i$. The first term on the right hand side of \eqref{eq:Jacobian_column} is non-zero only when computing sensitivity of flow on link $i$ with respect to changes in $w_i$, and hence is local in nature. The non-locality in the sensitivity comes from the second term, which is equal to the flow distribution in the network corresponding to power injection of magnitude $f_i/w_i$ at the tail node $\sigma(i)$, and power withdrawal of the same magnitude from the head node  $\tau(i)$.
%
\item Using \eqref{eq:Jacobian_column}, one can show that 
\begin{equation*}
\begin{split}
J w = \sum_{i \in \mc E} w_i J_i(w) = \sum_{i \in \mc E}  f_i(e_i - W A^T  L^\dagger a_i) & = f - W A^T  L^\dagger A f \\ & = f - W A^T L^\dagger p  = 0
\end{split}
\end{equation*}
where the fourth and fifth equalities follow from \eqref{model} and \eqref{eq:sol} respectively. Since $i$-th row of $ J $ is the gradient of $f_i(w)$, this implies that the gradient of $f_i(w)$, $i \in \mc E$, is orthogonal to the radial direction $ w $. In other words, the link flows are invariant under uniform scaling of the link weights. 

\end{enumerate}
\end{remark}

Computing sensitivity of link flows with respect to link weights, via \eqref{eq:jocobi}, requires considerable computation, especially for large networks. However, some entries of the Jacobian in \eqref{eq:jocobi} exhibit sign-definiteness, as stated in the next result. 
\begin{proposition}
  \label{lem:sign-of-flow-change}
  For a network with directed multigraph $\mc G=(\mc V, \mc  E)$, weights $w \in
  \real_{>0}^{\mc E}$, and supply-demand vector $p \in \real^{\mc V}$, the flow-weight Jacobian in
    \eqref{eq:jocobi} satisfies the following for all $i \in \mc E$: $\sign(J_{ki}(w)) \in \sign(f_i) \cup \{0\}$ for all $ k \in \{i\} \cup \mathcal{E}_{\sigma(i)}^{-}\cup \mathcal{E}_{\tau(i)}^{+} $ and 
$\sign(J_{ki}(w)) \in -\sign(f_i) \cup \{0\}$ for all $ k \in \{ \mathcal{E}_{\sigma(i)}^{+} \cup \mathcal{E}_{\tau(i)}^{-} \} \setminus \{i\}$.
\end{proposition}
\begin{proof}
We provide proof for the case when $f_i>0$; the case when $f_i \leq 0$ follows along similar lines. (\ref{eq:Jacobian_column}) implies that  
\begin{equation}
\label{eq:Jacobian-rewriting}
  J_{ii} w_{i} = f_{i} -f_{i} w_{i} \trans{a}_{i} L^{\dagger} a_{i}, \qquad 
  J_{ki} w_{i} = -f_{i} w_{k} \trans{a}_{k} L^{\dagger} a_{i} 
\end{equation}
for all $k \neq i$ characterized in the lemma.
\begin{figure}[ht]
  \centering
\begin{tikzpicture}[->,>=stealth',shorten >=1pt,auto,node distance=0.5cm and 1.8cm,
                   thick,main node/.style={circle,draw,font=\sffamily\bfseries,
                     fill=blue!40}, external/.style = {coordinate}]
 \node[main node] (1) {1};
 \node[main node] (2) [above right =of 1] {2};
 \node[main node] (3) [below right =of 1] {3};
 \node[main node] (4) [below right =of 2] {4};
 \node[external] (input)[below  = 0.6cm of 3] {};
 \node[external] (output)[above  = 0.6cm of 2] {};

 \path[every node/.style={font=\sffamily}]
   (1) edge node[above]{$k_{1}$} node[below, sloped, red]{\( \rightarrow \)} (2)
   (1) edge node[below]{$k_{2}$} node[above, sloped, red]{\( \leftarrow \)} (3)
   (3) edge node{\( i \)} node[right, red] {\( \uparrow \)} (2)
   (2) edge node[above]{\( k_{3} \)} node[sloped, below, red]{\( \leftarrow \)} (4)
   (3) edge node[below]{\( k_{4} \)} node[sloped, above, red]{\( \rightarrow \)} (4)
   (input)  edge[red, thick] node[left, black] {\( f_{i} \)} (3)
   (2) edge[red, thick] node[right, black]{\( f_{i} \)} (output);
\end{tikzpicture}
  \caption{Illustration of signs of \( \partial f/\partial w_{i} \): red arrows alongside each link denote the flow direction on the corresponding link under the supply-demand vector $f_i a_i$; for every link $\neq i$, if the red arrow alongside a link aligns with the link direction, then the corresponding component of  \( \partial f/\partial w_{i} \) is negative, and positive otherwise. Correspondingly,  \( \partial f_{k_{1}}/\partial w_{i} > 0 \),  \( \partial f_{k_{4}}/\partial w_{i} < 0 \),  \( \partial f_{k_{2}}/\partial w_{i} > 0 \),  \( \partial f_{k_{3}}/\partial w_{i} > 0 \). We always have \( \partial f_{i}/\partial w_{i} > 0 \).   \label{fig:Jacobian-signs} }
\end{figure}
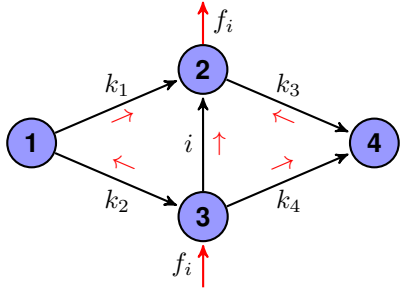

 Recalling Remark~\ref{rem:Jacobian-interpretation} (a) that $f_{i} w_{k}
 \trans{a}_{k} L^{\dagger} a_{i}$ can be interpreted as the flow on link
 $k$ under supply-demand vector $f_i a_i$, for which $\sigma(i)$ (node 3 in
 Fig. \ref{fig:Jacobian-signs}) and $\tau(i)$ (node 2 in
 Fig. \ref{fig:Jacobian-signs}) are the only supply and demand nodes. It is easy
 to see that when a network has only one supply node and only one load node,
 then the phase angles at the supply and the load nodes are largest and
 smallest, respectively, among phase angles associated with all the nodes. This
 implies that for all $k \in \mc E_{\sigma(i)}^{-}$ (link \( k_{2}
 \) in  Fig. \ref{fig:Jacobian-signs}) and $ k\in \mathcal{E}_{\tau(i)}^{+} $ (link \( k_{3} \) in
 Fig. \ref{fig:Jacobian-signs}), \ie, links incoming to $\sigma(i)$ and
 outgoing from $\tau(i)$, the phase angle difference along the direction of such
 links, \ie, $ \trans{a}_{k} L^{\dagger} f_{i} a_{i}$, and hence $w_k
 \trans{a}_{k} L_{\mc G}^{\dagger} f_{i} a_{i}$ is non-positive, and therefore
 \eqref{eq:Jacobian-rewriting} implies that $J_{ki}w_i$, and hence $J_{ki}$, for
 such links is non-negative. Similarly, one can show that or all $k \in \mc E_{\sigma(i)}^{+} \setminus \{i\}$
  (link \( k_{4} \) in Fig. \ref{fig:Jacobian-signs}) and $ k\in \mathcal{E}_{\tau(i)}^{-} \setminus \{i\}$ (link \( k_{1} \) in Fig. \ref{fig:Jacobian-signs}), \ie, links outgoing from $\sigma(i)$ and incoming to $\tau(i)$, $J_{ki}$ is non-positive.

Recalling again that $f_{i} w_{j} \trans{a}_{j} L_{\mc G}^{\dagger} a_{i}$ is the flow on link $j$ under supply-demand vector $f_i a_i$, flow conservation at node $\sigma(i)$ can be written as 
$$
f_i + \sum_{j \in \mathcal{E}_{\sigma(i)}^{-} } f_{i} w_{j} \trans{a}_{j} L_{\mc G}^{\dagger} a_{i} = f_{i} w_{i} \trans{a}_{i} L_{\mc G}^{\dagger} a_{i} + \sum_{j \in \mathcal{E}_{\sigma(i)}^{+}\setminus\{i\}} f_{i} w_{j} \trans{a}_{j} L_{\mc G}^{\dagger} a_{i} 
$$
The discussion in the second paragraph of this proof implies that terms inside the summation in left and right hand side are non-positive and non-negative respectively, implying that $f_i - f_{i} w_{i} \trans{a}_{i} L_{\mc G}^{\dagger} a_{i}$ is non-negative. Therefore, \eqref{eq:Jacobian-rewriting} implies that $J_{ii} \geq 0$.
\end{proof}

\begin{remark}
\leavevmode
\begin{enumerate}[label = (\alph*)]
\item For a given choice of directionality of links in $\mc E$, Proposition~\ref{lem:sign-of-flow-change} implies that, if $f_i \geq 0$, then an infinitesimal increase in the weight of link $i$ will not decrease flow on link $i$ or on links incoming to the tail node of $i$ or outgoing from the head node of $i$, and it will not increase flow on links outgoing from tail node of $i$ or incoming to head node of $i$. The conclusions are opposite when $f_i \leq 0$. We emphasize that these changes in flows are not in terms of absolute values, e.g., a change of $f_j$ from $-3$ to $-2$ is an increase in $f_j$. 
%
  \item Proposition~\ref{lem:sign-of-flow-change} can be interpreted as generalization of existing results, e.g., see  \cite{Lai.Low:Allerton13}, that study the effect of removal of a link on flows in neighboring links. We elaborate on this point further in Section~\ref{subsec:jacobian-multigraph-perspective}.
\item From a weight control perspective, Proposition~\ref{lem:sign-of-flow-change} implies that the direction of change in link flows on neighboring links due to change in weight in link $i$ can be computed in a completely decentralized fashion, which maybe be useful to develop a decentralized weight control heuristic. However, partly because this decentralized computation can be done only for immediately neighboring links, and partly because directions of change in link flow on a given link due to weight changes of other links are not necessarily aligned, such a heuristic can not be expected to be optimal in general. 
\end{enumerate}
\end{remark}

\subsection{A Multigraph Perspective for the Flow-weight Jacobian}
\label{subsec:jacobian-multigraph-perspective}
Definition~\ref{def:reduced-simple-graph} in Appendix~\ref{subsec:lap} describes the notion of a reduced simple digraph corresponding to a multigraph, which allows us to see how distinct networks can have the same Laplacian. We now introduce a one-link extension of a given (possibly multi-) graph to facilitate alternate derivation of the expression of the Jacobian $J(w)$ in Proposition~\ref{prop:Jacobian}. Such a construct will also help us to generalize the notion of the flow-weight Jacobian by allowing to study the change in link flows due to non-infinitesimal changes in weights, \eg, caused by addition or removal of links. 

Given a graph $\mc G=(\mc V, \mc E)$ with link weights $w \in \real_{>0}^{\mc E}$, its one-link extension corresponding to link $i \in \mc E$ and $\triangle w_i \in [0,w_i]$, denoted as $\mathcal{G}^{\mathrm{ex}}(w_i - \triangle w_i,\triangle w_i)$, is 
obtained from $\mc G$ by replacing link $i$ with two parallel links with weights $w_i - \triangle w_i$ and $\triangle w_i$ (see Figure~\ref{fig:multigraph-flow-change} for an illustration). In order to emphasize the dependence on link $i$ and its weight $w_i$, we denote the original graph as $\mc G(w_i)$. Clearly, $\mc G(w_i-\triangle w_i)=\mathcal{G}^{\mathrm{ex}}(w_i-\triangle w_i,0)$, and $\mc G(w_i)$ and $\mathcal{G}^{\mathrm{ex}}(w_i-\triangle w_i,0)$ both have the same reduced simple graph (cf. Definition~\ref{def:reduced-simple-graph}). Therefore, using Lemma~\ref{lem:base-graph-laplacian} in Appendix~\ref{subsec:lap}, we have that
\begin{equation}
\label{eq:simple-expanded-equivalence}
L_{\mc G(w_i-\triangle w_i)}=L_{\mathcal{G}^{\mathrm{ex}}(w_i-\triangle w_i,0)}, \qquad L_{\mc G(w_i)}=L_{\mathcal{G}^{\mathrm{ex}}(w_i-\triangle w_i, \triangle w_i)}
\end{equation}

\begin{figure}[htbp]
\centering
\begin{subfigure}[t]{0.49\textwidth}
        \centering
\begin{tikzpicture}[->,>=stealth',shorten >=1pt,auto,node distance=0.6cm and 1.2cm, thick,main node/.style={circle,draw,font=\sffamily\bfseries, fill=blue!40}]
 \node[main node] (1) {1};
 \node[main node] (2) [above right =of 1] {2};
 \node[main node] (3) [below right =of 1] {3};
 \node[main node] (4) [right = 2cm of 2] {4};
 \node[main node] (5) [right = 2cm  of 3] {5};
 \node[main node] (6) [below right =of 4] {6};

 \path[every node/.style={font=\sffamily\small}]
   (1) edge (2)
   (1) edge (3)
   (2) edge (3)
   (2) edge (4)
   (2) edge [bend left = 20] (5)
   (2) edge [bend right=20] node[sloped, above]{\( w_{i} \)} (5)
   (3) edge (5)
   (5) edge (4)
   (4) edge (6)
   (5) edge (6);
\end{tikzpicture}
\caption{}
    \end{subfigure}
\begin{subfigure}[t]{0.49\textwidth}
        \centering
        \begin{tikzpicture}[->,>=stealth',shorten >=1pt,auto,node distance=0.6cm 
          and 1.2cm, thick,main node/.style={circle,draw,font=\sffamily\bfseries, fill=blue!40}]
 \node[main node] (1) {1};
 \node[main node] (2) [above right =of 1] {2};
 \node[main node] (3) [below right =of 1] {3};
 \node[main node] (4) [right = 2cm of 2] {4};
 \node[main node] (5) [right = 2cm  of 3] {5};
 \node[main node] (6) [below right =of 4] {6}; 

 \path[every node/.style={font=\sffamily\small}] 
   (1) edge (2)
   (1) edge (3)
   (2) edge (3)
   (2) edge (4)
   (2) edge [bend left=35] (5)
   (3) edge (5)
   (5) edge (4)
   (4) edge (6)
   (5) edge (6);
       
\path[every node/.style={font=\sffamily\small}, dashed]
   (2) edge [bend right=35] node[sloped, above]{\( w_{i} - \Delta w_{i}\)} (5)
    (2) edge node[sloped, above]{\( \Delta w_{i} \)} (5); 
\end{tikzpicture}
\caption{}
    \end{subfigure}
\caption{Illustration of one-link extension of a graph. (a) A graph $ \mathcal{G}(w_{i}) $ with weight $ w_{i} $ on lower link $ (2, 5) $. (b) The one-link extension $ \mathcal{G}^{\mathrm{ex}}(w_{i}-\Delta w_{i}, \Delta w_{i}) $ of $ \mathcal{G}(w_{i}) $ corresponding to the lower link $ (2, 5) $. }
\label{fig:multigraph-flow-change} 
\end{figure}
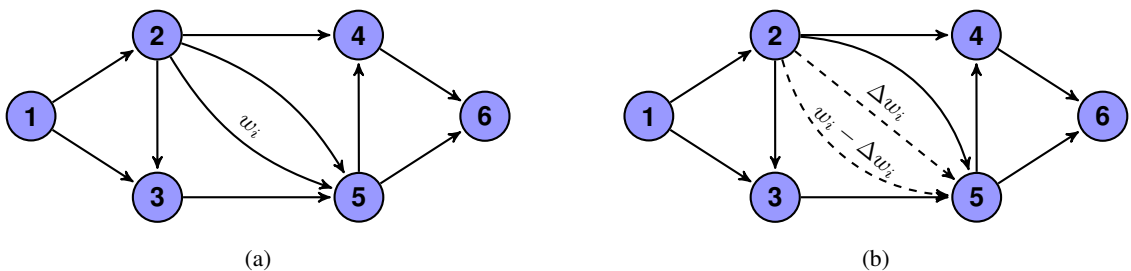

Since $\mathcal{G}^{\mathrm{ex}}(w_i-\triangle w_i,0)$ is obtained from $\mathcal{G}^{\mathrm{ex}}(w_i-\triangle w_i,\triangle w_i)$ by removing the link with weight $\triangle w_i$,  \cite[Lemma 2]{Ba.Savla.CDC16-ext} implies that 
\begin{equation}
\label{eq:pinv-laplacian}
  L_{\mathcal{G}^{\mathrm{ex}}(w_i-\triangle w_i,0)}^{\dagger} - L_{\mathcal{G}^{\mathrm{ex}}(w_i-\triangle w_i,\triangle w_i)}^{\dagger} = \frac{\triangle w_{i}}{1-\theta_{i}} L_{\mathcal{G}^{\mathrm{ex}}(w_i-\triangle w_i,\triangle w_i)}^{\dagger}a_{i}\trans{a_{i}}L_{\mathcal{G}^{\mathrm{ex}}(w_i-\triangle w_i,\triangle w_i)}^{\dagger}
\end{equation}
where \( \theta_{i} := \triangle w_{i} \, \trans{a_{i}}L_{\mathcal{G}^{\mathrm{ex}}(w_i-\triangle w_i,\triangle w_i)}^{\dagger}a_{i} \). Combining \eqref{eq:simple-expanded-equivalence} and \eqref{eq:pinv-laplacian}, and rearranging, we get that 
$$
\frac{1}{\triangle w_i} \left(L_{\mc G(w_i)}^{\dagger} - L_{\mc G(w_i-\triangle w_i)}^{\dagger}\right) = - \frac{1}{1-\theta_{i}} L_{\mc G(w_i)}^{\dagger}a_{i}\trans{a_{i}}L_{\mc G(w_i)}^{\dagger}
$$


Therefore, noting that $\theta_i \to 0^+$ as $\triangle w_i \to 0^+$, we have 
%
\begin{equation*}
  \lim_{\triangle w_{i}\to 0^+} \frac{1}{\triangle w_i} \left(L_{\mc G(w_i)}^{\dagger} - L_{\mc G(w_i-\triangle w_i)}^{\dagger}\right) = - L_{\mc G(w_i)}^{\dagger}a_{i}\trans{a_{i}}L_{\mc G(w_i)}^{\dagger}
\end{equation*}
One can similarly show that the right hand side derivative is the same, thereby giving \eqref{eq:derivative-pinv-laplacian}.

While the above discussion illustrates the utility of a multigraph perspective to find out sensitivity of link flows with respect to link weights, the same perspective can also be used to compute change in link flows due to non-infinitesimal change in the link weights. First note that the change in link flows due to decrease in link weights by $\triangle w_i> 0$ can be computed using \eqref{eq:Jacobian-column-i} as $\triangle f = \int^{w_i - \triangle w_i}_{w_i } J_i(\kappa) \, d\kappa$. However, due to the dependence of $J(w)$ on the pseudo-inverse of $L$, it is not possible to get an explicit expression for this integral in general. Alternately, the multigraph perspective implies that $\triangle f$ is equal to the difference in link flows between $\mathcal{G}^{\mathrm{ex}}(w_i - \triangle w_i, \triangle w_i)$ and $\mathcal{G}^{\mathrm{ex}}(w_i - \triangle w_i,0)$. The latter corresponds to change in link flows due to removal of the link with weight $\triangle w_i$ in $\mathcal{G}^{\mathrm{ex}}(w_i - \triangle w_i,  \triangle w_i)$. Therefore, \eqref{eq:sol} and \eqref{eq:pinv-laplacian} imply that 
\begin{equation*}
  \triangle f  = \frac{\triangle w_{i}\trans{a_{i}}L_{\mc G(w_i)}^{\dagger}p}{1-\theta_{i}} W A_{\mc G(w_i)}L_{\mc G(w_i)}^{\dagger}a_{i}
\end{equation*}
The above equation can also be obtained by using the sensitivity factor of changes in phase angles to flows on removed links in \cite{wollenberg1996power}. 

 

\subsection{A Projected Sub-gradient Algorithm}
\label{subsec:projected-subgradient}
We now utilize the flow-weight Jacobian derived in Section~\ref{sec:jacobian-exact} to design a projected sub-gradient algorithm for solving \eqref{eq:weight-control-param}. 
In order to re-write 
\eqref{eq:weight-control-param} and its counter part for $-p_0$ succinctly, we
consider the following notion of \emph{effective} line capacity. Given flow \(
f \), for all $  i \in \mc E$:  
\begin{equation}
\label{eq:c-def}
\begin{aligned}
c_i:= \left\{ \begin{array}{ll} \phantom{-}\cupper_i   \quad  & \text{if } f_{i} \ge 0   \\    \phantom{-}\clower_i & \text{if } f_{i} < 0   \end{array}\right. \text{ for } \delta = \frac{p_{0}}{\|p_{0}\|_{1}} ; \quad 
c_i:= \left\{ \begin{array}{ll} -\clower_i   \quad \, & \text{if } f_{i} \ge 0   \\ -\cupper_i  &  \text{if } f_{i} < 0   \end{array} \right. \text{ for } \delta = -\frac{p_{0}}{\|p_{0}\|_{1}}.
\end{aligned}
\end{equation}
\eqref{eq:weight-control-param} can then be equivalently written as:
\begin{equation}
\label{eq:opt-grad-descent}
\begin{aligned}
& \underset{ w \in \real_{>0}^{\mc E}}{\text{minimize}}  &&  \max_{i \in \mc E} \frac{f_i(w)}{c_i}  \\
& \text{subject to}  &&   
 \wlower \le w \le \wupper 
\end{aligned}
\end{equation}
Let $\ell(w) := \argmax_{i \in \mc E} { {f_i(w)}/{c_i} }$ be the links corresponding to the maximum value of $f_i(w)/c_i$. 
A projected sub-gradient method, along the lines of \cite[Section 2.1.2]{bertsekas2015convex}, for solving \eqref{eq:opt-grad-descent} is then given by:
\begin{equation}
\label{eq:w-update}
 w(t+1) = \argmin_{w^{l} \le w\le w^{u}} \left( \max_{i\in \ell(w(t))}
   \frac{J_{i}(w(t)) \cdot (w-w(t))}{c_{i}}  + \frac{1}{2\eta_{t}}
    \trans{(w-w(t))}(w-w(t)) \right)
\end{equation}
where $\eta_{t} > 0$ is the step-size. (\ref{eq:w-update}) gives
an unweighted version of projected gradient iteration -- it can be generalized by incorporating an appropriate positive definite weighting matrix into the regularization term, e.g., see 
\cite[Section 2.1.2]{bertsekas2015convex}. 

Convergence analysis for projected sub-gradient algorithms for convex optimization problems is a well-studied topic, e.g., see \cite{Bertsekas:99}. However, extensions to non-convex problems, as is the case with \eqref{eq:weight-control-param}, is still an ongoing work. We report supporting numerical evidence for the convergence of the proposed algorithm in \eqref{eq:w-update} in Section~\ref{sec:simulations}, and postpone formal analysis to future work.

%% file: multi-level-formulation.tex
\section{The Nongenerative Disturbance Case: A Multilevel Programming Approach}
\label{sec:bi-level}
In Section~\ref{sec:multiplicative}, we presented results for
  multiplicative disturbances. In this section, we consider a more general setting
  of \emph{nongenerative disturbances}, defined next.
  \begin{definition}
    \label{def:nongenerative-disturbance}
For a network with directed multigraph \( \mathcal{G}= (\mathcal{V},
\mathcal{E}) \) and supply-demand vector \( p\in \real^{\mathcal{V}} \), a balanced
disturbance \( \triangle \in \real^{\mc V}\) is called nongenerative with respect to $p$ if $p_v = 0$ implies $\triangle_v=0$ for all $v \in \mc V$. 
  \end{definition}
 Let the set of all nongenerative disturbances with respect to $ p\in \real^{\mathcal{V}} $ be denoted as $\nongenset(p) \subset \real^{ \mathcal{V}} $. The projected sub-gradient algorithm formulated in
Section~\ref{subsec:projected-subgradient}, besides being restricted to multiplicative disturbances, can not guarantee an optimal solution
because of the non-convexity of \eqref{eq:weight-control-param}. 
 In addition to expanding the set of admissible disturbances, in this section, we also develop a solution methodology with favorable computational properties.   
The
computational complexity of exhaustive search methods for solving \eqref{eq:weight-control-original}-\eqref{eq:weight-control-nu} and 
\eqref{eq:weight-control-param} grows exponential in the number of links. In
this section, we introduce novel notions of network reduction, which when
applied to reducible networks (cf. Definition \ref{def:reducible-net}) gives a multi-level
formulation of \eqref{eq:weight-control-original}-\eqref{eq:weight-control-nu}. While the optimization problem
at each level is still non-convex, the resulting decomposition of the original
monolithic problem in \eqref{eq:weight-control-original}-\eqref{eq:weight-control-nu} yields computational
savings when using exhaustive search for finding solution. 

We start by identifying a sufficient condition under which a certain class of optimization problems admit an equivalent bilevel formulation. 

\subsection{An Equivalent Bilevel Formulation}
Given continuous  maps $\map{q_i}{\real^n}{\real}$, $i \in \until{m}$, a generic feasibility problem can be written as:
\begin{equation}
\label{eq:generic-feasibility}
\text{Find } x \in D \subset \real^n \text{ s.t. } q_i(x) \leq 0, \qquad \forall i \in \until{m}
\end{equation}
where $D$ is the domain of \( n \)-dimensional variable \( x \). 
%

We are interested in $(q_1,\ldots,q_m)$ for which there exist partitions\footnote{That is,  \( \mc I_{1} \cup \mc I_{2} =  \{1, 2, \ldots, n\} \), \( \mc I_{1}\cap
\mc I_{2}=\emptyset \), \( \mc J_{1}\cup \mc J_{2} = \{1, 2, \ldots, m\} \) and \( \mc J_{1}
\cap \mc J_{2} = \emptyset. \)} $\{\mc I_1, \mc I_2\}$ and $\{\mc J_1,\mc J_2\}$
of $\until{n}$ and $\until{m}$, respectively, and continuous maps $\map{h_1}{D^{\mc I_1} \times \real}{\real}$ and $\map{h_2}{D^{\mc I_2}}{\real}$, such that \eqref{eq:generic-feasibility} is equivalent to finding $(x_{\mc I_1},x_{\mc I_2},y_1,y_2) \in D^{\mc I_1} \times D^{\mc I_2} \times \real \times \real$ satisfying:
\begin{equation}
\label{eq:feasibility-expanded-version}
\begin{split}
q_i(x_{\mc I_1},y_2) & \leq 0, \quad \forall i \in \mc J_1 \\
q_i(x_{\mc I_2},y_1) & \leq 0, \quad \forall i \in \mc J_2 \\
y_1 & = h_1(x_{\mc I_1},y_2) \\
y_2 & = h_2(x_{\mc I_2})
\end{split}
\end{equation}

In \eqref{eq:feasibility-expanded-version}, $x_{\mc I_1}$, $x_{\mc I_2}$
and $D^{\mathcal{I}_{1}}$, \(
 D^{\mathcal{I}_{2}} \)
denote subvectors  and domains of $x$ corresponding to indices in $\mc I_1$ and $\mc I_2$,
respectively. The next result gives an equivalent bilevel formulation of \eqref{eq:feasibility-expanded-version}.
\begin{proposition}
\label{thm:equiv-transf-feasibility}
\leavevmode
Let \( q_{i} \), $i \in \until{m}$, \( h_{1} \) and \( h_{2} \) be continuous
functions. Then, the following are true:
\begin{enumerate}[label=(\alph*)]
\item there exists a $(x_{\mc I_1},x_{\mc I_2},y_1,y_2) \in D^{\mc
  I_1} \times D^{\mc I_2} \times \real \times \real $ satisfying
  \eqref{eq:feasibility-expanded-version} if and only if there exists a $(x_{\mc
    I_1},y_2) \in D^{\mathcal{I}_{1}} \times D_{2}
  $ satisfying the following:
%
\begin{equation}
\label{eq:eqvl-problems}
\begin{aligned}
  q_i(x_{\mc I_1},y_{2}) & \leq 0 \quad \forall \, i \in {\mc J_1}  &
\quad \text{where}\quad & & G(y_2) :=  \{ z \in \real \, | \, & 
q_{i}(x_{\mc I_{2}}, z) \le 0 \quad \forall i \in \mc J_2\\
   h_1(x_{\mc I_1},y_2) & \in G(y_{2})  & & & &
   \text{for some }x_{\mathcal{I}_{2}} \in D^{\mathcal{I}_{2}} \text{ satisfying } h_{2}(x_{\mc I_{2}}) =
  y_2 \}
  \end{aligned}
\end{equation} 
and \( D_{2}= \range(h_{2}) \) is the domain of \( y_{2} \). Moreover, for every $y_2 \in D_2$, the set \( G(y_{2})
    \) is closed.
\item  the set \( G(y_{2}) \) is convex for all \( y_{2} \in D_2\) if 
  \begin{itemize}
  \item  for all $x_{\mc I_2} \in D^{\mc I_2}$ and $i \in \mc J_2$, $q_i(x_{\mc I_2},z)$ is quasiconvex with respect to $z$; and
 \item there exists a \( z_{0} \in \real \) such that \(
   q_{i}(x_{\mathcal{I}_{2}}, z_{0}) \le 0 \) for all \( x_{\mathcal{I}_{2}}\in
   D^{\mathcal{I}_{2}} \) and $i \in \mc J_2$. 
  \end{itemize}
\end{enumerate} 
\end{proposition}
\begin{proof} 
\begin{enumerate}[label=(\alph*)]
\item We refer to the feasibility problem on the left side of \eqref{eq:eqvl-problems} as \eqref{eq:eqvl-problems}-L. 

Consider a \( (\tilde{x}_{\mc I_{1}}, \tilde{x}_{\mc
  I_{2}},\tilde{y}_1,\tilde{y}_2)\) which satisfies
\eqref{eq:feasibility-expanded-version}. This implies that the first equation in
(\ref{eq:eqvl-problems})-L is satisfied by $(\tilde{x}_{\mc I_1},\tilde{y}_2)$,
and that \( z=\tilde{y}_{1} = h_{1}(\tilde{x}_{\mathcal{I}_{1}}, \tilde{y}_{2}) \in
G(\tilde{y}_{2}) \) with \( x_{\mathcal{I}_{2}} = \tilde{x}_{\mathcal{I}_{2}}
\). 


Now consider a $(\hat{x}_{\mc I_1},\hat{y}_2)$ which satisfies (\ref{eq:eqvl-problems})-L. Therefore, $(\hat{x}_{\mc I_1},\hat{y}_2)$ readily satisfies the first inequality in \eqref{eq:feasibility-expanded-version}.
Let \(
\hat{y}_{1}:=  h_{1}(\hat{x}_{\mc I_{1}}, \hat{y}_{2}) \), then \( \hat{y}_{1}
\in G(\hat{y}_{2}) \). Therefore, \( G(\hat{y}_{2}) \) is not empty and there
exists at least one \( \hat{x}_{\mathcal{I}_{2}} \) such that \(
h_{2}(\hat{x}_{\mathcal{I}_{2}}) = \hat{y}_{2} \) and \(
q_{i}(\hat{x}_{\mathcal{I}_{2}}, \hat{y}_{1})\le 0 \) for all \( i\in
\mathcal{J}_{2} \). That is to say, 
%
%
\( (\hat{x}_{\mc I_{1}},
\hat{x}_{\mc I_{2}},\hat{y}_1,\hat{y}_2)\) satisfies \eqref{eq:feasibility-expanded-version}.

Now we show that \( G(y_{2}) \) is a closed set for every $y_2 \in D_2$. 
Pick an arbitrary convergent sequence \( \{z_{r}\} \) in the set \(
G(y_{2}) \). It is sufficient to prove that \(
z^{*} = \lim_{r\to +\infty} z_{r} \in G(y_{2})  \). Suppose \( z^{*} \not \in
G(y_{2}) \), then \( \exists\, k \in \mc J_2\) s.t. \( q_{k}(x, z^{*}) > 0, \forall \, x \in
D^{\mc I_2}\) satisfying $h_2(x)=y_2$. Continuity of $q_k$ then implies that $q_{k}(x, z_r) > 0$, and hence implying $z_r \notin G(y_2)$, for all sufficiently large $r$. This leads to a contradiction. 
%

\item The second condition implies that \(z_{0} \in  G(y_{2}) \subset \real \) for all \( y_{2} \in D_{2}
  \). Since $ G(y_{2}) $ is close for  all \( y_{2} \in D_{2} \), let \( g^{l}(y_{2}) := \min G(y_{2})  \) and \( g^{u}(y_{2}) :=
  \max {G(y_{2})} \), then \( g^{l}(y_{2}) \le z_{0} \le g^{u}(y_{2}) \) for
  all \( y_{2} \in D_{2} \). Proving convexity of the set \( G(y_{2}) \) is equivalent to
  proving that \( [g^{l}(y_{2}), z_{0}]\subset G(y_{2}) \) and \( [z_{0},
  g^{u}(y_{2})]\subset G(y_{2}) \). We provide details for the first set; the proof for the 
  second set follows similarly. 

Consider a \( x^{*}_{\mathcal{I}_{2}} \in D^{\mc I_2}\) satisfying \( h_{2}(x^{*}_{\mathcal{I}_{2}}) = y_{2} \) and \(
q_{i}(x^{*}_{\mathcal{I}_{2}}, g^{l}(y_{2})) \le 0 \) for all \( i \in
\mathcal{J}_{2} \); closedness of the set $G(y_2)$ implies well-posedness of  $x^{*}_{\mathcal{I}_{2}}$. We also have \( q_{i}(x^{*}_{\mathcal{I}_{2}},z_{0}) \le 0 \) for all
\( i \in \mathcal{J}_{2} \) by assumption. Since \( q_{i}(x_{\mathcal{I}_{2}},
z) \) is quasiconvex with respect to \( z \) for all \( x_{\mathcal{I}_{2}}\in
D^{\mathcal{I}_{2}} \) and \( i\in \mathcal{J}_{2} \), 
\begin{equation*}
q_{i}(x_{\mathcal{I}_{2}}^{*}, \theta g^{l}(y_{2}) + (1-\theta)z_{0}) \le
\max\{q_{i}(x_{\mathcal{I}_{2}}^{*},  g^{l}(y_{2})),
q_{i}(x_{\mathcal{I}_{2}}^{*}, z_{0})
\} \le 0 \quad \forall \theta \in [0, 1], \forall i \in \mathcal{J}_{2}
\end{equation*}

Since \( y_{2} \) is arbitrary, \( [g^{l}(y_{2}), z_{0}]\subset G(y_{2}) \) for
all \( y_{2} \in D_{2} \). 
\end{enumerate} 
\end{proof}

  \begin{remark}
    \leavevmode
    Proposition~\ref{thm:equiv-transf-feasibility} can be extended along the following directions: 
    \begin{enumerate}
    \item The second condition in Proposition
      \ref{thm:equiv-transf-feasibility}(b) can be
      relaxed as follows: for every \( y_{2} \in D_{2} \), there exists $x^{*, l}_{\mathcal{I}_{2}}, x^{*, u}_{\mathcal{I}_{2}} \in D^{\mc I_2}$ and $z_l, z_u \in \real$ satisfying: 
  \begin{inparaenum}[(i)]
      \item $h(x^{*,s}_{\mc I_1})=y_2$ and $q_i(x^{*,s}_{\mc I_2}, g^s(y_2)) \leq 0$ for $s \in \{l,u\}$, $i \in \mc J_2$; and
	\item $q_i(x^{*,s}_{\mc I_2}, z_s) \leq 0$, for $s \in \{l,u\}$, $i \in \mc J_2$; and
	      \item $z_l \ge z_u$. 
   \end{inparaenum}
\item The set \( G(y_{2}) \) is not necessarily bounded, \ie, it is possible that \( g^{l}(y_{2}) = -\infty \) or \( g^{u}(y_{2}) = + \infty \). For example, if (\ref{eq:feasibility-expanded-version}) is feasible and \( q_{i}(x_{\mathcal{I}_{2}}, z) \) is nondecreasing (respectively, nonincreasing) with
  respect to \( z \) for all \( x_{\mathcal{I}_{2}} \in
  D^{\mathcal{I}_{2}}\) and \( i\in \mathcal{J}_{2} \), then it is straightforward to see that \( g^{l}(y_{2}) = - \infty \) (respectively, \(
  g^{u}(y_{2}) = +\infty \)). In this case, the second condition  in Proposition
      \ref{thm:equiv-transf-feasibility}(b) is trivially satisfied by \( z_{0} = - \infty \) (respectively, \(
  z_{0} = +\infty \)). 
    \end{enumerate}
  \end{remark}


Proposition~\ref{thm:equiv-transf-feasibility} can be straightforwardly extended to optimization problems as follows.
\begin{proposition}
\label{cor-eqvl-transformation}
Let $q_i(x_{\mc I_2},z)$ be quasiconvex with respect to $z$ for all
$x_{\mc I_2} \in \real^{\mc I_2}$, $i \in \mc J_2$ and \(z_{0}\in
  \real \) be such that \(
   q_{i}(x_{\mathcal{I}_{2}}, z_{0}) \le 0 \) for all \( x_{\mathcal{I}_{2}}\in
   D^{\mathcal{I}_{2}} \), $i \in \mc J_2$.
Then, for every $\map{q_0}{D^{\mc I_1} \times \real}{\real}$, 
\begin{equation}
  \label{opt:original}
  \begin{aligned}
  & \, \underset{\substack{\\[0pt]\ds x_{\mc I_{1}} \in D^{{\mc I}_1}, x_{\mc I_{2}} \in D^{{\mc I}_2}\\[5pt]\ds y_{1} \in \real, y_{2} \in \real}}{\max} &&  q_0(x_{\mc I_{1}}, y_{2})  \\
  &\text{subject to} &&  q_{i} (x_{\mc I_{1}}, y_{2}) \le 0, \quad \forall \, i \in \mc J_1 \\
    &&& q_i(x_{\mc I_2},y_1) \leq 0, \quad \forall i \in \mc J_2  \\
  &&&  y_1 = h_1(x_{\mc I_1},y_2) \\
&&& y_2 = h_2(x_{\mc I_2}) 
  \end{aligned}
\end{equation}
is equal to 
\begin{equation}
  \label{opt:bi-level}
  \begin{aligned}
 \underset{x_{\mc I_{1}} \in D^{\mc I_1}, y_{2} \in
   D_{2}}{\max}  &  q_0(x_{\mc I_{1}}, y_{2}) & & &  &    \\
  \text{subject to}\quad &  q_{i} (x_{\mc I_{1}}, y_{2}) \le 0, \quad \forall \, i
  \in \mc J_1  & \quad \text{where} \quad &  & g^{l}(y_{2}) := \min G(y_{2}) \\
    & g^{l}(y_{2})\le  h_{1}(x_{\mc I_{1}}, y_{2}) \le g^{u}(y_{2}) & & &
    g^{u}(y_{2}) := \max G(y_{2}) 
  \end{aligned}
\end{equation}
and, \( D_{2} \), \( G(y_{2}) \) are as defined in
Proposition~\ref{thm:equiv-transf-feasibility}. 
\end{proposition}

\begin{remark}
\label{rem:bilevel}
\begin{enumerate}
\item[(a)] The conditions in Proposition~\ref{thm:equiv-transf-feasibility}(a) do not guarantee that the set $G(y_2)$ is non-empty for every $y_2 \in D_2$. However, under the additional condition of the existence of $z_0$, as in Propositions~\ref{thm:equiv-transf-feasibility}(b) and \ref{cor-eqvl-transformation}, the set $G(y_2)$ is guaranteed to be non-empty for all $y_2 \in D_2$. In particular, this implies that $g^l(y_2)$ and $g^u(y_2)$ in Proposition~\ref{cor-eqvl-transformation} are well-defined. 
\item[(b)] Proposition~\ref{cor-eqvl-transformation} provides an equivalent bilevel formulation in \eqref{opt:bi-level} for a class of optimization problems described in \eqref{opt:original}. $(x_{\mc I_1},y_1)$ and $(x_{\mc I_2},z)$ are the upper and lower level variables, respectively. 
\item[(c)] When solving by exhaustive search, the bilevel formulation in \eqref{opt:bi-level} offers computational advantage over the original formulation in \eqref{opt:original} as follows. The number of non-redundant variables in \eqref{opt:original} is $|\mc I_1| + |\mc I_2|$. Therefore, the computational complexity in solving \eqref{opt:original} by exhaustive search is exponential in $|\mc I_1| + |\mc I_2|$. On the other hand, the computational complexity associated with computing $g^l(y_2)$ and $g^u(y_2)$, for every $y_2$, in the lower level problem in \eqref{opt:bi-level} is exponential in $|\mc I_2|$. Thereafter, the computational complexity of the upper level problem is exponential in $|\mc I_1|+1$. Therefore, the complexity of solving the bilevel problem is exponential in $\max\{|\mc I_1|+1, |\mc I_2| \}$, which is much less than that of \eqref{opt:original}. One can further reduce the computational complexity by extension to multilevel formulation, \eg, by recursive bilevel formulation of the lower level problem in  \eqref{opt:bi-level}. This multilevel extension will be explained in the context of the central problem \eqref{eq:weight-control-original}-\eqref{eq:weight-control-nu} of this paper in Section~\ref{subsec:multilevel}.

\end{enumerate}
\end{remark}

\subsection{A Novel Network Reduction and its Relationship to the Equivalent Bilevel Formulation}
In this subsection, we investigate conditions under which Proposition~\ref{cor-eqvl-transformation} can be applied to  
\eqref{eq:weight-control-original}-\eqref{eq:weight-control-nu}
to get an equivalent bilevel formulation.
We first transform \eqref{eq:weight-control-nu} into the form of
\eqref{opt:original}, where the link weight \( w \) will play the
  role of \( x \) and \( D^{\mathcal{E}} = [w^{l}, w^{u}] \) will be
  its domain.  
The partition and structure of constraints underlying
  \eqref{opt:original} will be made possible if the network and the nodes carrying non-zero demand and supply are relatively sparse. This condition is formalized in the following definition of \emph{reducible networks}. 


\begin{definition}
\label{def:reducible-net}
A network with directed multigraph $ \mathcal{G} = ( \mathcal{V}, \mathcal{E} )$ is called reducible about $ v_{1} \in \mc V$ and $ v_{2} \in \mc V$ under supply-demand vector $ p \in \real^{\mc V}$ if 
there exists a partition: $\mc G=\mc G_1 \cup \mc G_2$, with $\mc G_1=(\mathcal{V}_{1}, \mathcal{E}_{1})$ and $\mc G_2=(\mathcal{V}_{2}, \mathcal{E}_{2}) $, both satisfying Assumption~\ref{ass:connectivity}, and $\mathcal{V}_{1} \cup \mathcal{V}_{2} =
 \mathcal{V}$, $ \mathcal{V}_{1} \cap \mathcal{V}_{2} = \{v_1, v_2\}
 $, $\mathcal{E}_{1} \cup \mathcal{E}_{2} = \mathcal{E} $, $
 \mathcal{E}_{1} \cap \mathcal{E}_{2} = \emptyset $ and $|\mc E_{2}| \geq 2$, such that the supply-demand vector $ p $ is supported only on $\mc V_1$. $ \mathcal{G}_{2} $ is referred to as the reducible component and \( \tilde{\mathcal{G}}_{1}= ( \mathcal{V}_{1}, \tilde{ \mathcal{E}}_{1} ) \) is referred to as a reduction of $ \mathcal{G} $, where $ \tilde{ \mathcal{E}}_{1} = \mathcal{E}_{1}\cup (v_{1}, v_{2}) $, and $(v_1,v_2)$ is an additional (virtual) link, not originally present in $\mc G$.
\end{definition}

\begin{remark}
\label{rem:reducible-net}
In Definition~\ref{def:reducible-net},
\begin{enumerate}[label = (\alph*)]
\item reducibility of a network depends both on its topology as well as the location of the supply and demand nodes;
\item if $(v_1,v_2)$ is a link (or corresponds to several links when $\mc G$ is a multigraph), then it can be assigned arbitrarily to either $\mc E_1$ or $\mc E_2$; in this case the reduction process will result in an additional link $(v_1,v_2)$ in $\tilde{\mc E}_1$;
\item the supply-demand vector $p$ can be non-zero at $v_1$ or $v_2$.
\end{enumerate}
\end{remark}

%

\begin{figure}[htb!]
\begin{center}
\includegraphics[width=0.8\linewidth]{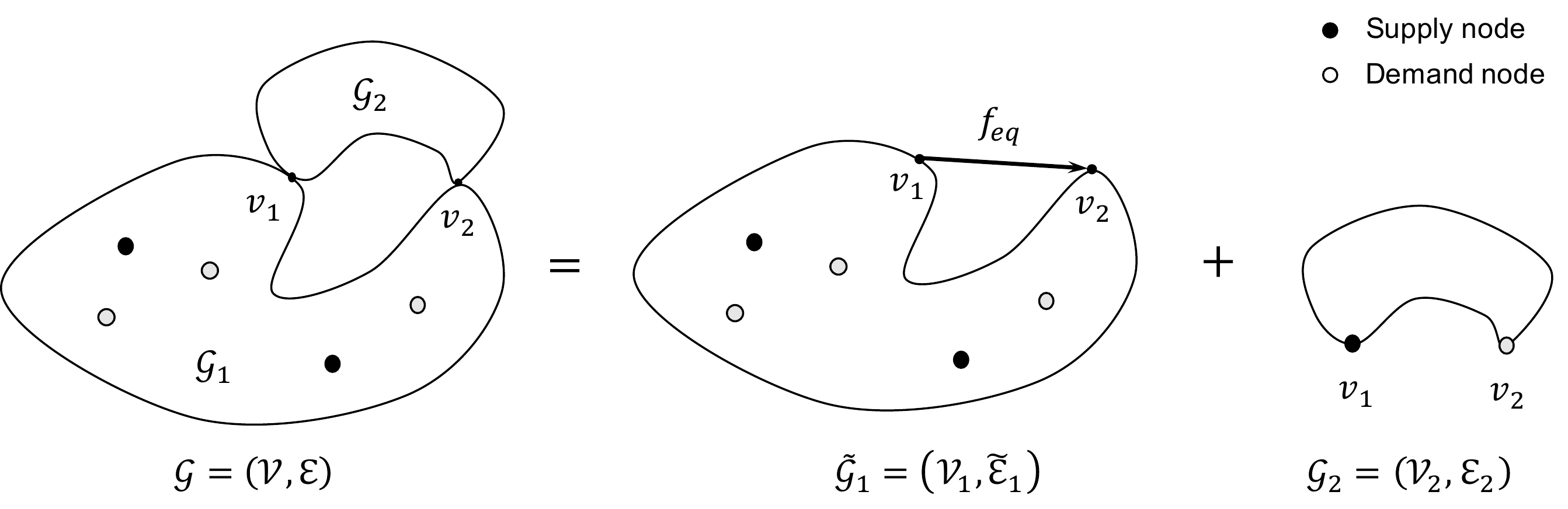}
\end{center}
\caption{Illustration of a reducible network}
\label{fig:network-reduction-illustration}
\end{figure}

We now describe a reduction procedure for a reducible network, \eg, as illustrated in Fig. \ref{fig:network-reduction-illustration}. Specifically, this network will be decomposed into two smaller sub-networks 
\( \tilde{\mathcal{G}}_{1} \) and \( \mathcal{G}_{2} \); $(v_1,v_2)$ is a virtual link with equivalent weight \( \weq = \mathcal{H}( w_{\mathcal{E}_{2}}, \mathcal{G}_{2}, v_{1}, v_{2}) \) as defined in
Definition \ref{def:equivalent-weight}, and $\mc G_2$ has a virtual supply-demand vector supported on nodes $v_1$ and $v_2$. The reduction is equivalent (cf. Lemma~\ref{lem:net-equiv-reduction}) in the sense that, the flows on links in $\mc E$ obtained from Lemma~\ref{lemma:flow-solution}, is the same as the flows on corresponding links in $\mc E_1$ and $\mc E_2$ by applying Lemma~\ref{lemma:flow-solution} to sub-networks $\tilde{\mc G}_1$ and $\mc G_2$, respectively. The flow on the virtual link is given by:
%
%
%
%
\begin{equation}
\label{eq:virtual-supply-demand}
  f_{\mathrm{eq}} = \sum_{i \in \mc E^+_{v_1}
  \cap \mc E_2} f_{i}  -\sum_{i\in \mc E_{v_1}^- \cap \mc E_2} f_{i} =
\sum_{i\in \mc E_{v_2}^- \cap \mc E_2} f_{i}  -\sum_{i \in \mc E_{v_2}^+ \cap
  \mc E_2} f_{i}.
\end{equation}
and the virtual supply-demand on $\mc G_2$ is $f_{\mathrm{eq}} a_{v_1 v_2}$, where  \( a_{v_1 v_2} \in \{-1,0,+1\}^{\mc V_2}\) is such that its $v_1$-th component is $+1$, $v_2$-th component is $-1$, and all the other components are zero.

\begin{lemma}
\label{lem:net-equiv-reduction}
Consider a network with directed multigraph $\mc G=(\mc V, \mc E)$, link weights $w \in \real_{>0}^{\mc E}$ and a supply-demand vector $p \in \real^{\mc V}$.  If $ \mathcal{G} $ is that is reducible (cf. Definition \ref{def:reducible-net}) about $v_1, v_2 \in \mc V$ under $ p $, the link flows $f^{\mc G}$ in $\mc G$, are equal to the corresponding link flows $f^{\tilde{\mc G}_1}$ and $f^{\mc G_2}$ in sub-networks $\tilde{\mc G}_1$ and $\mc G_2$, respectively. Formally, 
\begin{equation}
\label{eq:flow-eqvl}
\begin{aligned}
f_{\mathrm{eq}}&= f^{\tilde{\mathcal{G}}_{1}}_{v_{1}v_{2}}(w_{\tilde{\mathcal{E}}_{1}},
p_{\mathcal{V}_{1}}) \\
f_{i}^{\mathcal{G}}(w, p) &=
f_{i}^{\tilde{\mathcal{G}}_{1}}(w_{\tilde{\mathcal{E}}_{1}},
p_{\mathcal{V}_{1}}) \quad \forall i\in \mathcal{E}_{1}\\
f_{i}^{\mathcal{G}}(w, p) &= f_{i}^{\mathcal{G}_{2}}(w_{\mathcal{E}_{2}},
a_{v_{1}v_{2}}) f_{\mathrm{eq}} \quad \forall i \in \mathcal{E}_{2}
\end{aligned} 
\end{equation}
where \( w_{\tilde{\mathcal{E}}_{1}} \) and \( w_{\mathcal{E}_{2}} \) are the weight
matrices associated with sub-networks \( \tilde{\mathcal{G}}_{1} \) and \(
\mathcal{G}_{2} \), respectively, and \( p_{\mathcal{V}_{1}} \) is the
sub-vector of \( p \) corresponding to nodes in \( \mathcal{V}_{1} \). 
\end{lemma}
\begin{proof}
Noting that the components of the supply-demand vector at nodes in $\mc V_2$ are zero, Ohms's and Kirchhoff's laws for links $\mc E_2$ and nodes $\mc V_2$ in $\mc G_2$ can be written as:
\begin{equation}
\label{eq:kirchhoff-ohm-law-G2}
 w_{i} (\phi_{\sigma(i)} - \phi_{\tau(i)}) =  f_{i}^{ \mathcal{G}}  \quad \forall \, i \in \mathcal{E}_2, \qquad \sum_{ i\in \mathcal{E}_{v}^{+} \cap \mc E_2} f_{i}^{ \mathcal{G}} - \sum_{i\in \mathcal{E}_{v}^{-} \cap \mc E_2} f_{i}^{ \mathcal{G}} = 0
  \quad \forall \, v \in \mathcal{V}_{2}\backslash\{v_1, v_2\}.
\end{equation}
\eqref{eq:kirchhoff-ohm-law-G2}, along with \eqref{eq:virtual-supply-demand}, are the same
equations as one would get by writing Kirchhoff's and Ohm's law for $\mc G_2$
under supply-demand vector $f_{\mathrm{eq}}a_{v_{1}v_{2}}$. Taking this latter
interpretation of \eqref{eq:virtual-supply-demand} and
\eqref{eq:kirchhoff-ohm-law-G2}, Lemma~\ref{lemma:flow-solution} and its proof
then gives the flow solution on links in \( \mathcal{G}_{2} \) in 
\eqref{eq:flow-eqvl}, \ie, $f_{i}^{ \mathcal{G}} =
f_{i}^{\mathcal{G}_{2}}(w_{\mathcal{E}_{2}}, f_{\mathrm{eq}} a_{v_{1}v_{2}}) = f_{\mathrm{eq}}
f_{i}^{\mathcal{G}_{2}}(w_{\mathcal{E}_{2}}, a_{v_{1}v_{2}})$. Moreover, if 
$\phi_{\mc V_2}$ denotes the sub-vector of $\phi$ corresponding to nodes in $\mc
V_2$, then we have $\phi_{\mathcal{V}_{2}} =f_{\mathrm{eq}} L_{\mc G_2}^{\dagger}a_{v_{1}v_{2}}$,
and hence 
\begin{equation}
\label{eq:Ohm-law-eqv-link}
 \phi_{v_1} - \phi_{v_2} = \trans{a_{v_1v_2}} 
\phi_{\mathcal{V}_{2}} = f_{\mathrm{eq}}\trans{a_{v_1 v_2}} L_{\mc G_2}^{\dagger}a_{v_{1}v_{2}} = f_{\mathrm{eq}}/\weq
\end{equation}

$ f_{\mathrm{eq}} $ and (\ref{eq:Ohm-law-eqv-link}) can be seen as the flow on and Ohm's law for the virtual link (\( v_{1} \), \( v_{2} \)), respectively. Now writing Ohm's and Kirchhoff's laws for $\mc G_1$, we get
\begin{equation}
\label{eq:kirchhoff-ohm-law-G1}
\begin{aligned}
& w_{i} (\phi_{\sigma(i)} - \phi_{\tau(i)}) =  f_{i}^{ \mathcal{G}}  \quad \forall \, i \in
 \mathcal{E}_1, & \qquad & \sum_{ i\in \mathcal{E}_{v}^{+}} f_{i}^{ \mathcal{G}} - \sum_{i \in
   \mathcal{E}_{v}^{-}} f_{i}^{ \mathcal{G}}  =p_{v}
  \quad \forall \, v \in \mathcal{V}_{1}\backslash \{v_1, v_2\} \\
& \sum_{j\in \mathcal{E}^{+}_{v_1} \cap \mc E_1} f_{i}^{ \mathcal{G}} - \sum_{i\in \mathcal{E}^{-}_{v_1} \cap \mc E_1}
 f_{i}^{ \mathcal{G}}  + f_{\mathrm{eq}}= 0   & &
 \sum_{i\in \mathcal{E}^{+}_{v_2} \cap \mathcal{E}_{2}} f_{i}^{ \mathcal{G}}  +\sum_{i\in \mathcal{E}^{-}_{v_2}\cap \mathcal{E}_{2}}
 f_{i}^{ \mathcal{G}} -  f_{\mathrm{eq}}= 0  \\
\end{aligned}
\end{equation}
where we use the definition of $f_{\mathrm{eq}}$ in
\eqref{eq:Ohm-law-eqv-link}. As $f_{\mathrm{eq}}$ is interpreted to
be the flow on a virtual link $(v_1,v_2)$, \eqref{eq:Ohm-law-eqv-link} and
\eqref{eq:kirchhoff-ohm-law-G1} become Ohm's and Kirchhoff's laws for
$\tilde{\mc G}_1$. The expression for $f_{i}^{ \mathcal{G}}$, \( i\in \mathcal{E}_{1} \) and $ f_{\mathrm{eq}} $, in \eqref{eq:flow-eqvl} now follows from Lemma~\ref{lemma:flow-solution} and its proof.
%
\end{proof}

\eqref{eq:Ohm-law-eqv-link} motivates the following definition of equivalent weight.
\begin{definition}[Equivalent Weight]
\label{def:equivalent-weight}
Given a network with directed multigraph $\mc G=(\mc V, \mc E)$, link weights $w \in \real_{>0}^{\mc E}$, the equivalent weight between two given nodes $v_1, v_2 \in \mc V$ is defined as:
\begin{equation}
\label{eq:equivalent-weight-def}
\mc H(w, \mc G, v_1,v_2):=\frac{1}{\trans{a}_{v_1v_2}L_{\mc G}^{\dagger}a_{v_1 v_2}}
\end{equation}
where \( a_{v_1v_2} \in \{-1,0,+1\}^{\mc V}\) is such that its $v_1$-th
component is $+1$, the $v_2$-th component is $-1$, and all the other components
are zero.
\end{definition}

\begin{remark}
\label{rem:equivalent-weight}
Definition \ref{def:equivalent-weight} is well-posed, \ie, \( \trans{a}_{v_{1}v_{2}} L^{\dagger}_{\mathcal{G}}
a_{v_{1}v_{2}} > 0\) for all (connected) networks \( \mathcal{G} \) and \( v_{1},
v_{2} \in \mathcal{V} \). This is because, \( L^{\dagger}_{\mathcal{G}} \) is positive definite in space \(
\real^{\mathcal{E}}\setminus
\spn{\onebf} \). 
\end{remark}

At times, when the graph $\mc G$ and nodes $v_1$ and $v_2$ are clear from the context, we shall denote the equivalent weight simply by $\mc H(w)$ for brevity. 
\( \mathcal{H}(w) \) is generalization of rather standard formulae for equivalent resistances for serial and parallel connections from circuit theory, which we briefly state next for completeness.

\begin{example}[Equivalent Weight for Serial and Parallel Networks]
\label{ex:eq-weight}
For a network consisting only of $m$ parallel links from node $v_1$ to node $v_2$, with weights \( w_{i} \) (\( i =
1, 2, \ldots, m\)), the equivalent weight between $v_1$ and $v_2$, as given by \eqref{eq:equivalent-weight-def}, is 
\begin{equation*}
\label{eq:eqvl-weight-para-net}
  \mc H(w) = \left( [1, -1] \left[
      \begin{matrix}
        \sum_{i=1}^{m} w_{i} & -\sum_{i=1}^m w_{i} \\
        -\sum_{i=1}^m w_{i} & \sum_{i=1}^m w_{i}
      \end{matrix}
\right]^{\dagger} \left[\begin{matrix}
  1 \\ -1
\end{matrix} \right] \right)^{-1} = \sum_{i=1}^m w_{i}
\end{equation*}
Similarly, for a network consisting only of $m$ links in series from node $v_1$ to node $v_2$, with weights \( w_{i} \) (\( i =
1, 2, \ldots, m\)), the equivalent weight, as given by \eqref{eq:equivalent-weight-def} is $\left( \sum_{i=1}^{m} {1}/{w_{i}} \right)^{-1}$. 

The expressions for the equivalent weight in these two canonical cases, as given by 
\eqref{eq:equivalent-weight-def}, are the same as standard formulae from circuit theory.
\end{example}

For given $\mc G=(\mc V, \mc E)$, and $v_1, v_2 \in \mc V$, monotonicity of $\mc H(w)$ with respect to components of $w$ follows from Rayleigh's monotonicity law, \eg, see \cite[Section 1.4]{Snell.Doyle:00}.
Nevertheless, for the sake of completeness, and also to describe an alternate short proof based on the techniques developed in Section~\ref{sec:jacobian-exact}, we state this result next.
\begin{lemma}
  \label{lem:equiv-weight-fun-increasing}
 For a network with underlying graph $\mc G=(\mc V, \mc E)$ and weight $ w \in \real^{ \mathcal{E}} $, the equivalent weight function between any two nodes $v_1, v_2 \in \mc V$, as defined in \eqref{eq:equivalent-weight-def}, satisfies the following:
$$
\frac{\partial}{\partial w_i} \mc H(w, \mc G,v_1,v_2) \ge 0, \quad \forall i \in \mc E, \, \, w \in \real_{>0}^{\mc E}
$$
\end{lemma}
\begin{proof}
  Let \( \hat{\mc H}(w) := 
  \trans{a}_{v_1v_2}L^{\dagger} a_{v_1v_2} \). The lemma then follows from: 
\begin{equation*}
\frac{\partial \hat{\mc H}(w)}{\partial w_{i}} = \trans{a}_{v_1v_2} \frac{\partial
  L^{\dagger}}{\partial w_{i}} a_{v_1v_2} = - \trans{a}_{v_1v_2} L^{\dagger}a_{i}
\trans{a}_{i}L^{\dagger} a_{v_1v_2}  = - \left( \trans{a}_{i} L^{\dagger} a_{v_1v_2}
\right)^{2} \le 0
\end{equation*}
where the second equality is due to \eqref{eq:derivative-pinv-laplacian}.
In the above equation, $a_{v_1v_2}$ has the same meaning as in Definition~\ref{def:equivalent-weight}, and $a_i$ is the $i$-th column of the node-link incidence matrix $A$ associated with $\mc G$.
\end{proof}

\begin{remark}
\label{eq:eqv-weight-monotonicity-implication}
\begin{enumerate}
\item[(a)] Example~\ref{ex:eq-weight} implies that the equivalent weight function is strictly monotone for series and parallel networks.
\item[(b)] Monotonicity of $\mc H$ from Lemma~\ref{lem:equiv-weight-fun-increasing} along with its continuity implies that $\mc H(w)$ is not necessarily a one-to-one map from $[w^l,w^u] \subset \real_{>0}^{\mc E}$ to $[\mc H(w^l), \mc H(w^u)] \subset \real_{>0}$.
\end{enumerate}
\end{remark}

It is easy to see that the equivalent network reduction implied by Lemma~\ref{lem:net-equiv-reduction} reduces computational complexity for computing link flows by decomposing the original network into sub-networks. We now show that such a decomposition approach is naturally aligned with Proposition~\ref{cor-eqvl-transformation}, and leads to reduction in computational complexity of the weight control problem~\eqref{eq:weight-control-original}-\eqref{eq:weight-control-nu} by formulating an equivalent bilevel problem (cf. Remark~\ref{rem:bilevel}(c)). We first note that the network reduction implemented for the nominal supply-demand vector $p_0$ is also valid for nongenerative disturbances (cf. Definition~\ref{def:nongenerative-disturbance}) associated with $p_0$. Hence, Lemma~\ref{lem:net-equiv-reduction} is also applicable for all nongenerative disturbances. %
Therefore, one can rewrite the constraints in
(\ref{eq:weight-control-nu}) as: 
\begin{equation}
\label{opt:weight-control-reformulation}
\begin{aligned}
c^{l}_{i} \le f_{i}^{\tilde{\mathcal{G}}_{1}}(w_{\tilde{\mathcal{E}}_{1}},
p_{\triangle, \mathcal{V}_{1}}) \le c_{i}^{u} \quad \forall\, i \in \mathcal{E}_{1} \\
c^{l}_{i} \le  f_{\mathrm{eq}} f_{i}^{\mathcal{G}_{2}}(w_{\mathcal{E}_{2}},
a_{v_{1}v_{2}}) \le c^{u}_{i} \quad \forall \, i \in \mathcal{E}_{2} \\ 
f_{\mathrm{eq}} =  f_{v_{1}v_{2}}^{\tilde{\mathcal{G}}_{1}}(w_{\tilde{\mathcal{E}}_{1}},
p_{\Delta, \mathcal{V}_{1}}) \\
\weq = \mathcal{H}(w_{\mathcal{E}_{2}})
\end{aligned}
\end{equation}
where we recall that $ p_{\triangle} = p_{0} + \triangle $ is the disturbed supply-demand vector and $ p_{\triangle, \mathcal{V}_{1}}$ is the subvector corresponding to node set $ \mathcal{V}_{1} $ of $ p_{\triangle} $.
The analogy between
\eqref{opt:original} and \{(\ref{eq:weight-control-nu}), \eqref{opt:weight-control-reformulation}\} is more apparent now:
\( \mathcal{J}_{1} \equiv
\mathcal{E}_{1} \) and \( \mathcal{J}_{2} \equiv
\mathcal{E}_{2} \), \( x_{\mathcal{I}_{1}} \equiv w_{\mathcal{E}_{1}} \cup \{ \triangle \} \) and \(
x_{\mathcal{I}_{2}} \equiv w_{\mathcal{E}_{2}} \), \( y_{1} \equiv
f_{\mathrm{eq}} \) and \( y_{2} \equiv \weq \), \(
q_{i}(x_{\mathcal{I}_{1}}, y_{2}) \equiv  f_{i}^{\tilde{\mathcal{G}}_{1}}(w_{\tilde{\mathcal{E}}_{1}},
p_{\triangle, \mathcal{V}_{1}})  \) for all \( i\in \mathcal{E}_{1} \) and \(
q_{i}(x_{\mathcal{I}_{2}}, y_{1})\equiv  f_{\mathrm{eq}}  f_{i}^{\mathcal{G}_{2}}(w_{\mathcal{E}_{2}},
a_{v_{1}v_{2}}) \) for all \( i\in \mathcal{E}_{2} \) \footnote{Every
  constraint in \( \mathcal{E}_{1} \) and \( \mathcal{E}_{2} \)  corresponds to
  two constraints in the formulation of (\ref{opt:original}), \ie,  \(
  c_{i}^{l} \le q_{i}(\cdot)
  \le c^{u}_{i} \) corresponds to \( q_{i}(\cdot)-c_{i}^{u} \le 0 \) and \(
 - q_{i}(\cdot)+c_{i}^{l} \le 0  \). },
 \( h_{1}(x_{\mathcal{I}_{1}}, y_{2}) \equiv
f_{v_{1}v_{2}}^{\tilde{\mathcal{G}}_{1}}(w_{\tilde{\mathcal{E}}_{1}}, p_{\triangle, \mathcal{V}_{1}}) \) and \( h_{2}(x_{\mathcal{I}_{2}}) \equiv
\mathcal{H}(w_{\mathcal{E}_{2}}) \). Since \( q_{i}(\cdot, y_{1}) \) is a linear
function with respect to \(y_{1} \) for all \( w_{\mathcal{E}_{2}}\in
D^{\mathcal{E}_{2}} \) (recall $ D^{\mathcal{E}_{2}} = [w_{ \mathcal{E}_{2}}^{l}, w_{ \mathcal{E}_{2}}^{u}] $) and \( i\in \mathcal{E}_{2} \), both
\( q_{i}(\cdot, y_{1}) \) and \( -q_{i}(\cdot, y_{1}) \) are quasiconvex with
respect to \( y_{1} \) for all \( i\in \mathcal{E}_{2} \). Furthermore, it is
straightforward to see that 
\begin{equation}
\label{eq:eqv-capacity-bound}
c_{i}^{l} \le q_{i}(\cdot, 0) = 0 \le c_{i}^{u} \qquad \forall 
 i\in \mathcal{E}_{2} 
\end{equation}
Therefore, proposition~\ref{cor-eqvl-transformation} can then be applied and gives the following result.

\begin{proposition}
\label{prop:robustness-margin-bilevel}
Consider a network with directed multigraph $\mc G = ( \mathcal{V}, \mathcal{E})$, lower and upper link weights $w^l \in \real_{>0}^{\mc E}$ and $w^u \in \real_{>0}^{\mc E}$ respectively, a supply-demand vector $p \in \real^{\mc V}$. If $ \mathcal{G} $ is reducible (cf. Definition \ref{def:reducible-net}) about $v_1, v_2 \in \mc V$ under $ p $ and the disturbances are nongenerative (cf. \ref{def:nongenerative-disturbance}),   then
(\ref{eq:weight-control-param}) is equal to the following
\begin{equation}
  \label{opt:weight-control-equiv}
  \begin{aligned}
\underset{ \delta \in \nongenset(p_{0}) }{\min} \quad  & \underset{\begin{subarray}{c}  \mu \geq 0 \\  w_{\mathcal{E}_{1}} \in D^{\mc E_1} \\  \weq \in D_{2}   \end{subarray} }{\max} && \mu\\
&    \,\,    \text{subject to} && c_{i}^{l} \le 
        f_{i}^{\tilde{\mathcal{G}}_{1}}(w_{\tilde{\mathcal{E}}_{1}},
        p_0 + \mu \delta ) \leq c^{u}_{i} 
   \quad \forall\, i \in \mathcal{E}_{1} \\
&&& g^{l}(\weq) \le  f_{v_1
  v_2}^{\tilde{\mathcal{G}}_{1}}(w_{\tilde{\mathcal{E}}_{1}},
p_0 + \mu \delta ) \leq g^{u}(\weq) 
  \end{aligned}
  \end{equation}
where $D^{\mc E_1}:=[w_{\mc E_1}^l, w_{\mc E_1}^u]$, \( D_{2} := [\mathcal{H}(w_{\mathcal{E}_{2}}^{l}), \mathcal{H}(w^{u}_{\mathcal{E}_{2}})] \),  \( g^{l}(\weq) := \min G(\weq) \) and \(
g^{u}(\weq) := \max G(\weq) \) with the set \( G(\weq)
\) defined as:
  \begin{equation*}
   G(\weq) := \{z\in \real \,|\, c_{ \mathcal{E}_{2}}^{l} \le z f^{\mathcal{G}_{2}}(w_{\mathcal{E}_{2}},
  a_{v_{1}v_{2}}) \le c_{ \mathcal{E}_{2}}^{u} \text{ for some } w_{\mathcal{E}_{2}}\in
  D^{\mathcal{E}_{2}} \text{ satisfying } \mathcal{H}(w_{\mathcal{E}_{2}}) =
  \weq
  \}
  \end{equation*}
  where $D^{\mc E_2}:=[w_{\mc E_2}^l, w_{\mc E_2}^u]$.
\end{proposition} 

Noting the structural similarity between the two inequality constraints in the
upper level problem in \eqref{opt:weight-control-equiv}, it is compelling to
interpret $g^{l}(\weq)$ and \( g^{u}(\weq) \) as the
\emph{equivalent capacities} (lower and upper respectively) of the equivalent virtual link
$(v_1,v_2)$. 


\begin{definition}[Equivalent Capacities]
\label{def:equivalent-capacity}
Consider a network consisting of directed multigraph \(
\mathcal{G}=(\mathcal{V} ,\mathcal{E}) \), lower and upper bounds on link
weights $w^l \in \real_{>0}^{\mc E}$ and $w^u \in \real_{>0}^{\mc E}$,
respectively,  and link lower and upper capacity functions $\map{c_i^{l}}{[w_i^l,w_i^u]}{\real_{<
    0}}$ and $\map{ c_{i}^{u}}{[w_i^l,w_i^u]}{\real_{>
    0}}$, $i \in \mc E$, respectively. Between any two nodes $v_1, v_2 \in \mc V$ and for a given equivalent weight $\weq \in [\mathcal{H}(w_{\mathcal{E}_{2}}^{l}, \mathcal{G}, v_{1}, v_{2}), \mathcal{H}(w^{u}_{\mathcal{E}_{2}}, \mathcal{G}, v_{1}, v_{2})] $, the corresponding equivalent lower capacity $\mc C^{l}(\weq)$ and
upper capacity \( \mathcal{C}^{u}(\weq) \) are defined as:  

\begin{equation}
  \label{eq:eqv-cap-def}
  \mathcal{C}^{l}(\weq, \mc G, v_1,v_2) := \min G(\weq); \quad \mathcal{C}^{u}(\weq, \mc G, v_1,v_2) := \max G(\weq)
\end{equation}
where 
  \begin{equation*}
  G(\weq) := \{z\in \real \,|\, c^{l} \le z f(w,
  a_{v_{1}v_{2}}) \le c^{u} \text{ for some } w \in [w^l, w^u] \text{ satisfying } \mathcal{H}(w, \mathcal{G}, v_{1}, v_{2}) =\weq \}
%
  \end{equation*}
and \( a_{v_1v_2} \in \{-1,0,+1\}^{\mc V}\) is such that its $v_1$-th component is $+1$, the $v_2$-th component is $-1$, and all the other components are zero. 
\end{definition}

For brevity in notations, we drop the dependence of $\mc C^l$ and $\mc C^u$ on $\mc G$, $v_1$ or $v_2$, when clear from the context. 

 \begin{remark}
\leavevmode
\label{rem:eqv-cap-problem}
\begin{enumerate}
\item[(a)] Note that the link capacity functions in Definition~\ref{def:equivalent-capacity} are assumed to be weight-dependent. This general setup allows definition of equivalent capacity to be applicable to networks whose links themselves could be equivalent links for some underlying sub-network. This feature is specifically used in extending the bilevel formulation to a multilevel framework in Section~\ref{subsec:multilevel}.
\item[(b)] Remarkably, the equivalent capacity can be expressed concisely in
terms of the equivalent weight, as opposed to the entire weight vector $w_{\mc
  E_2}$. This considerably reduces the complexity of the weight control
problem (\ref{eq:weight-control-original})-(\ref{eq:weight-control-nu}). 
\item[(c)] Computing the equivalent capacities for a given $ \weq $ between two nodes \( v_{1} \) and \( v_{2} \) of a
  network \( \mathcal{G} \) is equivalent to solving the weight control problem (\ref{eq:weight-control-param}) for \( \mathcal{G} \) with a single supply node \( v_{1} \), a single demand node \( v_{2} \), and under multiplicative disturbances -- however, with the 
additional equality constraint \( \mathcal{H}(\mathcal{G}, w, v_{1}, v_{2}) =\weq \). Therefore, when the network contains only one supply node and one demand node, finding the equivalent capacity functions 
\( \mathcal{C}^{l} \) and \( \mathcal{C}^{u} \) can be considered to be a
generalization of solving \( \alpha_{-}^* \) and \( \alpha_{+}^* \) in 
\eqref{eq:weight-control-param}. More specifically,  
\begin{equation*}
\alpha^{*}_{+} = \max_{\mathcal{H}(w^{l}) \le \weq \le \mathcal{H}(w^{u}) } \mathcal{C}^{u}(\weq); \quad 
\alpha^{*}_{-} = -\min_{\mathcal{H}(w^{l}) \le \weq \le \mathcal{H}(w^{u}) } \mathcal{C}^{l}(\weq)
\end{equation*} 
\end{enumerate}
 \end{remark}

\subsection{A Nested Bilevel Approach for Multilevel Formulation}
\label{subsec:multilevel}
For a reducible network as per Definition~\ref{def:reducible-net}, 
Proposition~\ref{prop:robustness-margin-bilevel} shows that the weight
control problem \eqref{eq:weight-control-original}-\eqref{eq:weight-control-nu} can be transformed into a bilevel optimization
problem \eqref{opt:weight-control-equiv}, in which the lower level problem involves
finding the equivalent lower and upper capacity functions of an appropriate
subnetwork.  We now extend this to a multilevel framework.

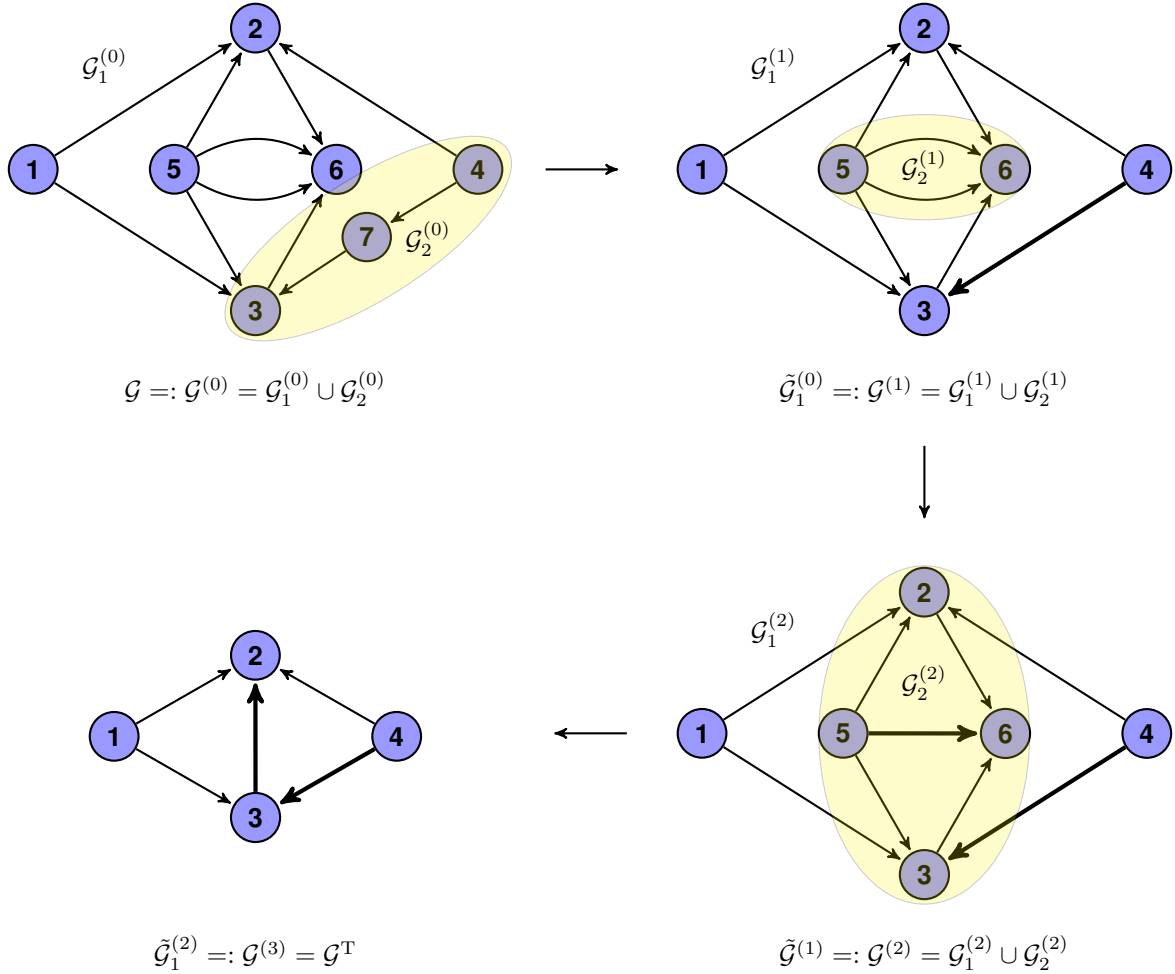
\begin{figure}[ht]
  \centering
\begin{tikzpicture}
 \matrix[row sep = 2cm, column sep = 2cm]
 {
 \begin{scope}[->,>=stealth',shorten >=1pt,auto, thick,every node/.style={circle,draw,font=\sffamily\bfseries,
                     fill=blue!40}]
  \node (2) at (0, 0) {2};
 \node (5) [below left = 1.4 and 0.6 of 2] {5};
 \node (6) [below right = 1.4 and 0.6 of 2] {6};
 \node (1) [left = 1.2  of 5] {1};
 \node (4) [right =1.2 of 6] {4};
 \node (3) [below right = 1.4 and 0.6 of 5] {3};
 \node (7) [above right = 0.5 and 1 of 3] {7};
 
 \path[every node/.style={font=\sffamily}]
   (1) edge (2)    (1) edge (3)    (7) edge (3)   (4) edge (2)    (4) edge (7)  
   (5) edge [bend left] (6)    (5) edge [bend right] (6)
   (5) edge (3)    (5) edge (2)    (2) edge (6)   (3) edge (6);
\end{scope}

\begin{scope}[every node/.style={font=\sffamily}]
\draw[fill=yellow, opacity=0.2, rotate around= {32.5: (1.5, -2.8)}] (1.5, -2.8)  circle [x radius=2.2cm, y radius=0.8cm];
\node (G02) at (2.3, -2.8) {\( \mathcal{G}_{2}^{(0)}\)};
\node (G01) at (-2, -0.5) {\( \mathcal{G}_{1}^{(0)} \)};  
\node (G0) at (0, -4.8) {\( \mathcal{G} =:\mathcal{G}^{(0)}  = \mathcal{G}_{1}^{(0)}\cup \mathcal{G}_{2}^{(0)} \)};  
\end{scope}

& 
 \begin{scope}[->,>=stealth',shorten >=1pt,auto, thick,every node/.style={circle,draw,font=\sffamily\bfseries, fill=blue!40}]
 \node (12) at (0, 0) {2};
 \node (15) [below left = 1.4 and 0.6 of 12] {5};
 \node (16) [below right = 1.4 and 0.6 of 12] {6};
 \node (11) [left = 1.2  of 15] {1};
 \node (14) [right =1.2 of 16] {4};
 \node (33) [below right = 1.4 and 0.6 of 15] {3};
 
 \path[every node/.style={font=\sffamily}]
   (11) edge (12)    (11) edge (33)    (14) edge[ultra thick] (33)    (14) edge (12)    
   (15) edge [bend left] (16)  (15) edge [bend right] (16) 
   (15) edge (33)    (15) edge (12)    (12) edge (16)   (33) edge (16);
\end{scope}

\begin{scope}[every node/.style={font=\sffamily}]
\draw[fill=yellow, opacity=0.2] (0,-1.85) circle [x radius=1.4cm, y radius=0.7cm];
   \node (G12) at (0, -1.85) {\(\mathcal{G}_{2}^{(1)}\)};
\node (G11) at (-2, -0.5) {\( \mathcal{G}_{1}^{(1)} \)};  
\node (G1) at (0, -4.8) {\( \tilde{ \mathcal{G}}^{(0)}_{1} =: \mathcal{G}^{(1)} =  \mathcal{G}_{1}^{(1)}\cup \mathcal{G}_{2}^{(1)} \)};  
\end{scope}

\\  
 \begin{scope}[->,>=stealth',shorten >=1pt,auto, thick,every node/.style={circle,draw,font=\sffamily\bfseries, fill=blue!40}]
 \node (42) at (0, -0.84) {2};
 \node (41) [below left = 0.6 and 1.4 of 42] {1};
 \node (44) [below right =0.6 and 1.4 of 42] {4};
 \node (43) [below right =0.6 and 1.4 of 41] {3};
 
 \path[every node/.style={font=\sffamily}]
   (41) edge (42)    (41) edge (43)    (44) edge [ultra thick] (43)   (44) edge (42)    (43) edge [ultra thick] (42);
\end{scope} 

\begin{scope}[every node/.style={font=\sffamily}]
\node (G1) at (0, -4.8) {\( \tilde{\mathcal{G}}_{1}^{(2)} =: \mathcal{G}^{(3)} = \mathcal{G}^{\mathrm{T}}  \)};  
\end{scope}

&
 \begin{scope}[->,>=stealth',shorten >=1pt,auto, thick,every node/.style={circle,draw,font=\sffamily\bfseries, fill=blue!40}]
 \node (12) at (0, 0) {2};
 \node (15) [below left = 1.4 and 0.6 of 12] {5};
 \node (16) [below right = 1.4 and 0.6 of 12] {6};
 \node (11) [left = 1.2  of 15] {1};
 \node (14) [right =1.2 of 16] {4};
 \node (13) [below right = 1.4 and 0.6 of 15] {3};
 
 \path[every node/.style={font=\sffamily}]
   (11) edge (12)    (11) edge (13)    (14) edge[ultra thick] (13)    (14) edge (12)    
   (15) edge [ultra thick] (16)  
   (15) edge (13)    (15) edge (12)    (12) edge (16)   (13) edge (16);
\end{scope}

\begin{scope}[every node/.style={font=\sffamily}]
\draw[fill=yellow, opacity=0.2] (0,-1.9) circle [x radius=1.4cm, y radius=2.25cm];
   \node (G22) at (0, -1.2) {\(\mathcal{G}_{2}^{(2)}\)};
\node (G21) at (-2, -0.5) {\( \mathcal{G}_{1}^{(2)} \)};  
\node (G2) at (0, -4.8) {\( \tilde{\mathcal{G}}^{(1)}  =: \mathcal{G}^{(2)} =  \mathcal{G}_{1}^{(2)}\cup \mathcal{G}_{2}^{(2)} \)};  
\end{scope}

\\
};

\begin{scope}[->,>=stealth',shorten >=1pt,auto, thick]
 \draw  (4)++(0.9, 0) -> +(1, 0);
 \draw  (11)++(-1, 0) -> +(-1, 0);
 \draw  (33)++(0, -1.8) -> +(0, -1);
\end{scope} 
\end{tikzpicture}
\caption{Illustration of recursive network reduction, where the supply node set is \(\{1, 4\} \) and the
    demand node set is \( \{ 2, 3 \}\); the thick edges denote the equivalent links. The original network $ \mathcal{G} $ is reduced into the terminal network $ {\mathcal{G}}^{\mathrm{T}} $ in three reductions:  (1) subnetwork $ \mathcal{G}_{2}^{(0)} $  $ \rightarrow $ link $ (4, 3) $; (2) subnetwork $ \mathcal{G}_{2}^{(1)} $  $ \rightarrow $ link $ (5, 6) $; (3) subnetwork $ \mathcal{G}_{2}^{(2)} $  $ \rightarrow $ link $ (3, 2) $. $ \mathcal{G}^{k} : = \tilde{\mathcal{G}}^{(k-1)}_{1}$ is the resulting network after $ k $th reduction. 
  Notice the first and the second reduction can be implemented in parallel and the terminal network $ \mathcal{G}^{\mathrm{T}} = \mathcal{G}^{(3)}$ is not reducible.  \label{fig:net-reduction} }
\end{figure}

%

A comparison with \eqref{eq:weight-control-original}-\eqref{eq:weight-control-nu} reveals that the upper
level problem \eqref{opt:weight-control-equiv} is indeed the same as
\eqref{eq:weight-control-original}-\eqref{eq:weight-control-nu} written for the sub-network $\tilde{\mc G}_1$,  where
the equivalent link $ (v_{1}, v_{2}) $ has weight $w_{\mathrm{eq}}=\mc H(w_{\mc E_2})\in
[\mathcal{H}(w^{l}_{\mathcal{E}_{2}}), \mathcal{H}(w^{u}_{\mathcal{E}_{2}})] $
and weight dependent lower and upper capacities $\mc C^{l}(\weq)$ and \(
\mathcal{C}^{u}(\weq) \), respectively. If the reduced sub-network \(
\tilde{\mathcal{G}}_{1} =:\mc G^{(1)} \) is also reducible as per Definition \ref{def:reducible-net}, with its sub-networks $\mc G_{1}^{(1)}$ and $\mc G_{2}^{(1)}$, one can apply Proposition~\ref{cor-eqvl-transformation} to \eqref{opt:weight-control-equiv} to get an equivalent bilevel formulation for $\mc G^{(1)}$ if: (a) $q_i$ for $i \in \mc E_{2}^{(1)} = \tilde{ \mathcal{E}}_{1}$ are quasiconvex, and (b) the equivalent lower and upper capacity functions for links in $\mc G_{2}^{(1)}$ are strictly negative and positive respectively, as in \eqref{eq:eqv-capacity-bound}. (a) is satisfied trivially as before because of linearity of $q_i$, and (b) follows from the next result. 
%
%
\begin{lemma}
Consider a network consisting of directed multigraph \(
\mathcal{G}=(\mathcal{V} ,\mathcal{E}) \),  lower and upper bounds on link
weights $w^l \in \real_{>0}^{\mc E}$ and $w^u \in \real_{>0}^{\mc E}$,
respectively,  and link lower and upper capacity functions $\map{c_i^{l}}{[w_i^l,w_i^u]}{\real_{<
    0}}$ and $\map{ c_{i}^{u}}{[w_i^l,w_i^u]}{\real_{>
    0}}$, $i \in \mc E$, respectively. The equivalent lower capacity $\mc C^{l}(\weq)$ and
upper capacity \( \mathcal{C}^{u}(\weq) \) between two given nodes $v_1, v_2 \in \mc V$ satisfy the following:
\begin{equation*}
  \mathcal{C}^{l}(\weq) < 0 <
  \mathcal{C}^{u}(\weq) \quad \forall \, \weq \in [\mathcal{H}(w^{l}), \mathcal{H}(w^{u})]
\end{equation*}
\end{lemma}
\begin{proof}
Since \( c_{i}^{l} < 0 < c_{i}^{u} \) for all \( i\in
\mathcal{E} \), we have \( 0 \in G(\weq) \) in Definition
\ref{def:equivalent-capacity}. It is easy to see that, for \( z_{0} = \min_{i\in \mathcal{E}}\{- \max_{w_{i}^{l}\le
  w_{i} \le w_{i}^{u}}
c_{i}^{l}(w_{i}), \min_{w_{i}^{l} \le w_{i}\le w_{i}^{u}} c_{i}^{u}(w_{i}) \}>0
\), we have \( \mathcal{C}^{l}(\weq) \le -z_{0} < 0 <
z_{0} < \mathcal{C}^{u}(\weq) \) for all \( \weq \in [\mathcal{H}(w^{l}),
\mathcal{H}(w^{u})] \). This is true because \( |f_{i}(w, a_{v_{1}v_{2}}) |\le 1
\) for all \( i\in \mathcal{E} \) from Lemma \ref{lem:flow-less-than-inflow} in Appendix~\ref{sec:flow-sol-prop}.
\end{proof} 

A recursive application of this procedure leads to an equivalent multilevel formulation for the  original weight control problem in \eqref{opt:weight-control-equiv}; the process stops when the sub-network corresponding to the upper level problem, referred to as the \emph{terminal network}, is not reducible, as per Definition \ref{def:reducible-net}. The resulting multilevel hierarchy consists of a series of a collection of lower level problems, and an upper level problem corresponding to the last recursion. We appropriately then refer to the former as \emph{reduction problems} and the latter as the \emph{terminal problem}. The reduction problem $\mathbf{P}_r$ is formalized next in Problem \ref{prob:reduction}, and the terminal problem is the generalized weight control problem $\mathbf{P}$ (cf. Problem \ref{prob:weight-control})  on the terminal network.
   
\begin{problem}[Reduction problem $ \mathbf{P}_{r} $]
 \label{prob:reduction}
\begin{center}
\begin{tikzpicture} 
  \node[minimum height = 1cm, minimum width = 2cm, draw,ultra thick] (1) {\( \mathbf{P}_{r} \) };
  \node (in) [left = 1cm of 1] {($ \mathcal{G}, w^{l}, w^{u},  c^{l}, c^{u}, v_{1}, v_{2} $)}; 
  \node (out) [right = 1cm of 1] { $ (w_{v_{1}v_{2}}^{l}, w_{v_{1}v_{2}}^{u}, c_{v_{1}v_{2}}^{l}, c_{v_{1}v_{2}}^{u} )$};
  \draw[->, thick]    (in) -- (1); 
   \draw[->, thick]    (1) -- (out); 
\end{tikzpicture}
\end{center}
\begin{description}
\item[input] network \( \mathcal{G} = (\mc V, \mc E) \) with link weights bounds $w^l \in \real_{>0}^{\mc E}$ and $w^u \in \real_{>0}^{\mc E}$, link capacity functions $\map{c_{i}^l}{[w_{i}^l, w_{i}^u]}{\real_{<0}} $ and $\map{c_{i}^u}{[w_{i}^l, w_{i}^u]}{\real_{>0}} $ for $ i\in \mathcal{E} $, and nodes $v_1, v_2 \in \mc V$.
\item[output] equivalent lower and upper weight bounds: $w_{v_{1}v_{2}}^l=\mc H(w^l)$ and $w_{v_{1}v_{2}}^u=\mc H(w^u)$, where $\mc H$ is as in Definition~\ref{def:equivalent-weight}; equivalent lower and upper capacity functions: $c_{v_{1}v_{2}}^{l} = \map{\mc C^l}{[w_{v_{1}v_{2}}^l, w_{v_{1}v_{2}}^u]}{\real_{<0}}$ and  $c_{v_{1}v_{2}}^{u} = \map{\mc C^u}{[w_{v_{1}v_{2}}^l, w_{v_{1}v_{2}}^u]}{\real_{>0}}$, where $\mc C^l(\weq)$ and $\mc C^u(\weq)$ are as in Definition~\ref{def:equivalent-capacity}.
\end{description}
\end{problem}   

\begin{problem}[Generalized weight control problem  $\mathbf{P}$]
\label{prob:weight-control}
\begin{center}
\begin{tikzpicture} 
  \node[minimum height = 1cm, minimum width = 2cm, draw,ultra thick] (1) {\( \mathbf{P} \) };
  \node (in) [left = 1cm of 1] {($ \mathcal{G}, w^{l}, w^{u},  c^{l}, c^{u}, p_{0} $)}; 
  \node (out) [right = 1cm of 1] { $ \nu^{*}( \mathcal{G})$};
  \draw[->, thick]  (in) -- (1) ; 
  \draw[->, thick]  (1) -- (out); 
\end{tikzpicture}
\end{center}
\begin{description}
\item[input] network \( \mathcal{G} = (\mc V, \mc E) \) with link weights bounds $w^l \in \real_{>0}^{\mc E}$ and $w^u \in \real_{>0}^{\mc E}$, link capacity functions $\map{c_{i}^l}{[w_{i}^l, w_{i}^u]}{\real_{<0}} $ and $\map{c_{i}^u}{[w_{i}^l, w_{i}^u]}{\real_{>0}} $ for $ i\in \mathcal{E} $, and initial supply-demand vector $ p_{0} \in \real^{ \mathcal{V}} $
\item[output]  margin of robustness: $ \nu^{*}( \mathcal{G}) =  \nu^{*}(\mathcal{G}, w^{l}, w^{u}, c^{l}, c^{u}, p_{0})$, which is obtained by solving \eqref{eq:weight-control-original} and (\ref{eq:weight-control-nu}) with weight dependent capacities.
\end{description}
\end{problem}  

Figure~\ref{fig:net-reduction} provides an illustration for a sample network, where the process of replacing $\mc G_2^{(0)}$ with an equivalent link in $\mc G^{(1)}_1:=\tilde{\mc G}^{0}_1$ corresponds to solving the reduction problem $\mathbf{P}_r$ with input comprising of weights and capacities bounds for links associated with $\mc G_2^{(0)}$, and the terminal problem corresponds to the weight control problem $\mathbf{P}$ for \(\mathcal{G}^{\mathrm{T}} =\mathcal{G}^{(3)}\). The formal description of the multilevel programming formulation in terms of recursive solution to reduction problems and solution to the terminal problem is provided in Algorithm~\ref{alg:multilevel}.  
   

\begin{algorithm}[htb!]
\Input {network \( \mathcal{G} = (\mc V, \mc E) \) with link weights bounds $w^l \in \real_{>0}^{\mc E}$ and $w^u \in \real_{>0}^{\mc E}$, link capacity bounds $c^l \in \real_{<0}^{\mc E}$ and $c^u \in \real_{>0}^{\mc E}$, and supply-demand vector $ p_{0} \in \real^{ \mathcal{V}} $ }
\Output{ margin of robustness \( \nu^{*}( \mathcal{G}) \) }
\textbf{initialization}: $ k = 0 $, $ \mathcal{G}^{(0)} = \mc G $, $ \left(w^{l}\right)^{(0)} = w^{l} $, $ \left(w^{u}\right)^{(0)} = w^{u} $, $ \left(c^{l}(\cdot) \right)^{(0)} \equiv c^{l} $, $ \left(c^{u}(\cdot) \right)^{(0)} \equiv c^{u} $ \;
\While{ $\mathcal{G}^{(k)} $ is reducible  under $ p_{0} $ about $ v_{1}^{(k)}, v_{2}^{(k)} \in \mathcal{V}$ }{
 implement network decomposition and obtain subnetworks $ \mathcal{G}_{1}^{(k)} = ( \mathcal{V}^{(k)}, \mathcal{E}_{1}^{(k)} ) $  and $ \mathcal{G}_{2}^{(k)} = ( \mathcal{V}_{2}^{(k)}, \mathcal{E}_{2}^{(k)} ) $ such that  $ \mathcal{G}^{(k)} = \mathcal{G}^{(k)}_{1} \cup \mathcal{G}^{(k)}_{2} $\;
  solve $ \mathbf{P}_{r} $ with input $ \left( \mathcal{G}^{(k)}_{2}, \left(w^{l}\right)^{(k)}, \left(w^{u}\right)^{(k)},  \left(c^{l}\right)^{(k)}, \left(c^{u}\right)^{(k)}, v_{1}^{(k)}, v_{2}^{(k)} \right) $ and obtain output $ \left( (w^{l}_{v_{1}^{(k)}v_{2}^{(k)}})^{(k+1)}, (w^{u}_{v_{1}^{(k)}v_{2}^{(k)}})^{(k+1)},  (c^{l}_{v_{1}^{(k)}v_{2}^{(k)}})^{(k+1)},  (c^{u}_{v_{1}^{(k)}v_{2}^{(k)}})^{(k+1)} \right) $  \;
   $ \mathcal{E}^{(k+1)} =  \mathcal{E}_{1}^{(k)}\cup (v_{1}^{(k)}, v_{2}^{(k)}) $, $ \mathcal{G}^{(k+1)} = \left( \mathcal{V}_{1}^{(k)}, \mathcal{E}^{(k+1)}  \right) $, $ \left(w_{i}^{s}\right)^{(k+1)} = (w_{i}^{s})^{(k)} $, $ \left(c_{i}^{s}\right)^{(k+1)} = (c_{i}^{s})^{(k)} $ for all  $ i \in  \mathcal{E}^{(k)} $ and  $ s \in \{ l, u\} $ \;
   $ k = k+1 $\;
   }
solve $ \mathbf{P} $ and obtain $ \nu^{*}( \mathcal{G}) = \nu^{*}\left( \mathcal{G}^{(k)}, \left(w_{ \mathcal{E}^{(k)}}^{l}\right)^{(k)}, \left(w_{ \mathcal{E}^{(k)} }^{u}\right)^{(k)},  \left(c_{ \mathcal{E}^{(k)}}^{l}\right)^{(k)}, \left(c_{ \mathcal{E}^{(k)} }^{u}\right)^{(k)}, p_{0} \right) $ \;
\Return{$\nu^{*}( \mathcal{G}) $}
\caption{ Multilevel programming formulation. \label{alg:multi-level-formulation}}
\label{alg:multilevel}
\end{algorithm}

%% file: tree-reducible-net.tex
\section{An Efficient Solution Methodology for the Multilevel Programming Formulation}
\label{sec:multilevel}
\label{sec:tree-reducible-net}

In this section, we show that the two types of problems in the multilevel formulation for the weight control problem (\ie, reduction and terminal problems) can be solved explicitly for \emph{tree reducible networks}.

\subsection{Tree reducible network}

\begin{definition}[Tree reducible network]
\label{def:tree-reducible-net} 
A network with directed multigraph $\mc G=(\mc V, \mc E)$ and supply-demand vector $p \in \real^{\mc V}$ is called tree reducible (see also \cite{gitler2011terminal}) if there exists a sequence consisting of the following three operations through which the undirected graph $\mc G^{u}=(\mc V, \mc E^u)$ corresponding to $\mc G$ can be reduced to a tree\footnote{An undirected graph is called a tree if any two nodes are connected by at most one path.}:
\begin{enumerate}
\item Degree-one reduction: delete a degree\footnote{In an undirected graph, degree of a node is
    equal to the number of links incident on it.} one vertex with $p_v=0$ and its incident edge. 
\item Series reduction: delete a degree two vertex $ v_{2} $ and its two
incident edges $ \{v_{1}, v_{2}\} $ and $ \{v_{2}, v_{3}\} $, and add a new edge $
\{v_{1}, v_{3}\} $.
\item Parallel reduction: if a node pair has multiple, \ie, two or more, links between them, then remove one of those links. 
\end{enumerate}
In particular, if the terminal network produced from the above three reduction operations contains only one link, then we call the original network $\mc G$ link reducible.
\end{definition}

Same as the definition of reducible network (cf. Definition \ref{def:reducible-net} and Remark \ref{def:reducible-net}), the definition of a tree reducible network involves 
conditions on the graph topology as well as the locations of supply and demands nodes. For example, a network consisting of the graph in 
Fig. \ref{fig:tree-reducible-net} is tree reducible if the supply
and demand nodes only include \( v_{1} \) and \( v_{4} \), while it is not
tree reducible if \( v_{1} \) and \( v_{2} \) are the supply nodes and
\( v_{4} \) is the demand node. 
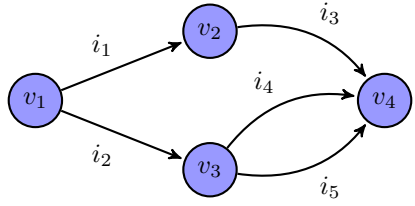
\begin{figure}[htbp!]
\begin{center}
\begin{tikzpicture}[->,>=stealth',shorten >=1pt,auto,node distance=0.4cm and 1.8cm,
                   thick,main node/.style={circle,draw,font=\sffamily\bfseries, fill=blue!40}]
 \node[main node] (1) {$v_1$};
 \node[main node] (2) [above right =of 1] {$v_2$};
 \node[main node] (3) [below right =of 1] {$v_3$};
 \node[main node] (4) [below right =of 2] {$v_4$};

 \path[every node/.style={font=\sffamily}]
   (1) edge node{$i_{1}$} (2)
   (1) edge node[below left]{$i_{2}$} (3)
   (2) edge[bend left] node{\( i_{3} \)} (4)
   (3) edge[bend left] node {\( i_{4} \)} (4)
   (3) edge[bend right] node[below right] {\( i_{5} \)} (4);
\end{tikzpicture}
\caption{A candidate graph topology for tree reducible network}
\label{fig:tree-reducible-net}
\end{center}
\end{figure}

It is straightforward to see, \eg, as in Remark \ref{rem:tree-flow}, that, for a network with tree topology, the link flows are independent of link weights. The next result shows that, for a tree reducible network, the link flow \emph{directions} are independent of link weights.
%
\begin{lemma}
\label{lem:tree-reducible-net-flow-direction}
For a tree reducible network consisting of directed multigraph \( \mathcal{G}=(\mathcal{V}, \mathcal{E}) \) and supply-demand vector \( p \in \real^{\mc V}\), 
\begin{equation*}
\sign \left(f_{i}( w, p) \right) = \sign\left(f_{i}(\tilde{w}, p)\right)
\quad \forall  i\in \mathcal{E} , 
w, \tilde{w} \in \real^{\mathcal{E}}_{> 0}
\end{equation*}
\end{lemma}
\begin{proof}
 It is clear that the above result holds for a tree, as a special case of tree
 reducible networks. For a general tree reducible network, the result follows from invariance of
 flow direction in the three operations in the definition of tree reducible networks. In degree-one
 reduction, the link removed has flow equal zero. In series reduction, \(
 \sign\left(f_{v_{1}v_{2}}(w, p) \right) = \sign\left( f_{v_{2}v_{3}}(w, p) \right) = \sign\left(
   f_{v_{1}v_{3}}(w, p)\right) \) for all \( w > 0 \). In parallel reduction, the removed link
 has the same direction of flow as the remaining links. 
\end{proof}
\begin{remark}
\label{rem:tree-reducible-net-positive-flow}
Lemma \ref{lem:tree-reducible-net-flow-direction} implies that, for a tree reducible network with a given supply-demand vector $p \in \real^{\mc V}$, one can choose direction convention for links such that \( f(w, p)\ge 0 \) for
  all \( w>0 \). We implicitly adopt this convention for the rest of this section\footnote{We emphasize that the lower and upper capacities $c^l$ and $c^u$, respectively, are defined with respect to chosen direction convention. 
%
%
%
%
    }.
\end{remark}

Recall from Section~\ref{subsec:multilevel} that a reduction problem in the
multilevel formulation is (an equality constrained) weight control problem for a subnetwork of the original network. Since the original network is assumed to be tree reducible, this subnetwork is link reducible. 
Therefore, Remark
\ref{rem:tree-reducible-net-positive-flow} implies that the reduction problem for the network, \ie, a problem of the kind \eqref{eq:eqv-cap-def}, can be simplified as 
\begin{equation}
  \label{opt:tree-reduction-problem-simplified}
  \begin{array}{@{}rc@{\quad}l@{\qquad}rc@{\quad}l@{}}
\mathcal{C}^{l}(\weq ) =  & \underset{z \in \real, w\in \real^{ \mathcal{E}}}{\text{min}} &  z  &
\mathcal{C}^{u}(\weq) = & \underset{z \in \real, w\in \real^{ \mathcal{E}}}{\text{max}} & z \\
    & \text{subject to} &  w^{l} \le w \le w^{u}   & &  \text{subject to} &  w^{l} \le w \le w^{u}  \\
    & & \mathcal{H}(w, \mathcal{G}, v_{1}, v_{2}) = \weq &&& \mathcal{H}(w, \mathcal{G}, v_{1}, v_{2})  = \weq \\ 
    && zf(w, a_{v_{1}v_{2}}) \ge c^{l}   &&& zf(w, a_{v_{1}v_{2}}) \le c^{u} \\
  \end{array}
\end{equation}
where $ \mathcal{G} = (\mathcal{V}, \mathcal{E}) $ is the network's underlying graph and \( \weq \in \range( \mathcal{H}(w, \mathcal{G}, v_{1}, v_{2}))  \). By setting \( z' := -z \) in the problem for \( \mathcal{C}^{l}(\weq) \), it is straightforward to
see that it is the same problem as that for \( \mathcal{C}^{u}(\weq) \). Setting \( c:= -c^{l} \) for \(
\mathcal{C}^{l}(\weq)\), and \( c:=c^{u} \) for \( \mathcal{C}^{u}(\weq) \),
the two problem instances in (\ref{opt:tree-reduction-problem-simplified}) can be uniformly written as follows. 
\begin{equation}
  \label{opt:tree-reduction-problem-uniform}
  \begin{aligned}
 \mathcal{C}(\weq) =  & \underset{z \in \real, w \in \real^{ \mathcal{E}}}{\text{max}} & & z \\
    & \text{subject to}
    & & w^{l}_{i} \le w \le w^{u}_{i} \\
    &&& z \le \frac{c_{i}(w_{i})}{f_{i}(w, a_{v_{1}v_{2}})} \quad \forall \, i \in \mathcal{E} \\
    &&&\mathcal{H}(w, \mathcal{G}, v_{1}, v_{2})  = \weq
  \end{aligned}
\end{equation}

We begin by focusing on solving the following \emph{simplified version of the reduction problem}~\eqref{opt:tree-reduction-problem-uniform}:
%
\begin{equation}
  \label{opt:simple-problem-class-simplified}
  \begin{aligned}
  g(\weq) = & \underset{\alpha \in \real, w \in \real^{\mc E}}{\max} & & \alpha \\
    & \text{subject to}
    & & w_{i}^{l} \le w_{i} \le w_{i}^{u} \quad \forall \, i \in \mc E \\
    &&& \alpha \le \psi_{i}(w_{i}) \quad \forall \, i \in \mc E \\
    &&& \mathcal{H}(w, \mathcal{G}, v_{1}, v_{2})  = \weq 
  \end{aligned}
\end{equation}
for given $\weq \in \range(\mathcal{H}(w, \mathcal{G}, v_{1}, v_{2}) ) $ and functions $\map{\psi_i}{\real_{>0}}{\real_{>0}}$,
$i \in \mc E$, representing the second set of inequalities in \eqref{opt:tree-reduction-problem-uniform}. Note that the second set of inequalities in
\eqref{opt:simple-problem-class-simplified} are separable across links, whereas they are not in
\eqref{opt:tree-reduction-problem-uniform}. This \emph{simplification} will be shown to be lossless. We shall
then devise a methodology that sequentially uses solution to
\eqref{opt:simple-problem-class-simplified} for parallel and serial networks, to obtain an iterative scheme to solve (\ref{opt:tree-reduction-problem-uniform}). 


\subsection{Input-output Properties of the Simplified Version of the Reduction Problem}
\label{sec:simple-prob-class}
In order to develop the sequential procedure, we interpret \eqref{opt:simple-problem-class-simplified} to be defining an \emph{output} function $g(\weq)$ with link level functions $\psi_i(w_i)$, $i \in \mc E$ as \emph{input}. We next introduce a property which will be shown to be invariant from the input functions to the output function, and will be helpful to compute the function $g(\weq)$ specified by \eqref{opt:simple-problem-class-simplified}.


%
%
\begin{definition}[$\mc S_0$ function]
  \label{def:function-class-0}
  A function $\map{\psi}{[x^l,x^u] \subset \real}{\real}$ is called a $\mc S_0$ function if it is continuous, and there exist $\ubar{x} \in [x^l,x^u]$ and $\bar{x} \in [\ubar{x},x^u]$ such that $\psi(x)$ is strictly increasing over $[x^l,\ubar{x}]$, constant over $[\ubar{x},\bar{x}]$, and strictly decreasing over $[\bar{x},x^{u}]$. We shall sometimes refer to $\ubar{x}$ and $\bar{x}$ as first and second transition points (w.r.t. $\mc S_0$ property), respectively, of $\psi(x)$.
\end{definition}

Figure~\ref{fig:S0-function} provides an example of a $\mc S_0$ function. It is easy to see that a $\mc S_0$ function is also quasiconcave, but the converse is not true in general.

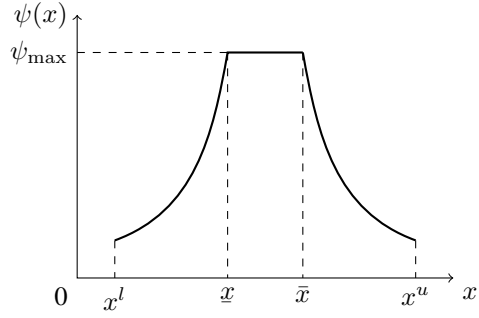
\begin{figure}[ht]
 \centering
\begin{tikzpicture} 
\draw [<->] (0,3.5) node [left] {\( \psi(x) \)}--
(0,0) node [below left] {0} -- (5,0) node [below right] {\( x \)}; 
\draw[thick] (0.5, 0.5) to [out=20, in=260] (2, 3) to [out=0, in= 180] (3, 3);
\draw[thick] (4.5, 0.5) [out = 160, in = -80] to (3, 3);
\draw[dashed] (0.5, 0.5) to (0.5, 0) node[below]{ \( x^{l} \)}; 
\draw[dashed] (2, 3) to(2, 0) node[below]{ \( \ubar{x} \)}; 
\draw[dashed] (3, 3) to (3,0) node[below] {\( \bar{x} \) }; 
\draw[dashed] (4.5, 0.5) to (4.5, 0) node[below] {\( x^{u} \)};
\draw[dashed] (2, 3) to (0, 3);
\node[left] at (0, 3) {\( \psi_{\mathrm{max}} \)};
\end{tikzpicture}
  \caption{A sample \( \mathcal{S}_{0} \) function.}
  \label{fig:S0-function}
\end{figure}



\begin{proposition}
\label{prop:So-invariance}
Consider a network consisting of graph topology \( \mathcal{G}=(\mathcal{V} ,\mathcal{E}) \), lower and upper bounds on link weights  $w^l \in \real_{>0}^{\mc E}$ and $w^u \in \real_{>0}^{\mc E}$ respectively and supply and demand node $ v_{1}, v_{2} \in \mathcal{V}  $ respectively. If the equivalent weight function $\mathcal{H}(w, \mathcal{G}, v_{1}, v_{2}) $ is strictly monotone with respect to $ w $ for this network, and $\psi_i(w_i)$ is a $\mc S_0$ function for all $i \in \mc E$, then the $g(\weq)$ function defined by \eqref{opt:simple-problem-class-simplified} is also a $\mc S_0$ function. 
%
\end{proposition}
\begin{proof}
In general, $\mc H(w)$ is not one-to-one, \ie, there could exist $\tilde{w}, \tilde{\tilde{w}} \in [w^l,w^u]$, $\tilde{w} \neq \tilde{\tilde{w}}$, such that $\mc H(\tilde{w})=\mc H(\tilde{\tilde{w}})$. However, the strict monotonicity of $\mc H(w)$ implies that the only feasible points of \eqref{opt:simple-problem-class-simplified} for \( \weqlow := \mc H(w^{l})\) and \( \wequp := \mc H(w^{u}) \) are \(( \min_{i \in \mc E} \psi_{i}(w_{i}^{l}), w^{l} ) \) and \( (\min_{i \in \mc E} \psi_{i}(w_{i}^{u}), w^{u}) \) respectively, and that $\weq \in [\weqlow, \wequp]$ for all $w \in [w^l,w^u]$. Hence $ g(\weq^{l}) =  \min_{i \in \mc E} \psi_{i}(w_{i}^{l}) $ and $ g(\weq^{u}) =  \min_{i \in \mc E} \psi_{i}(w_{i}^{u}) $. 
Let $\gmax:=\max_{\weq \in [\weqlow,\wequp]} g(\weq)$, then $g(\weqlow)=:g^l \leq g_{max}$ and  $g(\wequp) =:g^u \leq g_{max}$. Motivated by this, and with the objective of ultimately proving $\mc S_0$ property of $g(\weq)$, we construct inverse functions of $g(\weq)$ over $[g^l,\gmax]$ and $[g^u,\gmax]$. We denote these inverse functions as $\map{\hat{g}^+}{[g^l,\gmax]}{[\weqlow,\wequp]}$ and $\map{\hat{g}^-}{[g^u,\gmax]}{[\weqlow,\wequp]}$, respectively. We construct these inverses as compositions: 
\begin{equation}
\label{eq:g-inv-def}
\hat{g}^+(x)=\mc H \circ \omega^+(x) \qquad \hat{g}^-(x)=\mc H \circ \omega^-(x)
\end{equation}
where $\mc H$ is the equivalent weight function from \eqref{eq:equivalent-weight-def}, and $\map{\omega^+}{[g^l,\gmax]}{[w^l,w^u]}$ and $\map{\omega^-}{[g^u,\gmax]}{[w^l,w^u]}$ are defined as: for all $i \in \mc E$,
\begin{equation}
\label{eq:w-inv-def}
\begin{split}
 \omega_{i}^{+}(x) &:= \left\{\begin{array}{ll}
      w_{i}^{l} \quad& \text{if } x \le \psi_{i}(w_{i}^{l}) \\
      \min \{w_{i}: \psi_{i}(w_{i})=x\} \quad& \text{if } x > \psi_{i}(w_{i}^{l})
                 \end{array} \right.   \\
 \omega_{i}^{-}(x) &:=  \left\{\begin{array}{ll}
      w_{i}^{u} \quad& \text{if } x \le  \psi_{i}(w_{i}^{u}) \\
      \max \{w_{i}: \psi_{i}(w_{i})=x\} \quad& \text{if } x > \psi_{i}(w_{i}^{u})
                 \end{array} \right. 
                 \end{split}
                 \end{equation}
It is easy to see that 
\begin{equation}
\label{eq:gmax-expr}
\gmax = \min_{i \in \mc E} \, \max_{w_i \in [w_i^l, w_i^u]} \psi_i(w_i)
\end{equation}
Combining \eqref{eq:gmax-expr} with the fact that $\psi_i$ is a $\mc S_0$ function for all $i \in \mc E$, the definitions in \eqref{eq:w-inv-def} imply that, for all $i \in \mc E$, 
\begin{equation}
\label{eq:w-inv-prop}
\begin{split}
\omega_i^+(x) \in [w_i^l,\ubar{w}_i] \subseteq [w^l_i,w^u_i] \quad \& \quad & x \leq \psi_{i}({\omega^+_i(x)}), \qquad \forall \, x \in [g^l,\gmax] \\
\omega_i^-(x) \in [\bar{w}_i,w_i^u] \subseteq [w_i^l,w_i^u] \quad \& \quad & x \leq \psi_{i}({\omega_i^-(x)}), \qquad \forall \, x \in [g^u,\gmax]
\end{split}
\end{equation}
where we refer to Definition~\ref{def:function-class-0} for notations $\ubar{w}$ and $\bar{w}$. Moreover, since \( \psi_{i}(w_{i}) \in \mathcal{S}_{0} \), for all $i \in \mc E$, 
\( \omega_{i}^{+} \) is nondecreasing and \( \omega_{i}^{-}
\) is nonincreasing, and, it is easy to see that, for every $x \in [g^l,\gmax]$, there exists at least one $i \in \mc E$ such that \( \omega_{i}^{+}(x) \) is strictly increasing, and that, for every $x \in [g^u,\gmax]$, there exists at least one $i \in \mc E$ such that \( \omega_{i}^{-}(x) \) is strictly decreasing. This combined with the strictly  increasing property of $\mc H(w)$ implies that $\map{\hat{g}^+}{[g^l,\gmax]}{[\weqlow,\wequp]}$ and $\map{\hat{g}^-}{[g^u,\gmax]}{[\weqlow,\wequp]}$ are strictly increasing and strictly decreasing bijections, respectively. Moreover, it is easy to see that $\weqlow \leq \hat{g}^+(\gmax) \leq \hat{g}^-(\gmax) \leq \wequp$, where the middle inequality follows from \eqref{eq:g-inv-def}, \eqref{eq:w-inv-def}, and the strict monotonicity of $\mc H$.


In the remainder of the proof, our strategy for proving that $g(\weq)$ is a $\mc S_0$ function is as follows: we show that (i) $\hat{g}^+$ is the inverse of $g(\weq)$ over $\weq \in [\weqlow, \hat{g}^+(\gmax)]$, (ii) $\hat{g}^-$ is the inverse of $g(\weq)$ over $\weq \in [\hat{g}^-(\gmax),\wequp]$, and (iii) $g(\weq) \equiv \gmax$ over $\weq \in [\hat{g}^+(\gmax),\hat{g}^-(\gmax)]$. In particular, $\hat{g}^+(\gmax)$ and $\hat{g}^-(\gmax)$ will play the role of $\ubar{x}$ and $\bar{x}$ (cf. Definition~\ref{def:function-class-0}) in proving that $g(\weq)$ is a $\mc S_0$ function. The proof for (i) and (ii) are similar, and hence we provide details only for (i).

%

In order to show that $\hat{g}^+$ is the inverse of $g(\weq)$ over $\weq \in [\weqlow, \hat{g}^+(\gmax)]$, we show that $g(\hat{g}^+(x))=x$ for all $x \in [g^l,\gmax]$. In order to show this, we show that, for all $x \in [g^l,\gmax]$, $(x,\omega^+(x))$ is the unique optimizer for (\ref{opt:simple-problem-class-simplified}) corresponding to $\weq=\hat{g}^+(x)$. \eqref{eq:g-inv-def} and \eqref{eq:w-inv-prop} readily imply that $(x,\omega^+(x))$ is feasible for (\ref{opt:simple-problem-class-simplified}). Therefore, for all $x \in [g^l,\gmax]$,
\begin{equation}
\label{eq:ineq1}
g(\hat{g}^+(x))\geq x
\end{equation}
Consider an arbitrary $ \tilde{w} \in [w^{l}, w^{u}]  $ such that $ \tilde{w} \neq \omega^+(x) $ and $ \mc H(\tilde{w})=\hat{g}^+(x) = \mc H(\omega^+(x)) $. It is sufficient to show that $ \tilde{\alpha} < x $ for all $ \tilde{ \alpha} $ such that $ (\tilde{\alpha}, \tilde{w}) $ is feasible to \eqref{opt:simple-problem-class-simplified}. For $ x = g^{l} $, by definition $ \omega^{+}(x) = w^{l}$ and $ \mathcal{H}(\omega^{+}(x)) = \weq^{l}$. Strict monotonicity of $\mc H$ implies that $(x,\omega^+(x))$ is the only feasible point and hence the unique optimizer of (\ref{opt:simple-problem-class-simplified}). For all $ x \in ( g^{l}, \gmax ] $\footnote{ It is possible that $ \gmax =g^l $. In this case, considering the case $ x= g^{l} $ is sufficient.}, it is clear from the definition of $ g^{l} $ and $ \gmax $ that the set $ \{ i \in \mc E\, | \, x > \psi_i(w_i^l) \} $ is not empty. Since $ \tilde{w}_{k} \ge w_{k}^{l} = w_{k}^{+}(x) $ for all $ k\in \{i\in \mathcal{E} \, | \, x \le \psi_{i} \} $, if $ \tilde{w}_{k} \ge \omega^{+}(x) $ for all $k \in  \{ i \in \mc E\, | \, x >\psi_i(w_i^l)  \} $,   strict monotonicity of $\mc H$ implies $ \mathcal{H}(\tilde{w})> \mc H(\omega^+(x))  $. That is to say, in order to satisfy $ H(\tilde{w})= \mc H(\omega^+(x)) $ and $ \tilde{w} \neq \omega^+(x) $, there is at least one $k \in  \{ i \in \mc E\, | \, x > \psi_i(w_i^l) \} $ such that $\tilde{w}_k < \omega^+_k(x) \leq \ubar{w}_k$. Using this along with the fact that $\psi_k$ is a $\mc S_0$ function, and hence $ \psi_{k} $ is strictly increasing in $ [w_{k}^{l}, \ubar{w}_{k}] $, we get $\psi_k(\tilde{w}_k) < \psi_k(\omega^+_k(x))=x$, where the equality is due to the implication of $\psi_k(w_k^l)<x$ in \eqref{eq:w-inv-def}. Therefore, the last inequality constraint in \eqref{opt:simple-problem-class-simplified} implies that $\tilde{ \alpha} < x $ for all feasible $ \tilde{\alpha} $. In other words, $ (g(\weq ), \omega^{+}(g(\weq))) $ is the unique solution to (\ref{opt:simple-problem-class-simplified}) for all $ \weq \in [\weqlow, \hat{g}^+(\gmax)] $. Similar result is true for $ \weq \in [\hat{g}^{-}(\gmax), \weq^{u}] $.

Recall that $\gmax$ is the maximum value of $g(\weq)$ over all $\weq$. Therefore, in order to show that \( g(\weq) \equiv \gmax \) for all $\weq \in [\hat{g}^+(\gmax),\hat{g}^-(\gmax)]$, it suffices to show that, for every $\weq \in [\hat{g}^+(\gmax),\hat{g}^-(\gmax)]$, there exists a $\tilde{w} \in [w^l, w^u]$ such that $(\gmax,\tilde{w})$ is feasible for (\ref{opt:simple-problem-class-simplified}). Since, by definition in \eqref{eq:g-inv-def}, $\mc H(\omega^+(\gmax))=\hat{g}^+(\gmax)$ and $\mc H(\omega^-(\gmax))=\hat{g}^-(\gmax)$, continuity and monotonicity of $\mc H$ implies that, for all $\weq \in [\hat{g}^+(\gmax),\hat{g}^-(\gmax)]$, there exists $\tilde{w} \in [\omega^+(\gmax),\omega^-(\gmax)]$ satisfying $\mc H(\tilde{w})=\weq$. 
Moreover, the $ \mathcal{S}_{0} $ property of $ \psi_{i} $ implies that $ \gmax \le \psi_{i}(w_{i}) $ for all $ w_{i} \in   \in [\omega_{i}^+(\gmax),\omega_{i}^-(\gmax)] $. 
This shows that $(\gmax,\tilde{w})$ is feasible for (\ref{opt:simple-problem-class-simplified}).


Finally, the continuity of $g(\weq)$ follows from the continuity of the inverse functions $\hat{g}^+$ and $\hat{g}^-$, which in turn follows from the continuity of $\mc H$ from \eqref{eq:equivalent-weight-def}, and continuity of $\omega^+$ and $\omega^-$ from \eqref{eq:w-inv-def} implied by the continuity of $\psi_i$'s being $\mc S_0$ functions. 
\end{proof}

The solution to (\ref{opt:simple-problem-class-simplified}) is not unique in general for an arbitrary $ \weq $. However, it is unique for $ \weq $ within a certain range, as shown in the above proof and summarized in Remark \ref{rem:unique-sol-general}.
\begin{remark}
\label{rem:unique-sol-general}
\leavevmode
\begin{enumerate}
\item [(a)] (\ref{opt:simple-problem-class-simplified}) has unique solution $ (g(\weq ), \omega^{+}(g(\weq))) $ and  $ (g(\weq ), \omega^{-}(g(\weq))) $ for any $ \weq \in [\weqlow, \hat{g}^+(\gmax)] $ and $ \weq \in [\hat{g}^{-}(\gmax), \weq^{u}] $, respectively. 
\item [(b)] $ \omega^{+}(g(\weq)) $ is nondecreasing \wrt $ \weq $ for $ \weq \in  [\weqlow, \hat{g}^+(\gmax)]  $, since by definition $ \omega^{+} $ is nondecreasing function and $ \mathcal{S}_{0} $ property of $ g(\weq) $ implies that $ g(\weq) $ is strictly increasing for $ \weq \in  [\weqlow, \hat{g}^+(\gmax)] $. $ \omega^{-}(g(\weq)) $ is nondecreasing \wrt $ \weq $ for $  \weq \in [\hat{g}^{-}(\gmax), \weq^{u}]   $ due to similar reason. 
 \end{enumerate}
\end{remark}
The proof of Proposition~\ref{prop:So-invariance} implies that the solution to \eqref{opt:simple-problem-class-simplified} is given by:
\begin{equation}
  \label{eq:fun-g-solution}
  g(\weq) = \left\{ 
\begin{array}{ll}
 \inv \, {\hat{g}^{+}}(\weq)  &\quad \weqlow \le \weq < \hat{g}^+(\gmax)  \\
 \gmax  &\quad \hat{g}^+(\gmax) \le \weq \le \hat{g}^-(\gmax)  \\
  \inv \, {\hat{g}^{-}}(\weq) & \quad \hat{g}^-(\gmax) < \weq \le \wequp
\end{array}
 \right.
\end{equation}
where $\inv \, {\hat{g}^{+}}$ and $\inv \, {\hat{g}^{-}}$ are the inverses of $\hat{g}^{+}$ and $\hat{g}^{-}$, respectively, as defined in \eqref{eq:g-inv-def}, $\gmax$ is defined in \eqref{eq:gmax-expr}.
Proposition~\ref{prop:So-invariance} implies that $g$ is continuous. However, it may not be differentiable in general.
Let
\begin{equation}
\label{eq:der-def}
 g'(\weqminus) := \lim_{\triangle w_{eq} \uparrow 0 }
  \frac{ g(w_{eq}+\triangle w_{eq}) - g(w_{eq})}{\triangle w_{eq}}, \qquad 
  g'(\weqplus) := \lim_{\triangle w_{eq} \downarrow 0 }
  \frac{g(w_{eq}+\triangle w_{eq}) - g(w_{eq})}{\triangle w_{eq}} 
\end{equation}
be the left and right derivatives, respectively. 
We provide derivation for explicit expressions of these derivatives in the appendix. These expressions are used in Sections~\ref{sec:series} and \ref{sec:parallel} to provide an explicit solution for series and parallel networks. 

\subsection{Series Networks}
\label{sec:series}
 In a series network, $|\mc V|=|\mc E|+1$. A series network consisting of three links is shown in Fig. \ref{fig:3-link-chain}. 
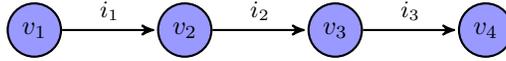
\begin{figure}[htbp]
\begin{center}
\begin{tikzpicture}[->,>=stealth',shorten >=1pt,auto,node distance=2cm,
thick,main node/.style={circle,draw,font=\sffamily\bfseries, fill=blue!40}]
\node[main node] (1) {$v_1$}; 
\node[main node] (2) [right of=1] {$v_2$}; 
\node[main
node] (3) [right of=2] {$v_3$}; 
\node[main node] (4) [right of=3] {$v_4$};
  
  \path[every node/.style={font=\sffamily\small}] 
(1) edge node{$ i_{1} $} (2)
(2) edge node{$i_{2}$} (3) (3) edge node{$i_{3}$} (4) ;
\end{tikzpicture}
\caption{A three link series network}
\label{fig:3-link-chain}
\end{center}
\end{figure}

Consider a series network with $n+1$ nodes numbered $v_1, \ldots, v_{n+1}$ such that $(v_j,v_{j+1}) \in \mc E$ for all $j \in \until{n}$, and link weights $w \in \real_{>0}^{n}$. As already shown in Example~\ref{ex:eq-weight}, the equivalent weight function between $v_1$ and $v_{n+1}$ is given by \( \mc H(w) = \sum_{i=1}^{n} \left(  1/w_{i}\right)^{-1} \). Moreover, the flow on any link $ i \in \until{n} $  is equal to one when a unit flow enters at node $v_1$ and leaves at $v_{n+1}$, \ie, $ f_{i}(w, a_{v_{1}v_{n}}) = 1 $. Therefore, \eqref{opt:tree-reduction-problem-uniform} can be simplified for a series network as (\ref{opt:equiv-cap-chain}), which gives the equivalent capacity function between nodes $v_1$ and $v_{n+1}$. 
\begin{equation}
  \label{opt:equiv-cap-chain}
  \begin{aligned}
  \mc C(\weq) = & \underset{z \in \real, w \in \real_{>0}^n}{\max} & & z \\
    & \text{subject to}
    & & w_{i}^{l}\le w_{i}\le w_{i}^{u} \\
    &&& z \le c_{i}(w_{i}), \quad i \in \until{n} \\
    &&& \left(\sum_{i=1}^{n} \frac{1}{w_{i}}\right)^{-1} =\weq 
  \end{aligned}
\end{equation}
For constant link capacities, \ie, $c_i(w_{i}) \equiv c_{i}$, $i \in \mc E$, then it is easily to see that  \( \mc C(\weq) = \min_{i \in \until{n}} c_{i} \). For weight-dependent capacities, 
we now establish a functional property of $\mc C(\weq)$, which is a stronger version of the $\mc S_0$ property defined in Definition~\ref{def:function-class-0}.

\begin{definition}
\label{def:fun-S-1}
A function $\map{\psi}{[x^l,x^u] \subset \real_{>0}}{\real}$ is called a $\mc S_1$ function if it is a $\mc S_0$ function (cf. Definition~\ref{def:function-class-0}), and if there exists a $x^o \in [x^l,\ubar{x}]$ such that \(
  \min \partial \psi(x) > \psi(x)/x \) for all \(x \in \left[x^{l}, x^{o}\right)\) and  \(
  \min \partial \psi(x) = \psi(x)/x \) for all \(x \in  \left[x^{o}, \ubar{x}\right)\),
 where $\ubar{x}$ is the first transition point, w.r.t. $\mc S_0$ property,  \( \partial \psi(x) \) denotes the set of  subgradients of  \(
  \psi(x) \). We shall sometimes refer to $x^o$ and $\ubar{x}$ as first and second transition points (w.r.t. $\mc S_1$ property), respectively, of $\psi(x)$.
\end{definition}
\begin{remark}
Note that, in Definition~\ref{def:fun-S-1}, we allow $x^o = \ubar{x}$, in which case, the only requirement for a $\mc S_0$ function to be $\mc S_1$ is that \(
  \min \partial \psi(x) > \psi(x)/x \) for all \(x \in \left[x^{l}, \ubar{x}\right)\).
\end{remark}

Definition~\ref{def:fun-S-1} clearly implies that, if $\psi(x)$ is a $\mc S_1$ function, then it is also a $\mc S_0$ function. The next result extends the $\mc S_0$ implication also to $\psi(x)/x$.  

\begin{lemma}
  \label{lem:S0-S1-relation}
  If $\map{\psi}{[x^l,x^u] \subset \real_{>0}}{\real_{>0}}$ is a $\mc S_1$ function, then $\map{\psi(x)/x}{[x^l,x^u] \subset \real_{>0}}{\real_{>0}}$ is a $\mc S_0$ function.
\end{lemma}
\begin{proof}
  Let \( \tilde{\psi}(x) := \psi(x)/x \). The continuity of $\tilde{\psi}(x)$ follows from that of $\psi(x)$.
  Then, the left and right derivative of $\tilde{\psi}(x)$ are, respectively, given by:
\begin{equation}
\label{eq:subgradient-expr}
\tilde{\psi}'(x^{-}) = \frac{\psi'(x^{-})x - \psi(x)}{x^{2}}, \quad
\tilde{\psi}'(x^{+}) = \frac{\psi'(x^{+})x -  \psi(x)}{x^{2}} 
\end{equation}
Note that these two derivatives completely specify the set of subgradients of $\tilde{\psi}(x)$. Since $\psi(x)$ is a $\mc S_1$ function, we have \(\min \partial \psi(x)>\psi(x)/x \) for all \( x\in [x^{l}, x^{o}) \). Therefore, \eqref{eq:subgradient-expr} implies that $\tilde{\psi}'(x^-)$ and $\tilde{\psi}'(x^+)$ are both strictly positive, and hence $\tilde{\psi}(x)$ is strictly increasing over $[x^{l}, x^{o})$. For \( x \in (x^{o}, \ubar{x}) \), \eqref{eq:subgradient-expr} implies that $\tilde{\psi}'(x^-)=\tilde{\psi}'(x^+)=0$, \ie, $\tilde{\psi}(x)$ is constant. Since $\psi(x)$ is also a $\mc S_0$ function, $\psi'(x^-)$ and $\psi'(x^+)$ are both nonpositive for $x \in (\ubar{x}, x^u]$. Therefore, \eqref{eq:subgradient-expr} implies that $\tilde{\psi}'(x^-)$ and $\tilde{\psi}'(x^+)$ are both strictly negative, and hence $\tilde{\psi}(x)$ is strictly decreasing over $(\ubar{x}, x^u]$. Collecting these facts, we establish that $\tilde{\psi}(x)$ is a $\mc S_0$ function. We conclude the proof by emphasizing that the transition points required for the $\mc S_0$ property of the $\tilde{\psi}$ function are the points corresponding to $x^o$ and $\ubar{x}$ used in specifying the $\mc S_1$ property of $\psi$ (cf. Definition~\ref{def:fun-S-1}).
%
%
%
%
\end{proof}

  \begin{remark}
\label{rem:S0-S1-relation}
 The proof of Lemma~\ref{lem:S0-S1-relation} implies that the first and second transition points, w.r.t. $\mc S_1$ property, of $\psi(x)$  \ie, $ x^{o} $ and $ \ubar{x} $, are the first and second transition points, w.r.t. $\mc S_0$ property, of $\psi(x)/x$, respectively.
%
  \end{remark}

\begin{lemma}
\label{lem:equiv-cap-S1-chain}
Consider a network consisting of series graph topology \( \mathcal{G}=(\mathcal{V} ,\mathcal{E}) \), where $\mc V=\{v_1, \ldots, v_{n+1}\}$ and $\mc E=\{(v_1,v_2), \ldots, (v_{n},v_{n+1})\}$, and lower and upper bounds on link weights $w^l \in \real_{>0}^{n}$ and  $w^u \in \real_{>0}^{n}$ respectively. If the link capacity functions $c_i(w_i)$ are $\mc S_1$ for all $i \in \until{n}$, then the the equivalent capacity function between $v_1$ and $v_{n+1}$, as given by \eqref{opt:equiv-cap-chain}, is also a $\mc S_1$ function. 
\end{lemma}
%
\begin{proof}
Since $c_i$ are $\mc S_1$ functions for all $ i \in \until{n} $, by definition, they are also $\mc S_0$ functions. Therefore, Proposition~\ref{prop:So-invariance} implies that $\mc C(\weq)$, as given by \eqref{opt:equiv-cap-chain}, is also a $\mc S_0$ function. In order to prove that \( \mc C(\weq)\) is a $\mathcal{S}_{1}$ function, we need to show that
there exists \( \weq^{o} \in [\weqlow, \ubar{w}_{\mathrm{eq}}]\) such that  \(\min \partial \mc C(\weq) >
\mc C(\weq)/\weq \) for  \( \weq \in [\weqlow, \weq^{o}) \) and
\(\min \partial \mc C(\weq) = \mc C(\weq)/\weq \) for \( \weq \in 
[\weq^{o}, \ubar{w}_{\mathrm{eq}}) \).  

We now show that there exist \( \weq^{o} \in [\weqlow, \ubar{w}_{\mathrm{eq}}]\)  such that \(
\mc C'(\weqplus) > \mc C(\weq)/\weq \) for \( \weq \in [\weqlow, \weq^{o}) \) and \( \mc C'(\weqplus) = \mc C(\weq)/\weq \) for \( \weq \in 
[\weq^{o}, \ubar{w}_{\mathrm{eq}}) \). Similar results hold true for \(
\mc C'(\weqminus) \). Since \( \min \partial \mc C(\weq)
= \min \{\mc C'(\weqminus), \mc C'(\weqplus) \} \), this then completes the proof.   

Noting the expression for the equivalent weight function in \eqref{opt:equiv-cap-chain}, we get that 
\begin{align*}
\frac{\partial \mc H(w)}{\partial w_{i}}  &= \frac{1}{w_{i}^{2}} \left(\sum_{i=1}^{n}
  \frac{1}{w_{i}}\right)^{-2}  = \frac{\weq^{2}}{w_{i}^{2}}. 
\end{align*}

Substituting into (\ref{eq:lumped-fun-derivative}), we get 
\begin{align*}
 \mc C'(\weqplus) & = \left(\sum_{i\in \tilde{\mc{K}}^+(\mc C(\weq))} \frac{\weq^2}{c_i'(w_i^+) w_i^2}\right)^{-1} \Big|_{w=\omega^+(\mc C(\weq))} \geq \frac{\mc C(\weq)}{\weq^2} \left(\sum_{i\in \tilde{\mc{K}}^{+}(\mc C(\weq))} \frac{1}{w_i}\right)^{-1} \Big|_{w=\omega^+(\mc C(\weq))} \\ & \geq \frac{\mc C(\weq)}{\weq^2} \left(\frac{1}{\weq} \right)^{-1} = \frac{\mc C(\weq)}{\weq}
\end{align*}
where the first inequality follows from the fact that, since \( c_{i}(w_{i}) \in \mathcal{S}_{1} \), \( c'_{i}(w_{i}^{+}) w_{i} \ge
\min \partial c_{i}(w_{i}) w_{i} \ge c_{i}(w_{i}) \), and by definition, \(
c_{i}(\omega^+_i(\mc C(\weq))) = \mc C(\weq) \) for all \( i\in \tilde{\mc{K}}^{+}(\mc C(\weq)) \), and it is equality if and only if $ \omega_{i}(\mc C(\weq)) \ge w_{i}^{o} $ for all $ i \in  \tilde{\mc{K}}^{+}(\mc C(\weq))$. The second inequality is equality if and only if \( \tilde{\mc{K}}^{+}(\mc C(\weq)) = \mc E \). If for some \( \tilde{w}_{\mathrm{eq}} \in [\weqlow, \ubar{w}_{\mathrm{eq}}]\), both inequalities are equalities, \ie, $ \omega_{i}(\mc C(\weq)) \ge w_{i}^{o} $ for all $ i\in \mathcal{E} $ and $\tilde{\mc{K}}^{+}(\mc C(\weq)) = \mc E  $, then $ \mathcal{C}'( \weq^{+}) = \mathcal{C}(\weq)/\weq $ holds for all $ \weq \in [ \tilde{w}_{\mathrm{eq}}, \ubar{w}_{\mathrm{eq}} ] $. This is because of the nondecreasing property of function $  \omega(\mc C( \cdot)) $ (cf. Remark \ref{rem:unique-sol-general}(b)) and $\tilde{\mc{K}}^{+}(\mc C(\cdot))$ (by definition).
Therefore, there exists $\weq^o \in [\weqlow,\ubar{w}_{\mathrm{eq}}]$ such that the both the inequalities are strict for $\weq \in [\weqlow, \weq^o)$ and is equality for $\weq \in [\weq^o,\ubar{w}_{\mathrm{eq}})$.
%
\end{proof}

\subsection{Parallel Networks}
\label{sec:parallel}
We now focus on networks with parallel graph topology, \ie, when $\mc G= (\mc V, \mc E)$, where $\mc V=\{v_1,v_2\}$, and all the links in $\mc E$ are from $v_1$ to $v_2$. An example is shown in
Fig. \ref{fig:3-link-para-net}. 
\begin{figure}[htbp]
\begin{center}
\begin{tikzpicture}[->,>=stealth',shorten >=1pt,auto,node distance=3cm,
thick,main node/.style={circle,draw,font=\sffamily\bfseries, fill=blue!40}]
\node[main node] (1) {$v_1$}; 
\node[main node] (2) [right of=1] {$v_2$}; 
  
  \path[every node/.style={font=\sffamily\small}] 
(1) edge[bend left] node{$ i_{1} $} (2)
(1) edge[bend right] node{$i_{2}$} (2);
\end{tikzpicture}
\caption{A two link parallel network}
\label{fig:3-link-para-net}
\end{center}
\end{figure}
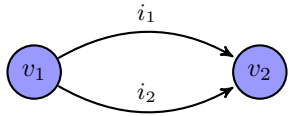

Consider a parallel network with \( n \) links from node $v_1$ to node $v_2$,
and link weights $w \in \real_{>0}^{n}$. 
As already shown in Example~\ref{ex:eq-weight}, the equivalent weight function
between $v_1$ and $v_2$ is given by $\mc H(w) = \sum_{i=1}^{n} w_i$. With unit
supply and demand on \( v_{1} \) and \( v_{2} \), the flow on link \( i \) is \(
f_{i} = w_{i}/\weq \). 
Substituting \( f_{i} = w_{i}/\weq \) into
\eqref{opt:tree-reduction-problem-uniform} and letting \( \tilde{z} = z/\weq \), 
%
the equivalent capacity function between nodes $v_1$ and $v_2$ takes the following simple form:
\begin{equation}
  \label{opt:equiv-cap-para-net}
  \begin{aligned}
  \frac{\mc C(\weq)}{\weq}=\, & \underset{\tilde{z} \in \real, w \in \real^n_{>0}}{\max} & & \tilde{z} \\
    & \text{subject to}
    & & w_{i}^{l} \le w_{i} \le w_{i}^{u} \\
    &&& \tilde{z} \le c_{i}(w_{i})/w_{i} \quad \forall\, i \in \until{n} \\
    &&& \sum_{i=1}^{n} w_{i} = \weq
  \end{aligned}
\end{equation}

The following result is the equivalent of Lemma~\ref{lem:equiv-cap-S1-chain} for parallel networks. 
\begin{lemma}
  \label{lem:equiv-cap-S1-para-net}
Consider a network consisting of parallel graph topology \( \mathcal{G}=(\mathcal{V} ,\mathcal{E}) \), where $\mc V=\{v_1, v_{2}\}$ and all the links in $\mc E$ are from $v_1$ to $v_2$, and lower and upper bounds on link weights are $w^l \in \real_{>0}^{n}$ and $w^u \in \real_{>0}^{n}$ respectively. If the link capacity functions $c_i(w_i)$ are $\mc S_1$ for all $i \in \until{n}$, then the the equivalent capacity function $\mc C(\weq)$ between $v_1$ and $v_{2}$, as given by \eqref{opt:equiv-cap-para-net}, is also a $\mc S_1$ function. 
\end{lemma}
\begin{proof}
Since $c_i(w_i)$ are $\mc S_1$ functions for all $i \in \until{n}$, Lemma~\ref{lem:S0-S1-relation} implies that $c_i(w_i)/w_i$ are $\mc S_0$ functions and Remark~\ref{rem:S0-S1-relation} implies that the second transition point of $ c_{i}(w_{i})/w_{i} $ \wrt $ \mathcal{S}_{0} $ property is the first transition point $ \ubar{w}_{i} $ of $ c_{i}(w_{i}) $ \wrt $ \mathcal{S}_{0} $ property and $ \max_{w_{i}^{l} \le w_{i}\le w_{i}^{u}} c_{i}(w_{i})/w_{i} = c_{i}(\ubar{w}_{i})/\ubar{w}_{i} $. Proposition~\ref{prop:So-invariance} and its proof then implies that $g(\weq):=\mc C(\weq)/\weq$ is a $\mc S_0$ function and $ g(\weq^{l}) = \min_{i\in \mathcal{E}}  c_{i}(w_{i}^{l})/w_{i}^{l}$, $ g(\weq^{u}) = \min_{i\in \mathcal{E}} c_{i}(w_{i}^{u})/w_{i}^{u} $, and $ g_{\max} = \min_{i\in \mathcal{E}} c_{i}(\ubar{w}_{i})/\ubar{w}_{i}$. Let $\ubar{w}_{\mathrm{eq}}$ and $\bar{w}_{\mathrm{eq}}$ denote the first and second transition points, respectively, w.r.t. $\mc S_0$ property, for $g(\weq)$. In order to establish $\mc S_1$ property of $\mc C(\weq)$, we look at its left and right derivatives:
\begin{equation}
\label{eq:left-right-derivatives}
\mc C'(\weqplus) = g(\weq) + \weq g'(\weqplus), \qquad 
\mc C'(\weqminus) = g(\weq) + \weq g'(\weqminus)
\end{equation}
Therefore, combining \eqref{eq:left-right-derivatives} with $\mc S_0$ property of $g(\weq)$, we get that: 
$\mc C'(\weqplus) > g(\weq) = \mc C(\weq)/\weq$ for all $\weq \in [\weqlow, \ubar{w}_{\mathrm{eq}})$; \( g'(\weq) \equiv 0 \), and hence \(
\mc C'(\weq) = g(\weq) = \mc C(\weq)/\weq \), for $\weq \in \left(\ubar{w}_{\mathrm{eq}}, \bar{w}_{\mathrm{eq}} \right)$. 
Moreover, using \eqref{eq:lumped-fun-derivative}, for $\weq \in (\bar{w}_{\mathrm{eq}},\wequp]$, 
\begin{equation}
\label{eq:gprime-plus-expr-long}
\begin{split}
g'(\weqplus) & = \left(\sum_{i\in \mathcal{K}^-(g(\weq))} \frac{1}{(c_i(w_i^+)/w_i)'}
 \right)^{-1}\Big|_{w=\omega^-(g(\weq))}
  = \left(\sum_{i\in \mathcal{K}^-(g(\weq))}  \frac{w_{i}}{{c'_{i}(w_{i}^{+})}
    -c_{i}(w_{i})/w_{i}}\right)^{-1}\Big|_{w=\omega^-(g(\weq))}  \\ & \le  - g(\weq) \left(\sum_{i\in \mathcal{K}^-(g(\weq))} w_{i} \right)^{-1}\Big|_{w=\omega^-(g(\weq))} \le -g(\weq) \left(\sum_{i=1}^n w_i \right)^{-1}\Big|_{w=\omega^-(g(\weq))}
  = -g(\weq)/\weq
  \end{split}
   \end{equation}
where the first inequality follows from the fact that, by definition of $\mc
K^-$, \(c_{i}(\omega^-_i(g(\weq)))/\omega^-_i(g(\weq)) = g(\weq)\) for all
\( i\in \mathcal{K}^-(g(\weq)) \), and  
\begin{equation}
\label{eq:S1-parallel-proof-intermediate}
 c'_{i}(w_{i}^{+}) \Big|_{w=\omega^-(g(\weq))}
\le 0 \qquad \forall \weq \in  (\bar{w}_{\mathrm{eq}}, \wequp]  
\end{equation}
\eqref{eq:S1-parallel-proof-intermediate} is because of the following. Due to the $ \mathcal{S}_{0} $ property of function $ g(\weq) $, $ g(\weq) \in [g(\weq^{u}, g_{\max}) $ for $ \weq \in  (\bar{w}_{\mathrm{eq}}, \wequp] $. Remark~\ref{rem:S0-S1-relation} implies that the second transition point, \wrt $\mc S_0$ property, of $c_i(w_i)/w_i$ is equal to the first transition point, $\ubar{w}_i$, w.r.t. $\mc S_0$ property, of $c_i(w_i)$. This, combined with the second equation in \eqref{eq:w-inv-prop}, further implies that \( \omega^-_{i}(g(\weq)) \ge \ubar{w}_{i} \) for $ g(\weq) \in [g(\weq^{u}, g_{\max}) $ \ie,  \( \weq \in  (\bar{w}_{\mathrm{eq}}, \wequp] \). Thereafter, the $\mc S_0$ property of $c_i$ implies \eqref{eq:S1-parallel-proof-intermediate}. 

Now we consider conditions for \eqref{eq:S1-parallel-proof-intermediate} taking equalities. The second inequality in \eqref{eq:gprime-plus-expr-long} takes equality if and only if \( \mathcal{K}^- =  \mathcal{E} \). Considering \( \mathcal{K}^- =  \mathcal{E} \), the first inequality in \eqref{eq:gprime-plus-expr-long} takes equality for \( \bar{w}_{\mathrm{eq}} \le \weq \le
\hat{g}^{-}\left( \max_{i\in \mathcal{E}} c_{i}(\bar{w}_{i})/\bar{w}_{i} \right) \). 
Furthermore, \( \mathcal{K}^-(g(\cdot)) \) is nonincreasing and Remark \ref{rem:unique-sol-general} (b) implies that $ \omega^-_{i}(g(\weq))  $ is nondecreasing for $ \weq \in (\bar{w}_{\mathrm{eq}}, \weq^{u} ] $. 
Therefore, there exists \( \tilde{w}_{\mathrm{eq}}\in [\bar{w}_{\mathrm{eq}}, \wequp] \) such that
for \( \bar{w}_{\mathrm{eq}} \le \weq \le \tilde{w}_{\mathrm{eq}} \), \( g'(\weqplus)  =
-g(\weq)/\weq\) and hence \( \mc C'(\weqplus)=0  \) from \eqref{eq:left-right-derivatives}; and for \( \tilde{w}_{\mathrm{eq}} <
\weq \le \wequp \), \( g'(\wequp)  < -g(\weq)/\weq \) and
hence \(\mc C'(\weqplus) < 0 \) from \eqref{eq:left-right-derivatives}.
One can show similar properties also for  \( \mc C'(\weqminus) \), thereby proving that \( \mc C(\weq)
\) is a $\mathcal{S}_{1}$ function.
%
\end{proof}

We now provide a characterization of the equivalent capacity function for a parallel network whose links have constant, \ie, weight-independent, capacities, in Example~\ref{eg:para-net-constant-cap}. This example generalizes our earlier work~\cite{Ba.Savla:ACC16}, where we compute only the maximum of the equivalent capacity function for parallel networks as solution to an optimization problem.


\begin{example}[Equivalent capacity for parallel networks with
  weight-independent link capacities]
\label{eg:para-net-constant-cap}
Consider a parallel network with \( n \) links from node $v_1$ to node
$v_2$. Let the lower and upper bounds on link weights be $w^l \in \real_{>0}^n$
and $w^u \in \real_{>0}^n$ respectively, and let the link capacities be $c_i>0$,
$i \in \until{n}$. Then, for every $i \in \until{n}$, $c_{i}$ is a $\mc
S_1$ function, with $w_i^l = w_{i}^{o} = \ubar{w}_{i}$. 
Let $ \mathcal{C}(\weq) $ be the equivalent capacity function and hence $g(\weq):=\mc C(\weq)/\weq$ is the solution to
\eqref{opt:equiv-cap-para-net} for this network. Lemma \ref{lem:equiv-cap-S1-para-net} implies that 
$ \mathcal{C}(\weq) $ and $ g(\weq) $ are $ \mathcal{S}_{1} $ and $ \mathcal{S}_{0} $ functions, respectively. Let $ \weq^{o} $, $  \ubar{w}_{\mathrm{eq}} $ and $ \bar{w}_{\mathrm{eq}} $ be the first and second transition points, \wrt $ \mathcal{S}_{1} $ property, and the second transition point, \wrt $ \mathcal{S}_{0} $ property, of $ \mathcal{C}(\weq) $, respectively. Remark \ref{rem:S0-S1-relation} implies that $ \weq^{o} $ and $ \ubar{w}_{\mathrm{eq}} $ are the first and second transition points, \wrt $ \mathcal{S}_{0} $ property, of $ g(\weq) $, respectively. 

 With \( \weqlow= \sum_{i=1}^{n} w_{i}^l  \)
and \( \wequp =   \sum_{i=1}^{n} w_{i}^u \), it is easy to see that  $ \weq^{o} = \weq^{l} $, 
  \begin{equation}
  \label{eq:gmax-def-parallel-example}
 \gmax =  g(\weqlow) =\min_{i \in \until{n}} c_{i}/w_{i}^{l} 
  \end{equation}
  and \( g(\wequp)= \min_{i \in \until{n}} c_{i}/w_{i}^{u} \).   Since  $\mc H(w)=\sum_{i=1}^n w_i$, the inverse function in \eqref{eq:g-inv-def}
  satisfies $\hat{g}^-(x)=\sum_{i=1}^n \omega^-_i(x)$ for all $x \in
  \left[\min_{i \in \until{n}} {c_i}/{w_i^u}, \min_{i \in
      \until{n}} {c_i}/{w_i^l} \right]$. Indeed, $\omega_i^-(x)$ can be
  explicitly written as \(  \omega_{i}^{-} (x) = \min\{ c_{i}/x, w_{i}^{u} \}
  \). Therefore, $\hat{g}^-(x)$ can be written as:  
  \begin{equation}
  \label{eq:ghat-parallel-example}
  \weq = \hat{g}^-(x)= \frac{1}{x} \sum_{i: w_i^u > c_i/x} c_{i} + \sum_{i: w_i^u \leq c_i/x} w_i^u 
  \end{equation}
Note $ \hat{g}^{-}(x) $ is decreasing.  By definition, $ \ubar{w}_{\mathrm{eq}} = \hat{g}^-(\gmax) \in [\weq^{l}, weq^{u}] $. It is straightforward that $  \hat{g}^-(\gmax) \le \weq^{u}  $, and \eqref{eq:gmax-def-parallel-example} implies that \( c_{i}/\gmax \ge w_{i}^{l}\) for all \( i \in \until{n}\), and hence \(  \hat{g}^-(\gmax)  \ge \sum_{i=1}^n w_{i}^{l} =\weqlow\). For $\weq \in \left[\weqlow, \ubar{w}_{\mathrm{eq}} \right]$, $g(\weq)=\gmax$. For \( \weq \in \left[ \ubar{w}_{\mathrm{eq}}, \wequp \right]\), the monotonicity of $ \hat{g}^{-} $ implies that $ \{i \, |\, w_i^u > c_i/g(\weq) \} = \{ i \, |\,  \weq < \hat{g}^{-}(c_{i}/w_{i}^{u}) \} $. Therefore, \eqref{eq:ghat-parallel-example} implies that 
\begin{equation}
\label{eq:g-parallel-example}
g(\weq) = \inv \, \hat{g}^{-}(g(\weq)) = \frac{\sum_{i:  \weq < \hat{g}^{-}(c_{i}/w_{i}^{u}) } c_{i}}{ \weq - \sum_{i:  \weq \ge \hat{g}^{-}(c_{i}/w_{i}^{u}) } w_{i}^{u} }
\end{equation}
Based on these calculations, the equivalent
capacity function is characterized as follows:  

For $\weq \in \left[\weqlow, \ubar{w}_{\mathrm{eq}} \right]$, 
\begin{equation}
\label{eq:eq-cap-parallel-first}
 \mc C(\weq) =  \weq \, \gmax = \weq \min_{i \in \until{n}} {c_{i}}/{w_{i}^{l}}
\end{equation}  
%
which is a linear function with slope \(\min_{i \in \until{n}} c_{i}/w_{i}^{l}\). 

If \( w_{i}^{u} > c_{i}/\gmax  \) for all \( i \in \until{n}\), \ie, \(
\gmax > \max_{i \in \until{n}} c_{i}/w_{i}^{u} \), then
\eqref{eq:ghat-parallel-example} implies that  \( \hat{g}^-(x) 
= \sum_{i=1}^{n} {c_i}/{x} \) for all $x \in \left[\max_{i \in \until{n}}
  {c_i}/{w_i^u}, \gmax\right]$. Equivalently, for all $\weq \in
\left[ \ubar{w}_{\mathrm{eq}}, \hat{g}^-(\max_{i \in \until{n}} c_i/w_i^u) \right]$, we
get $g(\weq) = \sum_{i=1}^n c_i  /\weq $, and
hence 
\begin{equation}
\label{eq:eq-cap-parallel-second}
  \mc C(\weq) =  \weq \, g(\weq) = \sum_{i=1}^{n} c_{i}
\end{equation}
It is straightforward to see that $ \bar{w}_{\mathrm{eq}} =   \max\{\hat{g}^{-}(\gmax), \hat{g}^{-}(\max_{i \in \until{n}} c_i/w_i^u)\}\in  [\ubar{w}_{\mathrm{eq}}, \weq^{u} ]$. 

Finally, for \( \weq \in \left[ \bar{w}_{\mathrm{eq}}, \wequp \right]\)
\begin{equation}
\label{eq:eq-cap-parallel-third}
  \mc C(\weq) = \weq g(\weq) = \frac{ \weq }{ \weq - \sum_{i:  \weq \ge \hat{g}^{-}(c_{i}/w_{i}^{u}) } w_{i}^{u} } \sum_{i:  \weq < \hat{g}^{-}(c_{i}/w_{i}^{u}) } c_{i}
\end{equation}

In summary, \eqref{eq:eq-cap-parallel-first}, \eqref{eq:eq-cap-parallel-second}
and \eqref{eq:eq-cap-parallel-third} completely characterize the equivalent
capacity function for parallel networks, and an illustration is provided in
Fig.~\ref{fig:equiv-cap-para-net}. Every point in the curve in Fig. \ref{fig:equiv-cap-para-net} $ (\weq, \mathcal{C}(\weq)) $ corresponds to an optimal solution of weight $ w $ to (\ref{opt:tree-reduction-problem-uniform}) for a parallel network with constant capacities on all the links and equivalent weight $ \weq $. In general, this optimal solution is not unique. However, since $ \ubar{w}_{\mathrm{eq}} $ is the second transition point of function $ c_{i}/w_{i} $ \wrt  $ \mathcal{S}_{0} $ property in this case, Remark \ref{rem:unique-sol-general} implies that the optimal solution is unique and nondecreasing for $ \weq \in [\ubar{w}_{\mathrm{eq}}, \weq^{u}] $. This is summarized in Remark \ref{rem:para-net-unique-sol}. As shown in Fig.~\ref{fig:equiv-cap-para-net}, $ \alpha^{*} $, being the maximum of function $ \mathcal{C}(\weq) $, can be computed explicitly, which in turn implies that the margin of robustness for parallel networks can be computed explicitly. 
%
\begin{figure}[ht]
 \centering
\begin{tikzpicture} 
\draw [<->] (0,3.5) node [left] {\( \mathcal{C}(w_{\mathrm{eq}}) \)}--
(0,0) node [below left] {0} -- (5,0) node [below right] {\( w_{\mathrm{eq}} \)}; 
\draw[dashed] (0, 0) to (0.5, 1); 
\draw[thick] (0.5, 1) to (1.5, 3) to (2.5, 3); 
\draw[thick] (2.5, 3) [out = -80, in = 170] to (4.5, 1);
\draw[dashed] (0.5, 1) to (0.5, 0) node[below]{ \( w_{\mathrm{eq}}^{l} \)}; 
\draw[dashed] (1.5, 3) to(1.5, 0) node[below]{ \(\ubar{w}_{\mathrm{eq}} \)}; 
\draw[dashed] (2.5, 3) to (2.5,0) node[below] {\(\bar{w}_{\mathrm{eq}} \) }; 
\draw[dashed] (4.5, 1) to (4.5, 0) node[below] {\( w_{\mathrm{eq}}^{u} \)};
\draw[dashed] (1.5, 3) to (0, 3);
\node[left] at (0, 3) {\( \alpha^{*} \)};
\end{tikzpicture}
  \caption{Equivalent capacity function for a parallel network consisting of links with constant capacities.
  \label{fig:equiv-cap-para-net} }
\end{figure}
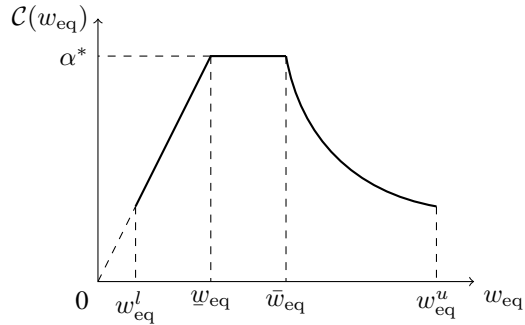
\end{example}

\begin{remark}
\label{rem:para-net-unique-sol}
For a parallel network with constant capacity on each link, $ \omega^{-}(g(\weq) ) = \min\{ c_{i}/g(\weq), w_{i}^{u}\} $ is the unique optimal solution to (\ref{opt:equiv-cap-para-net}) and is nondecreasing for $ \weq \in [\ubar{w}_{\mathrm{eq}}, \weq^{u}]  $, where $ g(\weq) $ is shown in (\ref{eq:g-parallel-example}). 
\end{remark}

\subsection{Computing Margin of Robustness for Tree Reducible Networks}
\label{subsec:tree-reducible}
Using Lemmas~\ref{lem:equiv-cap-S1-chain} and
\ref{lem:equiv-cap-S1-para-net}, and Definition~\ref{def:equivalent-capacity}, one sees that, for parallel and series networks, the equivalent capacity functions are $\mc S_1$ functions. Indeed, one can use Lemmas~\ref{lem:equiv-cap-S1-chain} and
\ref{lem:equiv-cap-S1-para-net} recursively to show $\mc S_1$ property for the equivalent capacity function for a broader class of networks. In order to see this, consider the network illustrated in Figure~\ref{fig:tree-reducible-net} where $p_{v_1}=-p_{v_4}>0$, and $p_{v_2}=p_{v_3}=0$.

Lemma~\ref{lem:equiv-cap-S1-para-net} (and Example~\ref{eg:para-net-constant-cap}) imply that the capacity of an equivalent link, say $i_{4,5}$ corresponding to links $i_4$ and $i_5$, is weight-dependent, and the capacity function for the equivalent link $i_{4,5}$ is a $\mc S_1$ function. Lemma~\ref{lem:equiv-cap-S1-chain} then implies that the equivalent capacity function for the equivalent link $i_{2,4,5}$ corresponding to links $i_2$ and $i_{4,5}$ is also a $\mc S_1$ function. The same property also holds true for equivalent link $i_{1,3}$ corresponding to $i_1$ and $i_3$. Finally, the equivalent capacity function between nodes $v_1$ and $v_4$ corresponding to links $i_{1,3}$ and $i_{2,4,5}$ can also be shown to be $\mc S_1$ by Lemma~\ref{lem:equiv-cap-S1-para-net}. Specific numerical examples are provided in Section~\ref{sec:eqv-cap}.
In summary, for the network in Figure~\ref{fig:tree-reducible-net}, the $\mc S_1$ property is
invariant from the capacities at individual link to equivalent capacity functions associated with
intermediate equivalent parallel and series reductions, finally to the one associated with the
equivalent link corresponding to the entire network. Since $\mc S_1$ implies $\mc S_0$,
\eqref{eq:fun-g-solution} then gives a computationally efficient recursive procedure to compute the
equivalent capacity function of the entire network in terms of capacities of individual links. Recalling that, for a given network, computing the equivalent capacity is the same as solving the reduction problem, the above procedure can also be used to solve the reduction problem for tree reducible network.

\begin{theorem}
\label{theorem:tree-reducible}
Consider the reduction problem \eqref{opt:tree-reduction-problem-uniform} on a link reducible
network $\mc G=(\mc V, \mc E)$, with lower and upper bounds on link weights as $w^l \in \real_{>0}^{n}$ and $w^u \in
\real_{>0}^{n}$ respectively. Then, its solution function $\mathcal{C}(\weq)$ is a $\mc S_1$ function.
\end{theorem}


Finally, the margin of robustness for a tree reducible network can be computed using the multilevel approach from Section~\ref{sec:bi-level} as follows. Recall from Section~\ref{subsec:multilevel} that the multilevel formulation consists of multiple reduction problems, and a single terminal problem. Theorem~\ref{theorem:tree-reducible} along with \eqref{eq:fun-g-solution} provides an explicit solutions to the reduction problems. Since the original network is tree reducible, the terminal problem is over a tree. Even though this tree has weight dependent capacity functions on the links, Proposition
\ref{prop:margin-min-cut-upper-bound} can be used to solve the terminal problem, and hence gives the margin of robustness. Specifically, for a tree network \( \mathcal{G}= (\mathcal{V, \mathcal{E}}) \) with link capacity functions \(
\mathcal{C}^{l}_{i}(w_{i}) \) and \( \mathcal{C}^{u}(w_{i}) \), \( i\in \mathcal{E} \), one can use Proposition
\ref{prop:margin-min-cut-upper-bound} with \( c_{i}^{l} := \min_{w_{i}} \mathcal{C}^{l}(w_{i}) \) and \( c_{i}^{u}
:= \max_{w_{i}} \mathcal{C}^{u}(w_{i})\) to compute the margin of robustness.

%% file: decentralized-control-1.tex
\section{Decentralized Control Policies}
\label{sec:decentralized}
In Sections~\ref{sec:problem-formulation}-\ref{sec:multilevel}, we described various approaches to compute the margin of robustness for a centralized control policy (which has information about link flows and weights, disturbances, as well as link flow capacities and operational range of weights), and we recall that this is an upper bound for any control policy. In this section, we analyze the robustness of decentralized policies, for parallel networks, that do not require information about the disturbance or link capacities, and moreover weight bounds information is private to each link. 

Consider a parallel network consisting of $n$ links from the supply node to the demand node. Let the magnitude of supply/demand be equal to $\alpha \geq 0$. We first specialize the margin of robustness computation to this setting. Since Remark~\ref{rem:eqv-cap-problem} (d) implies that the margin of robustness for a parallel network is related to the maximum of equivalent capacity over all feasible equivalent weights, Example~\ref{eg:para-net-constant-cap} implies that the margin of robustness for a parallel network is given by:
\begin{equation}
\label{eq:para-net-rob-sol}
  \marginparold = \max_{\weq^{l}\le w \le \weq^{u}} \mathcal{C}(\weq) = \gmax
  \hat{g}^{-}(\gmax) =  \gmax \sum_{i: w_i^u < c_i/\gmax} w_{i}^{u}   +  \sum_{i: w_i^u \ge c_i/\gmax}
  c_{i} 
\end{equation}
where we recall $\gmax=\min_{i \in \until{n}} c_i/w_i^l=1/ (\max_i w_i^l/c_i)$ and other notations used in  \eqref{eq:para-net-rob-sol} from Example~\ref{eg:para-net-constant-cap}. 
Moreover, an optimizer in \eqref{eq:para-net-rob-sol} is \( \weqopt = \hat{g}^{-}(\gmax) = \sum_{i=1}^{n}
\min\{c_{i}/\gmax, w_{i}^{u}\} \), with the corresponding link weights given by 
\begin{equation}
\label{eq:para-net-opt-sol}
 \wopt_i = w_{i}^{-}(\gmax) =
\min\{c_{i}  \max_i w_i^l/c_i , w_{i}^{u}\}.
\end{equation}

Indeed, for a parallel network, since all disturbances are of multiplicative type, and the link flows for a parallel network are explicitly given by $f_i=\alpha w_i/(\sum_{j=1}^n w_j)$, it is easy to see from \eqref{eq:weight-control-param}, as is also shown in \cite[Section III-B]{Ba.Savla:ACC16}, that the margin of robustness for a parallel network is equal to the following: 
\begin{equation}
  \label{opt:weight-control-para-net}
  \begin{aligned}
    & \underset{\alpha\in \real, w \in \real^{n}}{\text{max}} & & \alpha \\
    & \text{subject to}
    & & w^{l} \le w \le w^{u} \\
    &&& f_{i} = \frac{w_{i}}{\sum_{j=1}^n w_{j}} \alpha \quad \forall
    \, i \in \{1, \ldots, n\}     \\
    &&& 0 \le f \le c 
  \end{aligned}
\end{equation}

Remark \ref{rem:para-net-unique-sol} implies that $  \wopt $ defined in (\ref{eq:para-net-opt-sol}) is the minimal optimal solution to (\ref{opt:weight-control-para-net}), as summarized in Remark \ref{rem: para-net-minimal-sol}.

\begin{remark}
\label{rem: para-net-minimal-sol}
For a $ n $ link parallel network with constant capacities, $  \wopt $ defined in (\ref{eq:para-net-opt-sol}) is the minimal optimal solution to (\ref{opt:weight-control-para-net}), \ie, $ \tilde{w}_{i} \ge \wopt_{i} $ for all $ i \in \until{n}$ and all optimal solution $ \tilde{w} $ of  (\ref{opt:weight-control-para-net}). 
\end{remark}

The decentralized control policies considered in this paper are partially inspired by the implication of Proposition~\ref{lem:sign-of-flow-change} for a parallel network that, the decrease in the weight of a link leads to a decrease in flow on that link but an increase in flow on the parallel links connecting the same nodes. While this implication of Proposition~\ref{lem:sign-of-flow-change} does not necessarily extend to the case when multiple links change weights simultaneously, we identify conditions under which the decentralized control policies considered in this paper are provably robust, \ie, their margin of robustness is equal to the quantity computed in \eqref{eq:para-net-rob-sol}, or equivalently the optimal value of \eqref{opt:weight-control-para-net}. 

We now state two control policies and analyze their robustness within the dynamical framework of \eqref{eq:ss-model-decent-control}.


\subsection{A Memoryless Controller}
Consider the following control policy: for all $i \in \until{n}$
\begin{equation}
\label{eq:controller1}
u^1_{i}( w_{i}(t),f_i(t) )  = \left\{ \begin{array}{l@{\quad}l}
- \lambda_{i}  & f_{i}(t) > c_{i} \, \, \& \, \, w_i(t) > \wlower_i\\
0 & \text{otherwise}
\end{array}\right.
\end{equation}
where $ \lambda_{i}>0 $ is an arbitrary constant denoting the rate of decrease of $w_i$. 

Since $w(t)$ is nonincreasing under $u^1$ and is lower bounded by $w^l$, the dynamics in (\ref{eq:ss-model-decent-control}) always converges to an equilibrium under $u^1$. This is formally stated next. 


\begin{lemma}
\label{lemma:convergence}
Consider a network consisting of a directed multigraph $\mc G=(\mc V, \mc E)$, with lower and upper bounds on link weights as $w^l \in \real^{n}_{\ge 0}$ and $w^u \in \real^{n}_{>0}$, respectively. Then, for every $ \lambda \in \real^{\mc E}_{>0}$, and $ w(0) \in [w^{l}, w^{u}] $, there exists $\weqm \in [w^l,w(0)] \subseteq [w^l,w^u]$, such that, under the dynamics in \eqref{eq:ss-model-decent-control} with the controller $u^1$ in \eqref{eq:controller1}, $\lim_{t \to +\infty} w(t) = \weqm$.\footnote{Notice that Lemma \ref{lemma:convergence} is stated for a general, \ie, not necessarily parallel, networks.}
\end{lemma}
The flow $f(w^*)$ at the equilibrium $w^*$ established in Lemma~\ref{lemma:convergence} may not necessarily satisfy $f(w^*) \in [0,c]$ under all supply/demand $\alpha$. We next characterize the upper limit on this quantity and compare it with respect to the upper bound $\marginparold$.

Unless otherwise stated explicitly, in this section, we adopt the shorthand notation $\min_i$ and $\max_i$ to imply minimum and maximum, respectively, over $\until{n}$. 
Let 
\begin{equation}
\label{eq:initial-congestion-degree}
r_{i} := w_{i}(0)/c_{i}, \quad i \in \until{n}
\end{equation}
%
Without loss of generality, label the links in increasing order of $ r_{i} $, \ie, $ r_{1} \le r_{2} \le \ldots \le r_{n} $. 
Let 
\begin{equation}
\label{eq:r_star}
r^{*} := \max_{i} \frac{\wlower_{i}}{c_{i}} = \frac{1}{\gmax}
\end{equation}
Since $ w(0) \ge w^{l} $, $ r_{n} = \max_{i} w_{i}(0) /c_{i}  \ge \max_{i} w_{i}^{l} /c_{i}  $ and therefore $  r^{*} \le r_{n} $. 
Let $\bar{k}:=\min \setdef{j \in \until{n}}{r_j \ge r^*}$. This implies that 
\begin{equation}
\label{eq:rstar-ineq}
r_{\bar{k}-1} < r^* \le r_{\bar{k}}
\end{equation}
Consider the following functions: 
\begin{equation}
\label{eq:vk}
V_{k} := \frac{1}{r_{k}}\sum_{i=1}^{{k-1}} w_{i}(0) + \sum_{i={k}}^{n} c_{i}, \quad k \in \until{n}, \qquad  V^{*} := \frac{1}{r^{*}}\sum_{i=1}^{\bar{k}-1} w_{i}(0) + \sum_{i=\bar{k}}^{n} c_{i} 
\end{equation}
%
\eqref{eq:vk} implies that $ V_{1} = \sum_{i=1}^{n} c_{i} $ and, when $ r^{*} \le r_{1} $, $ \bar{k} = 1 $ and
 $ V^{*} = \sum_{i=1}^{n} c_{i} = V_{1}$.  Since $ r_{k}\le r_{k+1} $, $ V_{k} $ is nonincreasing in $ k $: 
\begin{equation*}
V_{k+1} =\frac{1}{r_{k+1}}\sum_{i=1}^{{k}} w_{i}(0) + \sum_{i={k+1}}^{n} c_{i}  
	 \le \frac{1}{r_{k}}\sum_{i=1}^{k-1} w_{i}(0) + \frac{w_{k}(0)}{r_{k}}+ \sum_{i={k+1}}^{n} c_{i} 
	=   V_{k}
	\end{equation*}
Similarly, we can show that 
\begin{equation}
\label{eq:Vstar-ineq}
V_{\bar{k}} \le V^{*} < V_{\bar{k}-1}
\end{equation}

\begin{theorem}
\label{thm:convergence}
Consider a parallel network consisting of $n$ links, with lower and upper bounds on link weights as $w^l \in \real^{n}_{>0}$ and $w^u \in \real^{n}_{>0}$, respectively, link capacities $c \in \real_{>0}^n$, and supply/demand with magnitude $\alpha \geq 0$. Then, for every $ \lambda \in \real^{n}_{>0}$ and $ w(0) \in [w^{l}, w^{u}] $, there exists $\weqm \in [w^l,w(0)] \subseteq [w^l,w^u]$, such that, under the dynamics in \eqref{eq:ss-model-decent-control} with the controller $u^1$ in \eqref{eq:controller1}, $w(t)$ monotonically converges to $\weqm \in [w^l,w(0)] \subseteq [w^l,w^u]$. 
Moreover, 
\begin{enumerate}[label=(\roman*)]
\item if $\alpha \in [0,V_n]$, then $\weqm = w(0)$ and $f(\weqm) \in [0,c]$;
\item if $\alpha \in (V_n,V^*]$ then
\begin{equation}
\label{eq:weqm-expr}
\weqm_i = \left\{ \begin{array}{ll} w_{i}(0) \quad& 1 \le i \le \hat{k}-1 \\  \hat{r} c_{i} & \hat{k} \le i \le n  \end{array} \right. 
\end{equation}
where $\hat{k}:=\min \setdef{j \in \until{n}}{\alpha \ge V_j}$, 
\begin{equation}
\label{eq:rhat-def}
\hat{r} := \left\{ 
\begin{array}{ll} 
\displaystyle{\frac{\sum_{i=1}^{\hat{k}-1} w_{i}(0)}{ \alpha - \sum_{i=\hat{k}}^{n} c_{i} } } \quad& \alpha < V_1 \\ 
 r_1 & \alpha = V_1 = V^* 
\end{array} \right. 
\end{equation}
and $f(\weqm) \in [0,c]$;
\item if $\alpha > V^*$ then $f(\weqm) \notin [0,c]$
\end{enumerate}
where $V_j$ and $V^*$ are as defined in \eqref{eq:vk}.
\end{theorem}
\begin{proof}
Monotonic convergence of $w(t)$ follows from Lemma~\ref{lemma:convergence}. If $ \alpha \in [0,V_{n}]$, then the initial flow on link $i \in \until{n}$ is given by: 
\begin{equation*}
f_{i}(0) = \frac{w_{i}(0)}{\sum_{j=1}^{n} w_{j}(0)} \alpha \le  \frac{w_{i}(0)}{\sum_{j=1}^{n} w_{j}(0)} \frac{ \sum_{j=1}^{n} w_{j}(0)}{r_{n}}\le c_{i}
\end{equation*}
\ie, the system is feasible at $t=0$. Therefore, if $ \alpha \le V_{n} $, then $u^1(t) \equiv 0$, and hence the equilibrium is $\weqm=w(0)$. This establishes part (i) in the theorem.  

If $ \alpha > V_{1} = \sum_{i=1}^{n} c_{i} $, then it is trivially $ f(w) \notin [0, c]$ for any $ w $. Hence $ \alpha \le V_{1} $ is considered in the following proof. 
Moreover, we emphasize that since $V^* \leq V_1$, $\alpha < V_1$ is satisfied in case (ii) if $V^* < V_1$. The definition of $\hat{k}$ implies that $ V_{\hat{k}} \le \alpha < V_{\hat{k}-1} $. 
This, combined with \eqref{eq:rhat-def} and \eqref{eq:vk}, implies that $\hat{r} \geq 0$, and hence $\weqm \geq 0$, for all $\alpha \in (V_n,V_1)$. \eqref{eq:vk} and \eqref{eq:initial-congestion-degree} imply that 
\begin{equation*}
\frac{1}{r_{\hat{k}}}\sum_{i=1}^{\hat{k}-1} w_{i}(0) \le \alpha - \sum_{i = \hat{k}}^{n} c_{i} < \frac{1}{r_{\hat{k}-1}}\sum_{i=1}^{\hat{k}-1} w_{i}(0) 
\end{equation*}
Therefore, the definition of $\hat{r}$ in \eqref{eq:rhat-def} implies that,   
\begin{equation}
\label{eq:rhat-bound}
\hat{r} \le r_{\hat{k}}, \qquad \forall \, \alpha \in (V_n,V_1]
\end{equation}
In writing \eqref{eq:rhat-bound}, we used the fact that, when $\alpha=V_1$, then $\hat{k}=1$, and hence $r_{\hat{k}}=r_1=\hat{r}$. Additionally, 
\begin{equation}
\label{eq:rhat-lower-bound}
\hat{r} > r_{\hat{k}-1}, \qquad \forall \, \alpha \in (V_n,V_1)
\end{equation}

We now establish the following claims: with $\weqm$ as given in \eqref{eq:weqm-expr},
\begin{enumerate}
\item[(I)] for $V_{n} < \alpha \le V_1$ $[\weqm,w(0)]$ is positively invariant under (\ref{eq:ss-model-decent-control}) with controller $u^{1}$;
\item[(II)] for $\alpha \in (V_n, V^*]$,
\begin{enumerate}
\item[(a)] $\weqm \in [w^l, w(0)] $,
\item[(b)] $\weqm$ is the only equilibrium in $[\weqm,w(0)]$,
\item[(c)] $f(\weqm) \in [0,c]$
\end{enumerate}
\item[(III)] for $ V^{*} < \alpha \le V_{1}$, $f(w) \notin [0,c]$ for all $w \in [\weqm,w(0)] \cap [w^{l}, w(0)] $.
\end{enumerate}

(I) and (II) establish part (ii) of the theorem, whereas (I) and (III) establish part (iii).

\underline{Proof of (I)}:
Since $w(t) \leq w(0)$ for all $t \geq 0$ under controller $u^1$, it suffices to show that $ w(t) \ge \weqm$ for all $t \geq 0$ under $u^1$. Assume by contradiction that this is not true. Continuity of $w(t)$ then implies that there exists $t_1 >0$ and $\hat{i} \in \until{n}$ such that $w(t) \geq \weqm$ for all $t \in [0,t_1]$, $w_{\hat{i}}(t_1)=\weqm_{\hat{i}}$ and $\dot{w}_{\hat{i}}(t_1)<0$. The latter implies that $f_{\hat{i}}(t_1) > c_{\hat{i}}$. 
%
However, 
\begin{equation}
\label{eq:contradiction}
f_{\hat{i}}(t_{1} ) = \frac{w_{\hat{i}}(t_{1} )}{ \sum_{j=1}^{n} w_{j}(t_{1} ) }  \alpha \le \frac{ \weqm_{\hat{i}} }{ \sum_{j=1}^{n} \weqm_{j} } \alpha 
\end{equation}

If $\alpha < V_1$, then \eqref{eq:weqm-expr}, \eqref{eq:rhat-def} and \eqref{eq:contradiction} imply $f_{\hat{i}}(t_{1} ) \leq {\weqm_{\hat{i}}}/{\hat{r}}$.  
For $ j \in \until{\hat{k}-1}$, \eqref{eq:initial-congestion-degree}, \eqref{eq:rhat-def}, \eqref{eq:rhat-bound} and \eqref{eq:rhat-lower-bound} imply $ \weqm_{j}/\hat{r} = {w}_{j}(0)/\hat{r} = c_{j} r_{j}/\hat{r} \le c_{j} $. For $ j \in \{\hat{k}, \ldots, n\} $, $ \weqm_{j}/\hat{r} = c_{j} $. These together imply $f_{\hat{i}}(t_1) \leq c_{\hat{i}}$, giving a contradiction.

If $\alpha=V_1$, then $\hat{k}=1$, and therefore \eqref{eq:weqm-expr} and \eqref{eq:rhat-def} imply $\weqm = r_1 c$. Using this with \eqref{eq:contradiction} implies $f_{\hat{i}}(t_1) \leq c_{\hat{i}} \alpha/(\sum_{j=1}^n c_i) = c_{\hat{i}}$, again giving a contradiction.

\underline{Proof of (II-a)}:
Following \eqref{eq:weqm-expr}, we only need to show that $\weqm_i \in [w_i^l, w_i(0)] $ for $i \in \{\hat{k}, \ldots, n\} $. It is sufficient to show that $r^* \leq \hat{r} \leq r_{\hat{k}}$. This is because $\hat{r} \geq r^*$ combined with \eqref{eq:r_star} implies that $\hat{r} \geq w_i^l/c_i$, and hence $\weqm_i \geq w_i^l$, for all $i \in \{\hat{k}, \ldots, n\}$;  and $\hat{r} \leq r_{\hat{k}}$, which has already been established in \eqref{eq:rhat-bound}, combined with the non-decreasing property of the sequence $\{r_k\}_{k=1}^n$ implies $\hat{r} \leq r_i$, and hence $\weqm_i \leq w_i(0)$ for all $i \in \{\hat{k}, \ldots, n\}$, from \eqref{eq:initial-congestion-degree}. 
Since $\alpha \in (V_{n}, V^{*}]$,  \eqref{eq:Vstar-ineq} implies $\hat{k} \ge \bar{k} $.  If $ \hat{k} = \bar{k} $, then $ \alpha - \sum_{i= \hat{k}}^{n} c_{i} \le V^{*} - \sum_{i= \bar{k}}^{n} c_{i} =  \left( \sum_{i=1}^{\bar{k}-1} w_{i}(0) \right) /r^{*}$. \eqref{eq:rhat-def} then implies $ \hat{r} \ge r^{*} $. If $\hat{k} > \bar{k}$, \ie,  $\hat{k}-1 \geq \bar{k}$, then the non-decreasing property of $\{r_k\}_{k=1}^n$ implies $r_{\hat{k}-1} \geq r_{\bar{k}}$, which when combined with \eqref{eq:rhat-lower-bound} and \eqref{eq:rstar-ineq} implies $ \hat{r} > r^{*}$ if $\alpha < V_1$. On the other hand, if $\alpha=V_1$, then $\hat{r}=r^*=r_1$. This completes the proof for $\weqm \in [w^l,w(0)]$
Combining this with the definition of $\weqm$ in \eqref{eq:weqm-expr} implies that 
\begin{equation}
\label{eq:w-invariance}
w_i(t) \equiv w_i(0), \qquad i \in \until{\hat{k}-1}
\end{equation}
If $ \alpha = V_{1} $,  then $ \hat{k}=1 $, the set $ \until{\hat{k}-1} $ is empty. However, in this case, $ w_{1}^{*} = r_{1}c_{1} = w_{1}(0) $. Therefore, 
\begin{equation}
\label{eq:w1-invariance}
w_{1}(t) \equiv w_{1}(0) , \quad \forall\, \alpha \in (V_{n} , V_{1}].
\end{equation}

\underline{Proof of (II-b)}: By contradiction, suppose $\tilde{w} \in [\weqm,w(0)]\setminus \{ \weqm \}$ is also an equilibrium. \eqref{eq:w-invariance} and \eqref{eq:w1-invariance} imply there exists $ \mathcal{E}' \subset \{ \max\{2, \hat{k}\}, \ldots, n \} $ such that $ \tilde{w}_{i} > \weqm_{i} $ for all $ i \in \mathcal{E}' $, and $ \tilde{w}_{i}=\weqm_{i} $ for $ i \not \in \mathcal{E}' $ (we have already proven $w(t) \geq \weqm$ for all $t \geq 0$). 
Therefore, 
\begin{equation*}
\sum_{i \in \mathcal{E}'} f_{i}(\tilde{w}) = \frac{ \sum_{ i \in \mathcal{E}'} \tilde{w}_{i} }{\sum_{ j \in \mathcal{E}'} \tilde{w}_{j} + \sum_{j \not \in \mathcal{E}'} \weqm_{j}  }  \alpha > \frac{ \sum_{i\in \mathcal{E}'} \weqm_{i} }{ \sum_{j=1}^{n} \weqm_{j} } \alpha = \sum_{i\in \mathcal{E}'} c_{i} 
\end{equation*}
where the inequality is due to the fact that $ \until{n} \setminus \mathcal{E}' $ is nonempty, and the equality follows from the same argument used in the proof of (I): if $\alpha<V_1$, then $ \alpha/(\sum_{j=1}^n \weqm_j)=1/\hat{r}$, and hence $(\sum_{i\in \mathcal{E}'} \weqm_{i}) \alpha/(\sum_{j=1}^n \weqm_j) = \sum_{i\in \mathcal{E}'} \weqm_i/\hat{r}=\sum_{i\in \mathcal{E}'} c_i$; if $\alpha=V_1=\sum_{i=1}^n c_i$, then $\hat{k}=1$, $\weqm = r_1 c$, and hence $(\sum_{i\in \mathcal{E}'} \weqm_{i}) \alpha/(\sum_{j=1}^n \weqm_j) = \sum_{i\in \mathcal{E}'} c_i$. Therefore, there exists at least one $ j \in \mathcal{E}' $ such that $ f_{j}(\tilde{w}) > c_{j} $. This, combined with the fact that $ \tilde{w}_{i} > \weqm_{i} \ge w^{l}_{i} $ for all $i \in \mc E'$, implies that $ \tilde{w} $ can not be an equilibrium under $u^1$. 

\underline{Proof of (II-c)}: For any $i \in \until{n}$, $f_i(\weqm_i)=\weqm_i \alpha/(\sum_{j=1}^n \weqm_j)$. Along the same argument used in the proof of (I), we have: 

If $\alpha<V_1$, then $ \alpha/(\sum_{j=1}^n \weqm_j)=1/\hat{r}$, and therefore $f_i(\weqm_i)=\weqm_i/\hat{r}$. This is equal to $c_i$ for $ i\in \{\hat{k} , \ldots, n \}$, from \eqref{eq:weqm-expr}. For $ i \in \until{\hat{k}-1}$, since $ \hat{r} > r_{\hat{k}-1} \ge r_{i} $ from \eqref{eq:rhat-bound} and non-decreasing property of $\{r_k\}_{k=1}^n$, we have, $ f_i(\weqm) = w_{i}(0)/\hat{r} \le w_{i}(0) /r_{i} = c_{i} $ from \eqref{eq:initial-congestion-degree}. 

If $\alpha=V_1$, then $\hat{k}=1$, and hence $\weqm = r_1 c$. Therefore, $f_i(\weqm_i)=c_i$ for all $i \in \until{n}$.

\underline{Proof of (III)}:
If $\alpha \in (V^*, V_1)$, then $ 2 \le \hat{k} \le \bar{k} $.  If $ \alpha = V_{1} $, then $ \hat{k} =1 $, $ \hat{r} = r_{1} $ and $ \weqm = r_{1}c$. In particular, $\weqm_1 = r_{1}c_{1} = w_{1}(0)$. Therefore, for convenience, we can set $ \hat{k}  = 2$ for $ \alpha = V_{1} $ and (\ref{eq:weqm-expr}) remains valid. In summary, we set the convention that $ \bar{k} \ge \hat{k} \ge 2 $ for all $ \alpha \in (V^{*}, V_{1}] $. Consequently, the set $ \{1, \ldots, \hat{k}-1\}  $ is not empty,
and, using similar argument as in the proof of (II-b), it can be shown that $ \sum_{i = \hat{k}}^{n} f_{i}(w)  >  \sum_{i = \hat{k}}^{n} c_{i} $ for any $ w \in [\weqm, w(0)] \setminus \{\weqm\}$, \ie, $f(w) \notin [0,c]$ for all $w$ in $  [\weqm, w(0)] $ other than $ \weqm $. Furthermore, $\weqm \notin [w^l, w(0)]$. This is because of the following:
Since $r^{*} > r_{\bar{k}-1} $ from \eqref{eq:rstar-ineq}, maximum of $\{w_{i}^{l}/c_{i}\}_{i=1}^n$ occurs in $ \{ \bar{k},\ldots, n \} $. If $\tilde{k}$ denotes one such maximizer, then $\tilde{k} \geq \bar{k} \geq \hat{k}$. 
%
%
Therefore, $ w_{\tilde{k}}^{l} > c_{\tilde{k}} \hat{r} = \weqm_{\tilde{k}} $, where the inequality is due to $ r^{*} > \hat{r} $, which can be shown using argument similar to the one in the proof of (II-a),  and the equality is due to the definition of $ \weqm_{\tilde{k}} $ for $ \tilde{k} \ge \hat{k}$  from \eqref{eq:weqm-expr}. 
\end{proof}  


Theorem \ref{thm:convergence} implies that the margin of robustness of $ u^{1} $ is equal to $ V^{*} $. The following proposition states sufficient conditions for $ u^{1} $ to be maximally robust, \ie, sufficient conditions for $V^* = \alpha^*$.
\begin{proposition}
\label{prop:controller1-nec-suff-cond}
Consider a parallel network consisting of $n$ links, with lower and upper bounds
on link weights as $w^l \in \real^{\mc E}_{>0}$ and $w^u \in \real^{\mc
  E}_{>0}$, respectively, link capacities $c \in \real_{>0}^n$. Then, for every
$ \lambda \in \real^{n}_{>0}$ and $ w(0) \in [w^{l}, w^{u}] $, we have $
R(u^1)=V^{*}= \alpha^{*} $ (cf. (\ref{eq:para-net-rob-sol}) and
(\ref{eq:vk})) if and only if  \( w(0) \ge \wopt \),  
where \( \wopt \) is an optimal solution to (\ref{opt:weight-control-para-net}), 
as defined in (\ref{eq:para-net-opt-sol}).  
\end{proposition}
\begin{proof}
Let \( \mathcal{E}_{0} := \{i\in \until{n}\, |\,
{w_{i}^{u}}/{c_{i}} < r^{*} = \max_{i}
{w_{i}^{l}}/{c_{i}}  \} \). (\ref{eq:para-net-rob-sol}) and (\ref{eq:para-net-opt-sol}) then imply that \( \alpha^{*} = \sum_{i\in \mathcal{E}_{0}}
w_{i}^{u} /r^{*} + \sum_{i\notin \mathcal{E}_{0}} c_{i} \)
and  
\begin{equation}
\label{eq:wopt-simple}
\wopt_i = w_{i}^{u}, \quad \forall \, i\in
\mathcal{E}_{0}; \qquad \wopt_i = c_{i}r^{*}, \quad \forall \, i \in \until{n} \setminus \mathcal{E}_{0} 
\end{equation}
We need to show that \( w(0)\ge \wopt \) is necessary and sufficient
condition for:
\begin{equation}
\label{eq:v_star-alpha}
V^{*} = \frac{1}{r^{*}} \sum_{i=1}^{\bar{k}-1} w_{i}(0) + \sum_{i=\bar{k}}^{n}
c_{i} = \frac{1}{r^{*}} \sum_{i\in \mathcal{E}_{0}}
w_{i}^{u}  + \sum_{i\notin \mathcal{E}_{0}} c_{i} = \alpha^{*}
\end{equation}


Since \( w(0)\le w^{u} \), by definition  \( \mathcal{E}_{0} \subseteq \until{\bar{k}-1} \). 
Moreover, by definition \(w_{i}(0)/r^{*}< c_{i} \) for \( i\le \bar{k}-1 \). 
Therefore, it is straightforward to see that (\ref{eq:v_star-alpha}) is true if and only if: 
\begin{inparaenum}[(i)]
  \item   \(\mathcal{E}_{0} = \{1, \ldots, \bar{k}-1\} \); and
    \item \( w_{i}(0) = w_{i}^{u} \) for all \( i\in \mathcal{E}_{0} \). 
\end{inparaenum}
Since \( \mathcal{E}_{0} \subseteq \{1, \ldots, \bar{k}-1\} \), 
(i) is equivalent to \(\{1, \ldots, n\}\setminus
 \mathcal{E}_{0} \subseteq \{\bar{k}, \ldots, n\}  \), which is
 further equivalent to \( w_{i}(0) \ge c_{i}r^{*}\)
for all \( i\in \{1, \ldots, n \}\setminus \mathcal{E}_{0} \). Therefore, considering the definition of \( \wopt \) from \eqref{eq:wopt-simple}, conditions (i) and (ii) can be succinctly written as
\( w(0) \ge \wopt \). 
\end{proof} 

\begin{remark}
\label{rem:controller1}
\leavevmode
\begin{enumerate}[label = (\alph*)]
\item Since the controller  $ u^{1} $ only decreases weights, the initial weight must be greater than at least one optimal solution of weight in order for the controller $ u^{1} $ to be maximally robust. Remark \ref{rem: para-net-minimal-sol} implies that the condition $ w(0) \ge \wopt $ is not conservative because $ \wopt $ is the minimal optimal solution. 
\item Since $ w^u \ge \wopt $, Proposition~\ref{prop:controller1-nec-suff-cond} implies that $u^1$ is maximally robust for parallel networks if $ w(0) = w^{u} $. 
\item Referring to Theorem~\ref{thm:convergence}, the weights on links in set $\{1, \ldots, \max\{1, \hat{k}-1\} \}$ does not change under $u^1$, whereas the weights on links in set $\{ \max\{2, \hat{k}\}, \ldots, n\}$ potentially changes, and indeed these links become capacitated at the equilibrium $\weqm$. 
\item  $ \lambda_{i} $ in \eqref{eq:controller1} can be arbitrary and time varying. 
\end{enumerate}
\end{remark}


%% file: decentralized-control-2.tex

\subsection{Controller with Memory}
We now present a control policy which augments $u^1$ by increasing weight on link $i$ when the flow $f_i$ is increasing. The control policy is formally stated as follows: for all $i \in \until{n}$, 
\begin{equation}
\label{eq:altru-control}
u^2_{i}( w_{i}(t),f_i(t) )  = \left\{ \begin{array}{l@{\quad}l}
- \lambda_i  & f_{i}(t) > c_{i} \, \, \& \, \, w_i(t) > \wlower_i \\ 
 \lambda_i & f_{i}(t) < c_{i} \, \, \& \, \, \dot{f}_i(t^-) > 0 \\
& \& \, \, w_i(t) < \wupper_i \\
0 & \text{otherwise}
\end{array}\right.
\end{equation}
where $\lambda_i >0$ is an arbitrary constant, and $\dot{f}_i(t^-):=\lim_{\triangle t \to 0^-}(f_i(t + \triangle t) - f_i(t))/\triangle t$ is the left derivative of $f_i(t)$.
The control policy in \eqref{eq:altru-control} has a natural altruistic interpretation as follows: 
the controller on link $i$ takes an action when either the flow on link $i$ exceeds its capacity, or it sees an increase in the flow on link $i$. In particular, in the latter case, controller $i$ increases weight on link $i$ in order to further increase the flow on link $i$, and thereby possibly avoiding infeasibility on other links. For parallel networks, if the disturbance at $t=0$ leads to increase in supply/demand, then it leads to increase in flows on all links. In such a case, under $u^2$, 
\begin{equation}
\label{eq:fdot-zero-pos}
\dot{f}_i(0^-)>0, \qquad  \forall \, i \in \until{n} 
\end{equation}
The maximal robustness of $u^2$ for $n=2$ links is proven next. 


\begin{proposition}
\label{pro:orderone-lowerbound}
Consider a parallel network consisting of $2$ links, with lower and upper bounds on link weights as $w^l \in \real^{2}_{>0}$ and $w^u \in \real^{2}_{>0}$, respectively, link capacities $c \in \real_{>0}^2$, and supply/demand with magnitude $\alpha \geq 0$. Then, for every $ \lambda \in \real^{2}_{>0}$ and $ w(0) \in [w^{l}, w^{u}] $, under the dynamics in \eqref{eq:ss-model-decent-control} with the controller $u^2$ in \eqref{eq:altru-control}, if $\alpha < \alpha^*$ (cf. \eqref{eq:para-net-rob-sol}), then $\lim_{t \to + \infty} f(w(t)) \in [0,c]$.
\end{proposition}
\begin{proof}
Assumption~\ref{ass:initial-flow-feasible} implies that, if the disturbance decreases the
supply/ demand $\alpha$, then the flow on each link decreases, and hence $u^2(t) \equiv 0$. Therefore, the  system is feasible. Hence, we only consider disturbances that increase \(\alpha\), in which case (\ref{eq:fdot-zero-pos}) applies. 

For $n=2$, the optimal solution characterized in \eqref{eq:para-net-rob-sol} and \eqref{eq:para-net-opt-sol} can be explicitly written as shown in Table~\ref{tab:sol-para-net}.
\begin{table}[htbp]
   \centering
   \begin{tabular}{ccc} 
      \toprule
           Configuration    & $ (\wopt_1, \wopt_2) $ & $  \alpha^{*} $ \\
      \midrule
       $ c_{1}/c_{2} < \wlower_{1}/\wupper_2$       & $ (\wlower_{1}, \wupper_{2})$ & $ c_{1}(1+ \wupper_2/\wlower_{1})$  \\
       $ \wlower_{1}/\wupper_2 \le  c_{1}/c_{2}  \le \wupper_1/\wlower_{2} $      &$ w_{1}/w_{2} = c_{1}/c_{2} $ & $ c_{1} + c_{2}$  \\
      $  c_{1}/c_{2} >  \wupper_1/\wlower_{2}  $     & $ ( \wupper_{1}, \wlower_{2}) $ & $ c_{2}(1+ \wupper_1/\wlower_{2}) $ \\
      \bottomrule
   \end{tabular}
      \caption{Explicit characterization of $\alpha^*$ and $\wopt$ from \eqref{eq:para-net-rob-sol} and \eqref{eq:para-net-opt-sol}, respectively, for $n=2$. }
      \label{tab:sol-para-net}
\end{table}

\begin{enumerate}
\item[(I)] If \(\alpha < (w_{1}(0)+w_{2}(0) ) \min\{c_{1}/w_{1}(0),
    c_{2}/w_{2}(0)\}  \), then it is
  straightforward to see that \( f(0)< c \). Note that, due to \eqref{eq:fdot-zero-pos}, this does not imply $u^2(0) = 0$. Accordingly, we consider the following three cases. 
  \begin{enumerate}
  \item[(I-A)] If \( w(0) = w^{u} \), then $u^2(t) \equiv 0$, and hence \( \lim_{t \to +
    \infty} f(w(t)) =f(0) \in [0,c] \). 
  \item[(I-B)] If \( w_{1}(0) <
  w_{1}^{u} \) and \( w_{2}(0) = w_{2}^{u} \), then \( u_{1}(0) = \lambda_{1} >0 \)
  and \( u_{2}(0) = 0 \). \( w_{1}(t) \) keeps increasing and \( w_{2}(t) \)
  stays unchanged, and consequently \( f_{1}(t) \) and \( f_{2}(t) \) keep increasing and
  decreasing, respectively, until either one of the following happens at some time \( \bar{t} \):
   \( w_{1}(\bar{t}) = w_{1}^{u} \) or \( f_{1}(\bar{t}) = c_{1} \).
The weights do not change thereafter, and hence $f_1(t) \leq c_1$ and $f_2(t)<c_2$, for all $t \geq \bar{t}$. The argument for the other scenario $w_1(0)=w_1^u$ and $w_2(0)<w_2^u$ is symmetrical.
\item[(I-C)] If \( w(0) < w^{u} \), then \( u^2(0) =
\lambda \) and hence 
$\dot{f}_{1}(0^{+}) = -\dot{f}_{2}(0^{+})= 
  \alpha  \left(\lambda_{1}w_{2}(0) - \lambda_{2}  w_{1}(0) \right) /(w_{1}+w_{2})^{2}$. 
  \begin{enumerate}
  \item[(I-C-i)] If \( w_{1}(0))/w_{2}(0) = \lambda_{1}/\lambda_{2} \), then \( \dot{f}_{1}(0^{+}) =
\dot{f}_{2}(0^{+}) = 0 \), and hence  $u^2(t) \equiv 0$.
  \item[(I-C-ii)] If \( w_{1}(0))/w_{2}(0) < \lambda_{1}/\lambda_{2} \), then \(
\dot{f}_{1}(0^{+}) > 0 \) and \( \dot{f}_{2}(0^{+})< 0 \). This implies that \( w_{1}(t) \) keeps
increasing and \( w_{2}(t) \) stays unchanged at $t=0$, and hence the asymptotic behavior is the same as in Case (I-B). 
Similar argument can be made for the other scenario when \( w_{1}(0))/w_{2}(0) > \lambda_{1}/\lambda_{2} \). 
  \end{enumerate}
  \end{enumerate}

\item[(II)] If \( \alpha = \left(w_{1}(0)+w_{2}(0))\min\{c_{1}/w_{1}(0),
    c_{2}/w_{2}(0)\} \right) \), then $f_1(0)=c_1$ or $f_2(0)=c_2$. Without loss of generality, assume \( f_{2}(0) = c_{2} \) and \( f_{1}(0) < c_{1} \), in which case, $w_1(t)$ keeps increasing and $w_2(t)$ remains unchanged at $t=0$, and the asymptotic behavior is the same as in Case (I-B). 

\item[(III)] If \( \left(w_{1}(0)+w_{2}(0))\min\{c_{1}/w_{1}(0),
    c_{2}/w_{2}(0)\} \right) < \alpha \le \alpha^{*} \), then either $f_1(0)>c_1$ and $f_2(0)<c_2$, or $f_1(0)<c_1$ and $f_2(0)>c_2$. Without loss of generality, assume \(
  f_{1}(0) > c_{1} \) and \( f_{2}(0) < c_{2} \). This implies \( c_{1}/c_{2} <
  w_{1}(0)/w_{2}(0) \le w_{1}^{u}/w_{2}^{l} \). 
 Then \( u_{1}(0) =
  -\lambda_{1}< 0 \) and \( u_{2}(0) = \lambda_{2} > 0 \). Thereafter, \( w_{1}(t) \) and \(
  f_{1}(t) \) keep decreasing until either one of the following happens at \( t_{1} \):  (e1) \( w_{1}(t_{1}) = w_{1}^{l} \) or (e2) \( f_{1}(t_{1}) = c_{1} \); and 
%
%
\( w_{2}(t) \) and \( f_{2}(t) \) keep increasing until either one of the following happens at \( t_{2} \):  (e3) \(  w_{2}(t_{2}) = w_{2}^{u} \) or (e4) \( f_{2}(t_{2}) = c_{2} \). 
We now consider the two cases: $t_1<t_2$ and $t_2<t_1$ separately (ties are broken arbitrarily).
%
\begin{enumerate}
\item[(III-A)] \underline{$t_1<t_2$}: we consider two sub-cases depending on which of (e1) or (e2) happens first.
\begin{enumerate}
\item[(e1)] \( w_{1}(t_1) = w_{1}^{l} \), \( f_{1}(t_1) \ge c_{1} \), \(
  f_{2}(t_1) < c_{2} \) and \( w_{2}(t_1) < w_{2}^{u} \). In this case,
  \( w_{1} \) stops decreasing  at \( t_1  \), but \(w_{2} \) keeps
  increasing until \( t_{2} \) when (e3) or (e4) happens. 
  
  If (e3) happens before (e4), then \( f_{1}(t_{2}) = \alpha w_{1}^{l}/\left(w_{1}^{l} +
    w_{2}^{u}\right) \le c_{1} \). Since (e4) has not occurred, then $f_2(t_2) \leq c_2$, and since the weights are at the boundary at $t_2$, they do not change thereafter.
  
  If (e4) happens before (e3), then $f_2(t_2)=c_2$, and therefore, $f_1(t_2) = \alpha -f_2(t_2) \leq c_1 + c_2 - f_2(t_2) \leq c_1$. 
\item[(e2)] \( f_{1}(t_1) = c_{1} \),  \( w_1(t_1) \ge w_1^l \), \(
  f_{2}(t_1) < c_{2} \) and \( w_{2}(t_1) < w_{2}^{u} \). In this case,
 \( w_{1} \) stops decreasing  at \( t_1  \), but \(w_{2} \) keeps
  increasing until \( t_{2} \) when (e3) or (e4) happens. It is straightforward to see that the system is feasible under both of these scenarios.
\end{enumerate}
\item[(III-B)] \underline{$t_2<t_1$}: we consider two sub-cases depending on which of (e3) or (e4) happens first.
\begin{enumerate}
\item[(e3)]  \( w_{2}(t_2) = w_{2}^{u} \), \( w_{1}(t_2) > w_{1}^{l} \),
  \( f_{1}(t_2) > c_{1} \) and \( f_{2}(t_2) <  c_{2} \). In this case, \( w_{2} \) stops
  increasing at $t_2$, but \( w_{1} \) keeps decreasing until \( t_{1} \) when (e1) or (e2) happens.
Indeed, in this case, (e2) always precedes (e1). This is because, \( \alpha \le \alpha^{*}
  \) implies that, in the all the relevant (\ie, first and second) configurations in Table \ref{tab:sol-para-net}, \(\alpha
  w_{1}^{l}/\left(w_{1}^{l} + w_{2}^{u}\right) \le c_{1} \). When (e2) happens, $f_1(t_1)=c_1$, and $f_2(t_1) = \alpha -f_1(t_1) \leq c_1 + c_2 - f_1(t_1) \leq c_2$. 
\item[(e4)] This is not possible because (e4) never precedes (e2). This is because, by contradiction, if it does, then $f_2(t_2)=c_2$ and $f_1(t_2)>c_1$ implying $\alpha=f_1(t_2)+f_2(t_2)>c_1 + c_2 \geq \alpha^*$.
\end{enumerate}
\end{enumerate}
\end{enumerate}

\end{proof}  

\begin{remark}
\leavevmode
\begin{enumerate}[label = (\alph*)]
\item Proposition~\ref{pro:orderone-lowerbound} implies that $u^2$ is maximally robust for parallel networks with 2 links. Moreover, this maximal robustness property of $u^2$, unlike $u^1$, does not require extra conditions on $w(0)$.
\item For parallel networks with 2 links, the action of the controller $u^2$ can
  be shown to be a descent algorithm to solve
(\ref{eq:opt-grad-descent}).
\item Note that the proof of Proposition~\ref{pro:orderone-lowerbound} implies that, under any disturbances, the asymptotic link weights under $u^2$ are on the boundary in many scenarios. It is possible to address this feature by proper selection of $\lambda_i$, $i \in \until{n}$, and by extending the criterion for a link to increase the weight. Robustness analysis under such extensions to general parallel networks will be reported in future.  
\end{enumerate}
\end{remark}

\begin{remark}
For both $ u^{1} $ and $ u^{2} $, it can be shown that the results of Theorem \ref{thm:convergence},  Proposition \ref{prop:controller1-nec-suff-cond} and Proposition \ref{pro:orderone-lowerbound} hold true for the case when $ w^{l} = 0 $. 
\end{remark}

%% file: simulations-qin.tex
\section{Simulations}
\label{sec:simulations}

We report numerical estimates of margin of robustness obtained from the various optimization methods proposed in this paper, along with the decentralized control policy $u^1$ on a standard IEEE benchmark network, as well as equivalent capacity functions for the network shown in Figure~\ref{fig:tree-reducible-net}. All the simulations were performed using Matlab 2015b on a desktop with the following
configurations: Intel(R) Core(TM) i7-6700K CPU 4.00GHz and 16GB RAM. 

\subsection{Margin of Robustness Estimates}
\begin{figure}[htb!]
\centering
\begin{minipage}[b]{.5\textwidth}
  \centering
   \includegraphics[width=\linewidth]{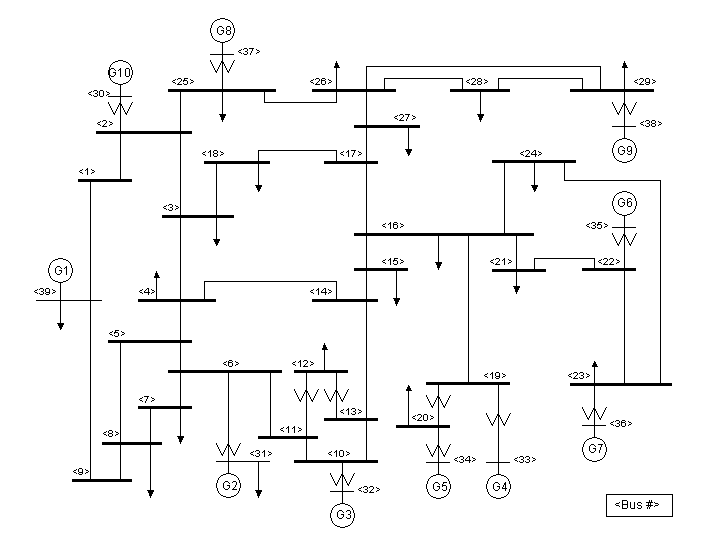}
   \caption{IEEE 39 bus system}
   \label{fig:ieee39}
\end{minipage}%
\begin{minipage}[b]{.5\textwidth}
  \centering
  \resizebox{0.8\linewidth}{!}{
\begin{tikzpicture}[->,>=stealth',shorten >=1pt,auto,node distance=0.6cm and 1.5cm,
                   thick,main node/.style={circle,minimum size=0.8cm, draw,font=\sffamily\bfseries, fill=blue!40}]
 \node[main node] (39) {39};
 \node[main node] (2) [above right =2.3cm of 39] {2};
 \node[main node] (3) [above right =of 2] {3};
 \node[main node] (17) [below right =of 2] {17};

 \node[main node] (8) [below right =2.3cm of 39] {8};
 \node[main node] (5) [below right =of 8] {5};
 \node[main node] (6) [above right = of 8] {6};
 
 \node[main node] (14) [below right = of 17] {14};
 \node[main node] (4) [right = of 14] {4};

 \path[every node/.style={font=\sffamily\small}]
   (39) edge  (2) 
   (39) edge  (8)
   (2) edge  (3) 
   (2) edge  (17)
   (3) edge  (17) 
     
    (8) edge (5)   
    (8) edge  (6)   
    (5) edge (6)   

    (17) edge  (14)   
    (6) edge (14) 
    
    (3) edge  (4) 
    (14) edge (4) 
    (5) edge (4) ;
      
\end{tikzpicture}
}
\caption{The terminal network for IEEE 39 bus system}
\label{fig:ieee39-reduction} 
\end{minipage}
\end{figure}

Consider the IEEE 39 bus
system shown in Figure~\ref{fig:ieee39} with the supply-demand vector  
$p_0$ chosen to be such that \( p_{0, 39} = 1 \), \(
p_{0, 4} = -1 \) and \( p_{0, v} = 0 \) for every other node \( v \). The corresponding terminal network obtained by the multilevel formulation is shown in Figure~\ref{fig:ieee39-reduction}. The flow capacities on every link were chosen to be symmetrical: $\cupper = -
\clower = 2.600 \, \onebf$. $\wupper$ was selected to be the value of susceptances
for this network provided by \cite{zimmerman2011matpower}. We consider multiplicative disturbances, \ie,  disturbances of the form $\alpha p_0$, $\alpha \in \real$. For this case, $\alpha^*$ is obtained by solving (\ref{eq:weight-control-param}). The margin of robustness, $\nu_M^*$ in this case can be obtained from $\alpha^*$ using \eqref{eq:nu-alpha-relationship}. Therefore, we present our results in this section in terms of $\alpha^*$.

Without weight control, \ie, when \( w^{l} = w^{u} \), $\alpha^* = 
4.725$. Under weight control, Proposition~\ref{prop:margin-min-cut-upper-bound} implies that $\alpha^* \leq 5.200$. We compared the solution to \eqref{eq:weight-control-param} obtained from the following three methods:
\begin{enumerate}
\item exhaustive search method for the original as well as the multilevel formulation as described in Algorithm~\ref{alg:multilevel};
\item random search method for the original as well as the multilevel formulation as described in Algorithm~\ref{alg:multilevel};
\item sub-gradient projection method described in Section~\ref{subsec:projected-subgradient}. 
\end{enumerate} 
For the exhaustive search method, the set $[w^l,w^u]$ is discretized with resolution $0.5$, and the cost function is evaluated at each of these discrete points according to a natural lexicographical order. In the random search method, the points for evaluation of the cost function are chosen random according to a uniform distribution over $[w^l,w^u]$. For both these methods, we choose \( w^{l} = 0.95 w^{u} \). 

For the exhaustive search method, the average time for evaluation of a single feasible for the original and the multilevel formulation was \( 1.24\times 10^{-4} \) and \( 7.68\times 10^{-5} \) seconds respectively, illustrating the computational gains per evaluation from the network reduction procedure underlying the multilevel formulation. For the original formulation, due to the large number of feasible discrete points, it was found to be impractical to exhaustively evaluate the cost function at each of these discrete points. However, for the multilevel formulation, the exhaustive search method terminated in about $59.3$ hours yielding $\alpha^* \approx 4.806$. 

For the random search method, we performed 10 runs, each for 30 minutes. The average and the maximum values of $\alpha^*$ obtained for the original formulation are $4.825$ and $4.822$ respectively, and for the multilevel formulation are $4.831$ and $4.830$ respectively. These values also illustrate computational advantage of the multilevel formulation.

For the projected sub-gradient method, a larger controllable weight range was used by setting
\( w^{l} = 0.5 w^{u} \) and the step size of the descent method was chosen to be
$0.2$. The estimates of $\alpha^*$ using this method for different initial points were found to be $5.200$, which matches the upper bound. 
This suggests convergence of the projected sub-gradient method to an optimal solution. This is to be contrasted with possible theoretical results which only ensure convergence to a critical point. These observations, along with better performance of random search in comparison to exhaustive search suggest that optimal solutions are dense. Further analysis of this aspect is left to future work. 
%
%

Under controller $u^1$ with \( w(0) = w^{u} \), the estimate of $\alpha^*$ was also found to be $5.200$, suggesting optimality of $u^1$ for this setting. This is to be contrasted with point 2) in Remark~\ref{rem:controller1} (b), which guarantees optimality of $u^1$ only for parallel networks.


\subsection{Equivalent Capacities for Tree Reducible Networks}
\label{sec:eqv-cap}
Consider the network shown in Fig. \ref{fig:tree-reducible-net} with nodes \( v_{1} \) and \(
v_{4} \) being the supply and demand nodes, respectively, and the weights bounds and  link capacities are selected as follows: $ w^{l} = [4\,\; 3\,\; 4\,\; 1\,\; 2]^{T} $, $ w^{u} = [9\,\; 10\,\; 18\,\; 5\,\; 8]^{T} $, and $ c = [16\,\; 18\,\;  20\,\; 10\,\; 10]^{T} $. The equivalent capacities for sub-networks formed during the sequential reduction process described in Section~\ref{sec:multilevel}, are illustrated in Figure~\ref{fig:tree-reducible-capFunPlot}. Note that each of the equivalent capacity function is $\mc S_1$ function. 

\begin{figure}[htb!]
  \centering
\includegraphics[width=0.6\linewidth]{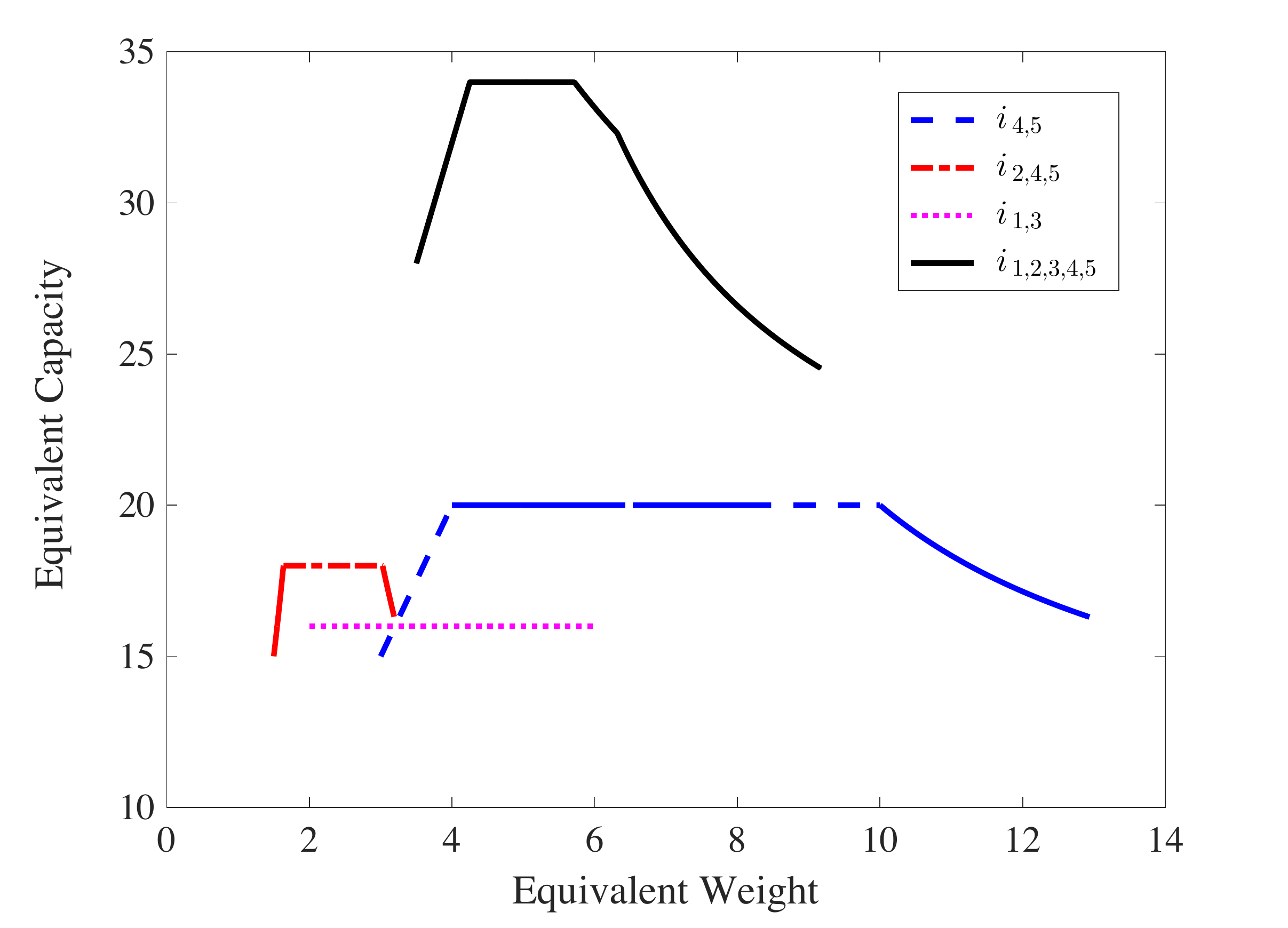}
  \caption{The equivalent capacity functions in the process of tree reduction
    for the network shown in Fig. \ref{fig:tree-reducible-net}.}
  \label{fig:tree-reducible-capFunPlot}
\end{figure}

%% file: conclusions.tex
\section{Conclusion and Future Work}
\label{sec:conclusions}
In this paper, we studied robustness of control policies for DC power networks, that use information about link flows and weights, and disturbance to change line weights in order to ensure that the line flows remain within prescribed limits. Robust control design in the centralized case can be cast as an optimization problem, which is non-convex in general. We proposed a gradient descent algorithm for multiplicative perturbations, and a multilevel programming approach, which lead to substantial computational savings when adopting exhaustive search solution technique for reducible networks. We also presented robustness analysis of natural decentralized control policies. Beyond the robust weight control problem, the paper makes a few contributions which are of independent interest, including exact derivation of the flow-weight Jacobian, characterization of a class of decomposable non-convex network optimization problems, and formalization of the notion of equivalent transmission capacity for a DC power network. The results of this paper collectively provide a new set of analytical tools for DC power networks in general, and for online susceptance control in particular.

This paper opens up several directions for future research. We plan to investigate extensions of the proposed methodologies, possibly under suitable approximations, when key assumptions in this paper are relaxed. This includes generalization to the case of additive disturbances and to networks which are not reducible. Designing distributed control policies for non-parallel networks with provable robustness guarantees is also an important direction of research. Finally, we plan to evaluate the performance of the proposed control policies on AC power flow models, possibly under linear approximations, \eg, as proposed in \cite{Coffrin.VanHentenryck:14}. 

%% file: appendix.tex
\appendix 
\label{sec:appendix}

\subsection{Laplacian Matrix of Reduced Simple Graphs}
\label{subsec:lap}
Consider the following notion of a simple graph corresponding to a given multigraph. 
\begin{definition}[Reduced Simple Graph]
\label{def:reduced-simple-graph}
Given a multigraph $\mc G=(\mc V, \mc E)$, the corresponding reduced simple graph is  denoted as $\mc G^s=(\mc V^s, \mc E^s)$, where $\mc V^s=\mc V$, and $\mc E^s \subseteq \mc E$ is constructed as follows. For every node pair $\{v_1, v_2\} \in \mc V \times \mc V$, for all the links from $v_1$ to $v_2$ in $\mc E$, there exists only one link from $v_1$ to $v_2$ in $\mc E^s$; if there is no link from $v_1$ to $v_2$ in $\mc E$, then there is no link from $v_1$ to $v_2$ in $\mc E^s$. For every $i \in \mc E^s$, let $\mc M_{i}$ be the corresponding links in $\mc E$. The weight matrix for $\mc G^s$, denoted as $W^s \in \real_{>0}^{\mc E^s \times \mc E^s}$, is defined as $w^s_i :=\sum_{j \in \mc M_i} w_j$ for all $i \in \mc E^s$.
\end{definition}
Let $A_{\mc G^s}$ denote the node-link incidence matrix of $\mc G^s$. The next result states that the weighted Laplacians of $\mc G$ and $\mc G^s$ are equal.

 \begin{lemma}
\label{lem:base-graph-laplacian}
Let $L_{\mc G}$ and $L_{\mc G^s}$ be the weighted Laplacian matrices associated with a multigraph $\mc G$ and its reduced simple graph $\mc G^s$ (cf. Definition~\ref{def:reduced-simple-graph}), respectively. Then, $L_{\mc G}=L_{\mc G^s}$.
 \end{lemma}
 \begin{proof}
Definition~\ref{def:laplacian} of the Laplacian implies that
%
\begin{equation*}
  L_{\mc G^s} = A_{\mc G_s} W^s A_{\mc G_s}^T = \sum_{i \in \mc E^s} a^s_{i} \, w^s_{i} \, 
          {a^s}^T_{i}
  = \sum_{i \in \mc E^s} a^s_{i} (\sum_{j \in \mathcal{M}_{i}} w_{j} )
   {a^s}^T_i  = \sum_{j \in \mc E} a_{j} w_{j} \trans{a}_{j} = L_{\mc G}
\end{equation*}
where $a^s_i$ is the \( i \)-th column of \( A_{\mc G^s} \), $a_{j}$ is the $j$-th column of $A$, \(
w^s_{i} \) is the \( i \)-th diagonal element of \( w^s \), and the fourth equality is due to the fact that $a_{j} = a^s_{i}$ for all $j \in \mc M_{i}$, $i \in \mc E^s$. 
 \end{proof}

\subsection{Flow Solution for DC Power Network}
\label{sec:flow-sol-prop}
Lemma~\ref{lemma:flow-solution} gives the link flows for a DC power network for given link weights and supply-demand vector. One can alternately obtain these link flows as solution to a quadratic program, as formalized next. 
%
\begin{lemma}
  \label{lem:same-opt-flow}
Consider a network with graph topology $\mc G=(\mc V, \mc E)$, link weights $w \in \real_{>0}^{\mc E}$ and supply-demand vector $p \in \real^{\mc V}$. The unique solution $f \in \RR^{\mc E}$ satisfying \eqref{model} is the solution to the following:
\begin{equation}
  \label{opt:flow-sol-quad}
  \begin{aligned}
    & \underset{z \in \real^{\mc E}}{\text{min}} & & \trans{z}W^{-1}z \\
    & \text{subject to}
    & & Az = p
  \end{aligned}
\end{equation}
\end{lemma}
\begin{proof}
With \( z_{w} := W^{-1/2} z \) and \( A_{w} := A W^{1/2} \), \eqref{opt:flow-sol-quad} can be rewritten as $\text{min}_{z_w \in \real^{\mc E}} \, \, z_w^T z_w$ subject to $A_w z_w = p$, \ie, finding the minimum $2$-norm solution to $A_w z_w=p$. Since $p$ is balanced, and the network is connected, the minimum $2$-norm solution is given by $f_w^*=A_{w}^{\dagger}p= \trans{A}_{w}(A_{w} \trans{A}_{w})^{\dagger}p=A_w^TL^{\dagger} p $. Reversing the scaling by $W^{1/2}$, this can be rewritten as $W^{-1/2} f^* = W^{1/2} A^T L^{\dagger} p$, \ie, $f^*=WA^T L^{\dagger} p$, which has been shown in Lemma~\ref{lemma:flow-solution} to be the unique $f \in \real^{\mc E}$ satisfying \eqref{model}.
%
\end{proof}

\begin{remark}
Lemma \ref{lem:same-opt-flow} is proved in
\cite{Lai.Low:Allerton13} through a different method, using Lagrange multipliers.
\end{remark}

The following result shows that the flow on every link, under a DC power flow model, is no greater than the total supply/demand. The latter is equal to $\|p\|_1/2$.
 
\begin{lemma}
\label{lem:flow-less-than-inflow}
Consider a network with graph topology $\mc G=(\mc V, \mc E)$, link weights $w \in \real_{>0}^{\mc E}$ and supply-demand vector $p \in \real^{\mc V}$. The unique solution $f \in \RR^{\mc E}$ to \eqref{model} satisfies 
\begin{equation*}
|f_i(w,p)| \leq \|p\|_1/2, \qquad \forall \, i \in \mc E
\end{equation*}
\end{lemma}
\begin{proof}
Let $\tilde{\mc E}$ be the union of links in $\mc E$ with positive flows and reverse of links in $\mc E$ with negative flows. Note that $\tilde{\mc E}$ does not contain links with zero flow, and that the flows on links in $\tilde{\mc E}$ is positive, \ie, $\tilde{f} >0$. Therefore, in order to show the lemma, we need to show that $\tilde{f}_i \leq \|p\|_1/2$ for all $i \in \tilde{\mc E}$. 

It is easy to see that $\tilde{\mc G}:=(\mc V, \tilde{\mc E})$ does not contain cycles. This is because, for every cycle $\mc C \in \tilde{\mc E}$, one can construct a different flow $\tilde{f}' :=\tilde{f}-\onebf_{\mc C} \min_{j \in \mc C}\tilde{f}_j$ for $\tilde{\mc E}$, and hence the corresponding flow $f'$ for the original graph $\mc G$. This construction of $\tilde{f}'$ implies that $|f'| \leq |f|$, with the inequality being strict on at least one component, and hence ${f'}^TW^{-1}f' < f^T W^{-1} f$, and that $f'$ also satisfies flow conservation, \ie, it is a feasible point for \eqref{opt:flow-sol-quad}. Lemma~\ref{lem:same-opt-flow} then leads to a contradiction that $f$ is the solution to \eqref{model}.

Since $\tilde{\mc G}$ does not contain cycles, every path in $\tilde{\mc G}$ containing $i \in \tilde{\mc E}$ is a supply-demand path. Therefore, for all $i \in \tilde{\mc E}$, $\tilde{f}_i$ is no greater than the sum of supply/demand associated with paths containing $i$, which in turn is no greater than the sum of total supply/demand in the network, \ie, $\|p\|_1/2$.
%
%
\end{proof}

\subsection{Minimizing A Quasi-concave Function over A Polytope}
A \emph{polytope} is the convex hull of finitely many points \( \{ b_{1},  \ldots, b_{m}\} \)~\cite[p. 12]{Rockafellar:70}.
\begin{lemma}
\label{lem:min-quasi-concave}
Let \( \map{h}{\real^n}{\real} \) be a quasi-concave function and \( S \subset \real^n\) be a polytope whose elements can be expressed as convex combinations of \( b_{1} \ldots, b_{m} \).  Then, $\min_{x\in S} h(x)=\min_{i \in \until{m}} h(b_i)$.
\end{lemma}
\begin{proof}
We prove by contradiction. Suppose \( \argmin_{x\in S} h(x) \cap \{b_1, \ldots, b_m\} = \emptyset\). Since \( S \)
  is the convex hull of \( \{b_{1} \ldots, b_{m}\} \), for any \( x^{*}
  \in \argmin_{x\in S} h(x) \), there exist \(\eta_{k} \geq 0, k \in \until{m} \), with  \( \sum_{k=1}^{m} \eta_{k} =1 \) such that \( x^{*} = \sum_{k=1}^{m}
  \eta_{k} b_{k} \). Since \( x^{*} \notin \{b_1, \ldots, b_m\}\) by assumption, we have \( \eta_{k} < 1 \) for all \( k \in \until{m} \). Quasi-concavity of $h(x)$ then implies:
%
%
\begin{equation}
\label{eq:quasi-concave-ineq}
h(x^{*}) = h(\sum_{k=1}^{m} \eta_{k} b_{k}) = h(\eta_{1}b_{1} +
(1-\eta_{1}) \sum_{k=2}^{m} \frac{\eta_{k}}{1-\eta_{1}} b_{k} ) \ge
\min\{ h(b_{1}), h( \sum_{k=2}^{m} \frac{\eta_{k}}{1-\eta_{1}} b_{k})\}
\end{equation}
where, in the second equality, it is easy to see that, due to $\sum_{k=2}^m \eta_k = 1 - \eta_1$, we have $\sum_{k=2}^m \frac{\eta_k}{1-\eta_1} b_k \in S$. Recursive application of \eqref{eq:quasi-concave-ineq} then implies \( h(x^{*}) \ge \min_{k \in \until{m}} h(b_{k})\)\footnote{If \(
\eta_{k} = 0 \) for some \( k \in \until{m}\), then we exclude \( h(b_{k}) \) for that $k$ from the
minimization.}, giving a contradiction.  
\end{proof}

\subsection{Derivative of Pseudoinverse of Laplacian Matrix}
\label{sec:laplacian-derivative}
The following is an adaptation of the result from \cite{golub1973differentiation} on the derivative of the pseudo-inverse of a matrix. 
\begin{theorem}
\label{th:derivative_pseudoinverse}
Let $ \mc X \subset \real $ be an open set, and $ P(x) \in \RR^{m\times n}$,  $x \in \mc X$, be a Fr\'echet differentiable matrix function with local constant rank in $ \mc X $. Then for any $x \in \mc X$, 
\begin{equation*}
\frac{d P^\dagger(x)}{d x} = - P^\dagger \frac{d P}{d x} P^\dagger + P^\dagger \trans{{P^\dagger}} \frac{d \trans{P}}{d x}(I - PP^\dagger) + (I - P^\dagger P) \frac{d \trans{P}}{d x}  \trans{ {P^\dagger}} P^\dagger 
\end{equation*}
where $ P^\dagger $ is the pseudo-inverse of $P$.
\end{theorem}
\begin{proof}
The local constant rank condition ensures that $P^\dagger(x)$ and $P(x)$ are continuous and differentiable in the calculations to follow. 
Since $ P^\dagger $ is the pseudo-inverse of matrix $P$, we have $ P  P^\dagger P = P$ and $  P^\dagger P  P^\dagger =  P^\dagger $. Then 
\begin{equation*}
\frac{dP}{d x} = \frac{d(P  P^\dagger P) }{d x} =  \frac{d(P  P^\dagger) }{d x} P + P  P^\dagger \frac{dP }{d x} 
\end{equation*}
Multiplying from the right by $ P^\dagger $ and re-arranging, we get 
\begin{equation*}
\frac{d(P  P^\dagger) }{d x} (PP^\dagger)  = \frac{dP}{d x} P^\dagger- P  P^\dagger \frac{dP }{d x} P^\dagger = (I - P P^\dagger) \frac{dP }{d x} P^\dagger
\end{equation*}
Since $ (P  P^\dagger) (P  P^\dagger)  = P  P^\dagger $ and  $ P  P^\dagger $ is symmetric, 
\begin{align}
\label{eq:dier_a}
  \frac{d(P  P^\dagger) }{d x}  & =   \frac{d(P  P^\dagger)^2 }{d x}  = \frac{d(P  P^\dagger) }{d x} (PP^\dagger)  +  (PP^\dagger) \frac{d(P  P^\dagger) }{d x} \nonumber    \\
  & =  \frac{d(P  P^\dagger) }{d x} (PP^\dagger)  + \trans{\left[ \frac{d(P  P^\dagger) }{d x} (PP^\dagger) \right]} \nonumber \\
  & =  (I - P P^\dagger) \frac{dP }{d x} P^\dagger + \trans{{P^\dagger} }  \frac{d\trans{P} }{d x}(I - P P^\dagger) 
\end{align}
Likewise, we can get 
\begin{equation}
\label{eq:dier_b}
  \frac{d( P^\dagger P) }{d x} = P^\dagger \frac{dP }{d x}   (I -  P^\dagger P) + (I - P^\dagger P)  \frac{d\trans{P} }{d x}\trans{{P^\dagger} }
\end{equation}

Since $ P^\dagger = P^\dagger P  P^\dagger  $, we have following identities. 
\begin{subequations}
\begin{align}
\frac{d P^\dagger}{d x} &= \frac{d (P^\dagger P  P^\dagger)}{d x} =  \frac{d P^\dagger }{d x} P P^\dagger + P^\dagger \frac{d (P  P^\dagger)}{d x} \label{eq:deri1} \\
\frac{d P^\dagger}{d x} &= \frac{d (P^\dagger P  P^\dagger)}{d x} =  \frac{d (P^\dagger P)}{d x} P^\dagger + P^\dagger P \frac{d P^\dagger}{d x} \label{eq:deri2} \\
\frac{d P^\dagger}{d x} & = \frac{d (P^\dagger P  P^\dagger)}{d x} =  \frac{d P^\dagger }{d x} P P^\dagger + P^\dagger \frac{d P}{d x}   P^\dagger+   P^\dagger P \frac{P^\dagger}{d x} \label{eq:deri3}
\end{align}
\end{subequations}
Computing \eqref{eq:deri1} $+$ \eqref{eq:deri2} $-$ \eqref{eq:deri3}, and substituting the resulting expression in \eqref{eq:dier_a} and \eqref{eq:dier_b}, gives the theorem.
\end{proof}

\subsection{Derivatives of \texorpdfstring{$g(\weq)$}{g(x)}}
In this section, we provide explicit expression for the derivatives defined in \eqref{eq:der-def}. 
The derivatives depend on active links, which for the left derivative of $g$, are defined as $\mc K^+(x):=\setdef{i \in \mc E}{\psi_i(w_i^l)<x}$ for $x \in [g^l,\gmax]$, and $\mc K^-(x):=\setdef{i \in \mc E}{\psi_i(w_i^u)<x}$ for $x \in [g^u,\gmax]$. The left derivative, for $\weq\in (\weqlow, \hat{g}^+(\gmax)]$, is then given by:
\begin{equation*} 
  \begin{aligned}
    g'(\weqminus) &= \frac{1}{{\hat{g}^+}'(x^{-})} \Big|_{x=g(\weq)}
    = \left(\sum_{i\in\mathcal{K}^+(g(\weq))} \frac{\partial \mc H(w)/\partial
        w_{i}}{\psi'_{i}(w_i^-)} \right)^{-1}\Big|_{w=\omega^+(g(\weq))}
%
  \end{aligned}
\end{equation*}
where \( {\hat{g}^+}'(x^{-}) \) and \( \psi'_{i}(w_{i}) \) denote left derivatives, similar to \(  g'(\weqminus)  \); the first equality is because $\hat{g}^+$ is inverse of $g$, and the second equality follows from chain rule. Following along the same lines, all the left and right derivatives of $g$ are gathered as:
\begin{equation}
\label{eq:lumped-fun-derivative}
\begin{aligned}
g'(\weqminus) &=  \left\{ 
\begin{array}{cl}
  {\displaystyle   \left(\sum_{i\in\mathcal{K}^+(g(\weq))} \frac{\partial \mc H(w)/\partial
        w_{i}}{\psi'_{i}(w_i^-)}\right)^{-1} \Big|_{w=\omega^+(g(\weq))} } &\; \weq \in \left(\weqlow, \hat{g}^+(\gmax)\right]  \\
  {\displaystyle 0}  &\;  \weq \in \left(\hat{g}^+(\gmax), \hat{g}^-(\gmax)\right]
  \\
 {\displaystyle  \left(\sum_{i\in\tilde{\mathcal{K}}^-(g(\weq))}  \frac{\partial \mc H(w)/\partial
        w_{i}}{\psi'_{i}(w_{i}^{-})}\right)^{-1} \Big|_{w=\omega^-(g(\weq))} } &\;  \weq \in \left(\hat{g}^-(\gmax),\wequp\right]
 \end{array} \right.     \\
g'(\weqplus) &=  \left\{ 
\begin{array}{cl}
  {\displaystyle  \left(\sum_{i\in\tilde{\mathcal{K}}^+(g(\weq))} \frac{\partial \mc H(w)/\partial
        w_{i}}{\psi'_{i}(w_i^+)}\right)^{-1} \Big|_{w=\omega^+(g(\weq))} }  &\; \weq \in \left[\weqlow, \hat{g}^+(\gmax)\right)   \\
          {\displaystyle 0 } &\; \weq \in \left[\hat{g}^+(\gmax), \hat{g}^-(\gmax)\right) 
          \\
   {\displaystyle  \left(\sum_{i\in\mathcal{K}^-(g(\weq))}  \frac{\partial \mc H(w)/\partial
        w_{i}}{\psi'_{i}(w_{i}^{+})}\right)^{-1} \Big|_{w=\omega^-(g(\weq))} } &\; \weq \in \left[\hat{g}^-(\gmax),\wequp\right)
\end{array} \right.   
\end{aligned}
\end{equation}
where $\tilde{\mc K}^+(x):=\setdef{i \in \mc E}{\psi_i(w_i^l) \leq x}$ for $x \in [g^l,\gmax]$, and $\tilde{\mc K}^-(x):=\setdef{i \in \mc E}{\psi_i(w_i^u) \leq x}$ for $x \in [g^u,\gmax]$. 
